\documentclass[english,11pt,twoside,a4paper]{article}

\usepackage{amsmath,amsthm}
\usepackage{amssymb}
\usepackage{amsfonts}
\usepackage{amscd}
\usepackage[all]{xy}
\font \smallrm=cmr10 at 9pt
\font \smallsl=cmsl10 at 9pt

\newcounter{amoi}
\newtheorem{thm}{Theorem}[subsection]
\newtheorem{theorem}[thm]{Theorem}

\newtheorem{corollary}[thm]{Corollary}
\newtheorem{lemma}[thm]{Lemma}
\newtheorem{proposition}[thm]{Proposition}

\theoremstyle{definition}
\newtheorem{definition}[thm]{Definition}
\newtheorem{example}[thm]{Example}
\newtheorem{examples}[thm]{Examples}
\newtheorem{remark}[thm]{Remark}
\newtheorem{remarks}[thm]{Remarks}

\newtheorem{free text}[thm]{}

\newcommand{\N} {\mathbb{N}}
\newcommand{\Z} {\mathbb{Z}}

\newcommand{\bk} {\Bbbk}

\newcommand{\bG} {\mathbf{G}}

\newcommand{\cO}{\mathcal{O}}

\newcommand{\cL} {\mathcal{L}}
\newcommand{\cJ} {\mathcal{J}}
\newcommand{\cP} {\mathcal{P}}
\newcommand{\cR} {\mathcal{R}}

\newcommand{\fa} {\mathfrak{a}}
\newcommand{\fg} {\mathfrak{g}}
\newcommand{\fh} {\mathfrak{h}}

\newcommand{\fk} {\mathfrak{k}}

\newcommand{\zero} {{\mathbf{0}}}
\newcommand{\uno} {{\mathbf{1}}}

\newcommand{\Uuno} {\mathbf{1\hskip-3,4pt{}l}}

\newcommand{\alg} {\text{\rm (alg)}}
\newcommand{\aalg} {\text{\rm (a-alg)}}
\newcommand{\salg} {\text{\rm (salg)}}
\newcommand{\asalg} {\text{\rm (a-salg)}}
\newcommand{\Hsalg} {\text{\rm (H-salg)}}
\newcommand{\wkspsalg} {\text{\rm (wksp-salg)}}
\newcommand{\splsalg} {\text{\rm (spl-salg)}}
\newcommand{\stspsalg} {\text{\rm (stsp-salg)}}
\newcommand{\cexsalg} {\text{\rm (cex-salg)}}
\newcommand{\sets}{\text{\rm (sets)}}
\newcommand{\grps} {\text{\rm (groups)}}
\newcommand{\assch}{\text{\rm (assch)}}
\newcommand{\sgrps} {\text{\rm (sgroups)}}

\newcommand{\fsgrps} {\text{\rm (fsgroups)}}
\newcommand{\gsssgrps} {\text{\rm (gss-sgroups)}}
\newcommand{\gssfsgrps} {\text{\rm (gss-fsgroups)}}

 \newcommand{\lgssfsgrps} {\text{\rm (lgss-fsgroups)}}
\newcommand{\lie} {\text{\rm (Lie)}}
\newcommand{\slie} {\text{\rm (sLie)}}
\newcommand{\sHCp} {\text{\rm (sHCp)}}
\newcommand{\lsHCp} {\text{\rm (lsHCp)}}

\newcommand{\Ad}{\hbox{\sl Ad}}
\newcommand{\ad}{\hbox{\sl ad}}
\newcommand{\Hom}{\hbox{\sl Hom}}
\newcommand{\End}{\hbox{\sl End}}
\newcommand{\Der}{\hbox{\sl Der}}
\newcommand{\Lie}{\hbox{\sl Lie}}
\newcommand{\Ker}{\hbox{\sl Ker}}

\newcommand{\uspec}{\underline{\hbox{\sl Spec}\,}}

\newcommand{\rGL}{\mathrm{GL}}

\newcommand{\rgl}{\mathfrak{gl}}

\newcommand{\eps}{\underline{\epsilon}}

\begin{document}

{\ }

\vskip-41pt

   \centerline{\smallrm  {\smallsl Transactions of the American Mathematical Society\/}  (to appear)}


\vskip39pt   {\ }

\centerline{\Large \bf GLOBAL SPLITTINGS AND}
 \vskip9pt
\centerline{\Large \bf SUPER HARISH-CHANDRA PAIRS}
 \vskip9pt
\centerline{\Large \bf FOR AFFINE SUPERGROUPS}

\vskip26pt

\centerline{ Fabio GAVARINI }

\vskip9pt

\centerline{\it Dipartimento di Matematica, Universit\`a di Roma ``Tor Vergata'' } \centerline{\it via della ricerca scientifica 1  --- I-00133 Roma, Italy}

\centerline{{\footnotesize e-mail: gavarini@mat.uniroma2.it}}

\vskip61pt

\begin{abstract}
 \vskip7pt
 {\smallrm
   This paper dwells upon two aspects of affine supergroup theory, investigating the links among them.
                                     \par
   First, I discuss the ``splitting'' properties of affine supergroups, i.e.~special kinds of factorizations they may admit   --- either globally, or point-wise.  Almost everything should be more or less known, but seems to be not as clear in literature (to the author's knowledge) as it ought to.
                                     \par
   Second, I present a new contribution to the study of affine supergroups by means of super Harish-Chandra pairs (a method already introduced by Koszul, and later extended by other authors).  Namely, I provide a new functorial construction  $ \Psi $  which, with each super Harish-Chandra pair, associates an affine supergroup that is always  {\smallsl globally strongly split\/}  (in short,  {\smallsl gs-split\/})   --- thus setting a link with the first part of the paper.  One knows that there exists a natural functor  $ \Phi $  from affine supergroups to super Harish-Chandra pairs: then I show that the new functor  $ \Psi $   --- which goes the other way round ---   is indeed a quasi-inverse to  $ \Phi \, $,  provided we restrict our attention to the subcategory of affine supergroups that are  gs-split.  Therefore, (the restrictions of)  $ \Phi $  and $ \Psi $  are equivalences between the categories of gs-split affine supergroups and of super Harish-Chandra pairs.  Such a result was known in other contexts, such as the smooth differential or the complex analytic one, via different approaches  (see  \cite{koszul},  \cite{ms2}, \cite{cf}):  nevertheless, the novelty in the present paper lies in that I construct a  {\smallsl different\/}  functor  $ \Psi $  and thus extend the result to a much larger setup, with a totally different, more geometrical method.  In fact, this method (very concrete, indeed) is universal and characteristic-free: I present it here for the algebro-geometric setting, but actually it can be easily adapted to the frameworks of differential or complex-analytic supergeometry.
                                     \par
   The case of  {\smallsl linear\/}  supergroups is treated also as an intermediate, inspiring step.
                                     \par
   Some examples, applications and further generalizations are presented at the end of the paper.}
 \hskip-5pt
   \footnote{\ 2010 {\it MSC}\;: \, Primary 14M30, 14A22; Secondary 17B20.}
%
\end{abstract}

\vskip31pt

\tableofcontents

\vskip59pt

\section{Introduction}

\smallskip

   {\ } \;\;  The study of ``supergroups'' is a chapter of ``supergeometry'', i.e.\ geometry in a  $ \Z_2 $--graded  sense.  In particular, the relevant structure sheaves of (commutative) algebras sitting on top of the topological spaces one works with are replaced with sheaves of (commutative)  {\sl superalgebras}.

\smallskip

   Every superalgebra  $ A $  is built from (homogeneous) even and odd elements.  It is then natural   --- especially in the commutative case, when these elements can be thought of as ``functions'' on some superspace ---   to look for some ``separation of variables'' result for  $ A \, $,  in the form of a ``splitting'', i.e.\ a factorization of type  $ \, A = \overline{A} \otimes A' \, $  where  $ \, \overline{A} \, $  is a totally even subalgebra and  $ A' $  is a second algebra which encodes the ``odd part'' of  $ A \, $.  Actually, in the commutative case the best one can hope for is that  $ A' $  be an algebra freely generated by some subsets of odd elements in  $ A \, $,  hence  $ A' $  is a Grassmann (super)algebra, i.e.\ the ``polynomial (super)algebra'' on some set of odd variables.

\smallskip

   When coming to supergeometry, we deal with ``superspaces'' such as smooth or analytic supermanifolds (in the differential and complex holomorphic setup) or superschemes (in the algebro-geometric framework).  Any such superspace can be considered as a  {\sl classical\/}  (i.e.\ non-super) space   -- in the appropriate category ---   endowed with a suitable sheaf of commutative superalgebras.
                                                    \par
   A natural question then arises: can one parallelize this sheaf?  In other words, is it globally trivial, in some ``natural'' sense?  For superspaces (in any sense: differential, analytic, etc.) the answer in general is in the negative: indeed, counterexamples do exist.  Instead, if we restrict to  {\sl supergroups\/}  then the answer in most cases is  {\sl positive}.  Indeed, this is the case for real Lie supergroups (see  \cite{be},  \cite{ma},  \cite{ccf})  and for complex analytic supergroups (see  \cite{vis}  and  \cite{cf});  in the algebro-geometric setting, the best result I am aware of is by Masuoka (see \cite{ms1}),  who proved that for all affine supergroups over fields of characteristic different from 2 the answer still is positive.
                                                    \par
   It might be worth minding the analogy with the situation of the tangent bundle on a classical space: for a generic space (manifold, complex analytic variety or scheme) in general it is  {\sl not\/}  parallelizable; for  {\sl groups\/}  instead (real Lie groups, complex analytic Lie groups and group-schemes) it is known to be parallelizable.  This might lead us to expect, from scratch, that a similar result occur with supergroups and their structure sheaf   --- although this is nothing but a sheer analogy.
                                                    \par
   Note that in the  {\sl affine\/}  case having a parallelization of the structure sheaf on a superspace  $ X $  amounts to having a ``splitting'' of its superalgebra of global sections  $ \cO(X) \, $:  this sets a link with the previously mentioned theme of splitting (commutative) superalgebras, and also leads us to saying that  $ X $  has a ``global splitting'', or it is globally split, whenever its structure sheaf is parallelizable.
                                                    \par
   On the other hand, one can study any supergroup  $ G \, $,  like any superspace, via its functor of points: then, for each commutative superalgebra  $ A $  one has the group  $ G(A) $  of  $ A $--points  of  $ G \, $.  Such a group may have remarkable ``splittings'' (in group-theoretical sense) on its own; this kind of ``pointwise splitting'' is often considered in literature (e.g.\ in Boseck's papers  \cite{bos1},  \cite{bos2},  \cite{bos3}),  but must not be confused with the notion of ``global splitting''.

\smallskip

   Roughly speaking, a parallelized ``supersheaf''  $ \mathcal{S} $  over a superspace  $ X $  is ``encoded'' by a pair  $ \, \big(\mathcal{S}_\zero \, , \text{\sl S}_{x_{{}_0}} \big) \, $  where  $ \mathcal{S}_\zero $  is the ``even part'' of  $ \mathcal{S} \, $  and  $ \text{\sl S}_{x_{{}_0}} \! $  is the fiber of  $ \mathcal{S} $  over some point  $ x_0 \, $;  as  $ \mathcal{S}_\zero $  is encoded in the classical (i.e.\ non-super) space  $ X_\zero $  underlying  $ X \, $,  one can also use the pair  $ \, \big(X_\zero \, , \text{\sl S}_{x_{{}_0}} \big) \, $  instead.  When  $ \, X = G \, $  is a supergroup, we can take  $ x_0 $  to be the identity element in the (classical) group  $ \, G_\zero \, $  and approximate  $ \text{\sl S}_{x_{{}_0}} \! $  with the cotangent space at  $ G_\zero $  in that point; we can also replace this cotangent space with its dual, i.e.\ the
 \hbox{\sl tangent Lie superalgebra  $ \, \fg := \Lie\,(G) \, $  of  $ G \, $}.
                                                    \par
   This leads us to another   --- tightly related ---   way of formulating the problem, namely inquiring whether it is possible (via a ``parallelization'' of the structure sheaf, etc.) to describe a supergroup  $ G $  in terms of the pair  $ \, \big(G_\zero\,,\,\fg\,\big) \, $   which is naturally associated with it.  Indeed, this is the core of the problem of studying supergroups via ``super Harish-Chandra pairs'', as I now explain.

\smallskip

   The notion of ``super Harish-Chandra pair'' (a terminology first found in  \cite{dm}),  or just sHCp in the sequel, was first introduced in the real differential setup, but naturally adapts to the complex analytic or the algebro-geometric context (see, e.g.,  \cite{vis}  and  \cite{cf}).  Whatever the setup, a sHCp is a pair  $ (G_+,\fg) $  made of a classical group (real Lie, complex analytic, etc.) and a Lie superalgebra obeying natural compatibility constraints.  Indeed, the definition itself is tailored in a such a way that there exists a natural functor  $ \Phi $  from the category of supergroups to the category of sHCp's which associates with each supergroup  $ G $  its sHCp  $ \big(G_{\!\text{\it ev}} \, , \Lie\,(G)\big) \, $  made of the ``classical subgroup'' and the ``tangent Lie superalgebra'' of  $ G \, $.  The question is: can one recover a supergroup out of its associated sHCp\,?  In other words, does there exist any functor  $ \Psi $  from sHCp's to supergroups which be a quasi-inverse for  $ \Phi \, $?  And if the answer is positive, how much explicit such a functor is\,?
                                                     \par
   In the real differential framework   --- i.e.\ for real Lie supergroups and real smooth sHCp's ---   Kostant proved  (see  \cite{kostant},  and also  \cite{koszul})  that  $ \Phi $  is an equivalence i.e.~one has a quasi-inverse for it.
                                                     \par
   Besides, Vishnyakova  (see  \cite{vis})  fixed both the real smooth and the complex analytic cases.
                                                     \par
   As to the  {\sl algebraic\/}  setup, more recently Carmeli and Fioresi (see  \cite{cf})  proved the same result for  {\sl algebraic affine supergroup schemes\/}  (and the corresponding category of sHCp's) over a ground ring  $ \bk $  that is an algebraically closed field of characteristic zero.  Indeed, their method   --- which extends Vishnyakova's idea, so applies to the real smooth and complex analytic setup too ---   provides an explicit construction of a quasi-inverse functor  $ \Psi $  for  $ \Phi \, $.  This was improved by Masuoka (in  \cite{ms2}),  who only required that  $ \bk $  be a field whose characteristic is not 2, and applied his result to a characteristic-free study of affine supergroup schemes.  Later on (see \cite{mas-shi}),  Masuoka and Shibata further extended Koszul's method up to work on every commutative ring, via an algebraic version of the notion of sHCp   --- devised to treat the matter with Hopf (super)algebra techniques.
                                                     \par
   In the second part of this paper I present a new solution to these problems, providing explicitly a new functor  $ \Psi $  (different from those by other authors), which does the job; in particular, I also show that any positive answer is possible if and only if we restrict our attention to those (affine) supergroups which are globally strongly split   --- thus setting a link with the first part of the paper.

\smallskip

   The above mentioned construction of the functor  $ \Psi $  is made in the setup (and with the language) of algebraic supergeometry.  Nevertheless, it is worth stressing that one can easily reformulate everything in the setup (and with the language) of real differential supergeometry or complex analytic supergeometry: in other words, the method presented here also applies,  {\it mutatis mutandis},  to real or complex Lie supergroups (which, as we mentioned, are known to be all globally split).

\smallskip

   The paper is organized as follows.  First (Sec.~2) we establish the language and notations we need.  Then (Sec.~3) we treat the notions of ``splittings'' for superalgebras, Hopf superalgebras, superschemes and supergroups; in particular, we present some results about global splittings of supergroups and about their ``local'' splittings, i.e.\ splittings on  $ A $--points.  Finally (Sec.~4), we study the relation between supergroups and super Harish-Candra pairs, and the construction of a functor  $ \Psi $  which is quasi-inverse to the natural one  $ \Phi $  associating a sHCp with any supergroup.

\vskip35pt

   \centerline{\bf Acknowledgements}
 \vskip7pt
   \centerline{ The author thanks A.~D'Andrea, M.~Duflo and R.~Fioresi for their priceless suggestions, }
   \centerline{ and most of all, in particular, A.~Masuoka for his many valuable comments and remarks. }

\bigskip

\section{Preliminaries}  \label{preliminaries}

\smallskip

   {\ } \;\;   In this section I introduce some preliminaries of (affine) supergeometry.  Classical references for that are  \cite{dm},  \cite{ma}  and  \cite{vsv},  but I shall mainly rely on  \cite{ccf}.
                                                      \par
   All over the paper,  $ \bk $  will be a commutative, unital ring.

\medskip

 \subsection{Superalgebras, superspaces, supergroups}  \label{first_preliminaries} {\ }

\smallskip

   This subsection is devoted to fix terminology and notation for some basic notions.

\smallskip

\begin{free text}  \label{superalgebras}
 {\bf Supermodules and superalgebras.}  A  {\it  $ \bk $--supermodule\/}  is by definition a  $ \bk $--module  $ V $  endowed with a  $ \Z_2 $--grading,  say  $ \, V = V_\zero \oplus V_\uno \, $,  where  $ \, \Z_2 = \{\zero,\uno\} \, $  is the group with two elements.  The  $ \bk $--submodule  $ V_\zero $  and its elements are called  {\it even},  while  $ V_\uno $  and its elements  {\it odd}.  By  $ \, |x| $  or  $ p(x) $  $(\in \Z_2) \, $  we denote the  {\sl parity\/}  of any homogeneous element, defined by the condition  $ \, x \in V_{|x|} \, $.

\vskip4pt

   We call  {\it  $ \bk $--superalgebra\/}  any associative, unital  $ \bk $--algebra  $ A $  which is  $ \Z_2 $--graded  (as a  $ \bk $--algebra):  so  $ A $  has a  $ \Z_2 $--splitting  $ \, A = A_\zero \oplus A_\uno \, $,  and  $ \, A_{\mathbf{a}} \, A_{\mathbf{b}} \subseteq A_{\mathbf{a}+\mathbf{b}} \; $.  All  $ \bk $--superalgebras  form a category, whose morphisms are all those in the category of  $ \bk $--algebras  that preserve the unit and the  $ \Z_2 $--grading.  A  {\it Hopf  superalgebra\/}  over  $ \bk $  is a Hopf algebra  $ \, H = H_\zero \oplus H_\uno \, $  in the category of  $ \bk $--super-algebras, where the multiplication in a tensor product  $ \, H' \mathop{\otimes}\limits_\bk H'' \, $  is given by  $ \; \big( h'_1 \otimes h''_1 \big) \cdot \big( h'_2 \otimes h''_2 \big) := (-1)^{|h''_1|\,|h'_2|} \big( h'_1 \cdot h'_2 \big) \otimes \big( h''_1 \cdot h''_2 \big) \; $.  Morphisms among Hopf superalgebras are then the obvious ones.
                                                                \par
   In the following, if  $ H $  is any Hopf superalgebra with counit  $ \epsilon $  we shall write  $ \, H^+ := \Ker\,(\epsilon) \, $.

\smallskip

   A superalgebra  $ A $  is said to be  {\it commutative\/}  iff  $ \; x \, y = (-1)^{|x|\,|y|} y \, x \; $  for all homogeneous  $ \, x $,  $ y \in A \, $  and  $ \; z^2 = 0 \; $  for all odd  $ \, z \in A_\uno \, $.  We denote  by $ \salg $  the category of commutative  $ \bk $--superalgebras;  if necessary, we shall stress the role of  $ \bk $  by writing  $ \salg_\bk \, $.  A Hopf superalgebra is said to be commutative if it is such as a superalgebra, and we denote  by  $ \Hsalg_\bk \, $,  or simply  $ \Hsalg \, $,  the category of commutative Hopf  $ \bk $--superalgebras.  We shall also denote by  $ \alg_\bk $   --- or simply  $ \alg $  ---   the category of (associative) commutative unital  $ \bk $--algebras.
                                                                \par
   For  $ \, A \in \salg_\bk \, $,  $ \, n \in \N \, $,  we call  $ A_\uno^{\,[n]} $  the  $ A_\zero \, $--submodule  of  $ A $  spanned by all products  $ \, \vartheta_1 \cdots \vartheta_n \, $  with  $ \, \vartheta_i \in A_\uno \, $  for all  $ i \, $,  and then  $ A_\uno^{(n)} $  and  $ A_\uno^{\,n} $  respectively the unital  $ \bk $--subalgebra  and the ideal of  $ A $  generated by  $ A_\uno^{\,[n]} \, $.  Similarly we consider  $ \, H_\uno^{\,[n]} $,  $ \, H_\uno^{(n)} $,  $ \, H_\uno^{\,n} \, $  for  $ \, H \in \Hsalg_\bk \, $.

\vskip4pt

   We need also to consider the following constructions.  Given  $ \, A = A_\zero \oplus A_\uno \in \salg_\bk \, $,  let  $ \, J_A := (A_\uno) \, $  be the ideal of  $ A $  generated by  $ A_\uno \, $:  then  $ \, J_A = A_\uno^{[2]} \oplus A_\uno \, $,  and  $ \, \overline{A} := A \big/ J_A \, $  is a commutative superalgebra which is  {\sl totally even},  i.e.~$ \, \overline{A} \in \alg_\bk \, $;  moreover, there is an obvious isomorphism  $ \, \overline{A} := A \big/ (A_\uno) \cong A_\zero \big/ A_\uno^{[2]} \, $.  Also, the construction of  $ \overline{A} $  is functorial in  $ A \, $; and similarly for the constructions of  $ A_\zero $  and of  $ A_\uno^{(n)} \, $.  This yields functors
 $ \, \overline{(\ )} \, , \, {(\ )}_\zero : \salg_\bk \! \longrightarrow \alg_\bk \, $  and  $ \, {(\ )}_\uno^{(n)} \! : \salg_\bk \! \longrightarrow \salg_\bk \, $
 respectively defined on objects by  $ \, A \! \mapsto \overline{A} \, $,  $ \, A \! \mapsto A_\zero \, $,  $ \, A \! \mapsto A_\uno^{(n)} \, $  ($ \, n \in \N \, $)\,.
                                                                \par
   On the other hand, there is an obvious functor  $ \, \cJ_{\alg_\bk}^{\salg_\bk} : \alg_\bk \longrightarrow \salg_\bk \, $  given by taking any commutative  $ \bk $--algebra  as a totally even superalgebra; both  $ \overline{(\ )} $  and  $ {(\ )}_\zero $  are retractions of  $ \cJ_{\alg_\bk}^{\salg_\bk} \, $.
\end{free text}

\smallskip

   We shall now introduce the  {\it affine superschemes},  which by definition are representable functors from  $ \salg $  to the category  $ \sets $  of all sets:

\vskip11pt

\begin{definition}  \label{aff-spec}
  For any  $ \, R \in \salg_\bk \, $,  we call  {\sl spectrum of}  $ R \, $,  denoted  $ \uspec(R) $  or also  $ h_R \, $,  the representable functor  $ \; \uspec(R) = h_R : \salg_\bk \longrightarrow \sets \; $  associated with  $ R \, $.  Explicitly,  $ h_R $  is given on objects by  $ \; h_R(A) := \Hom_{\salg_\bk} \big( R \, , A \big) \; $  and on arrows by  $ \; h_R(f)(\phi) := f \circ \phi \; $.  All such spectra are also called  {\it affine  $ \bk $--superschemes}.  Any affine superscheme is said to be  {\it algebraic\/}  if its representing (commutative) superalgebra is finitely generated.
                                                          \par
   When  $ h_R $  is actually a functor from  $ \salg_\bk $  to  $ \grps $,  the category of groups, we say that  $ h_R $  is a  {\it (affine) group  $ \bk $--superscheme},  in short a  {\it (affine)  $ \bk $--supergroup\/};  indeed, this is equivalent to the fact that  $ \, R \, $  be a (commutative)  {\sl Hopf superalgebra},  i.e.~$\, R \in \Hsalg_\bk \, $.  In other words, the (affine) group superschemes are nothing but the functors from  $ \salg_\bk $  to  $ \grps $  which are representable.  Any affine  $ \bk $--supergroup is algebraic if it is such as an affine $ \bk $--superscheme,  i.e.~its representing (Hopf)  $ \bk $--superalgebra  is of finite type.
                                                          \par
   All affine  $ \bk $--superschemes  form a category, with suitably defined morphisms, denoted by  $ \assch_\bk $  which is isomorphic to the category  $ \salg_\bk^\circ \, $  opposite to  $ \salg_\bk \, $: an isomorphism  $ \; \salg_\bk^\circ \,{\buildrel \cong \over \longrightarrow}\, \assch_\bk $  \; is given on objects by  $ \, R \mapsto h_R \, $,  and we denote its inverse  $ \; \assch_\bk \,{\buildrel \cong \over \longrightarrow}\, \salg_\bk^\circ \; $  by  $ \, X \mapsto \cO(X) \; $.  Similarly, all affine  $ \bk $--supergroups  form a category, denoted by  $ \sgrps_\bk \, $,  isomorphic to the category  $ \Hsalg_\bk^\circ \, $  opposite to  $ \Hsalg_\bk \, $:  explicit isomorphisms are given (with same notation) by restrictions of the previous ones between  $ \salg_\bk^\circ \, $  and  $ \assch_\bk \, $  respectively.
                                                          \par
   More in general, we call respectively  {\it superset\/  $ \bk $--functor\/}  and  {\it supergroup\/  $ \bk $--functor\/}  (possibly dropping the  ``\,$ \bk $--\,'')  any functor  $ \, X \! : \salg_\bk \! \longrightarrow \sets \, $
 \hbox{and any functor  $ \, G : \salg_\bk \! \longrightarrow \grps \, $.
    $ \diamondsuit $}
%
\end{definition}

\vskip5pt

\begin{example}
   The  {\sl affine superspace\/}  $ \, \mathbb{A}_\bk^{p|q} \, $,  denoted  $ \, \bk^{p|q} \, $  too, is defined (for  $ \, p \, $,  $ q \in \N \, $)  as  $ \; \mathbb{A}_\bk^{p|q} := \uspec \big( \bk[x_1,\dots,x_p] \mathop{\otimes}\limits_\bk \bk[\xi_1 \dots \xi_q] \big) \; $  where  $ \bk[\xi_1 \dots \xi_q] $  is the exterior (or ``Grassmann'') algebra generated by  {\sl odd\/} variables  $ \xi_1 $,  $ \dots $,  $ \xi_q \, $,  and  $ \, \bk[x_1,\dots,x_p] \, $  the polynomial algebra in  $ p $  commuting variables.  The superdimension of  $ \mathbb{A}_\bk^{p|q} $  is easily seen to be  $ \, p|q \, $.
 \hskip13pt \hfill  $ \blacklozenge $
\end{example}

\vskip6pt

\begin{remark}
   More in general, one can consider the broader notions of (not necessarily affine) superscheme and supergroup, still defined over  $ \salg_\bk \, $   --- see  \cite{ccf}  for more details.  In the present work, however, we do not need to consider such more general notions.
\end{remark}
%
%
%
%
%

\vskip4pt

   The next examples turn out to be very important in the sequel.

\vskip7pt

\begin{examples}  \label{exs-supvecs}  {\ }
 \vskip2pt
   {\it (a)} \,  Let  $ V $  be a free  $ \bk $--supermodule,  that is a  $ \bk $--supermodule  for which both  $ V_\zero $  and  $ V_\uno $  are free as  $ \bk $--modules.  For any superalgebra  $ A $  we define  $ \; V(A) \, := \, {(A \otimes V)}_\zero \, = \, A_\zero \otimes V_\zero \oplus A_\uno \otimes V_\uno \; $.  This is a representable functor in the category of superalgebras, whose representing object is the  $ \bk $--superalgebra  of polynomial functions on  $ V \, $.  Hence  $ V $  can be seen as an affine  $ \bk $--superscheme.
 \vskip2pt
   {\it (b)} \,  {\sl  $ \rGL(V) $  as an affine algebraic supergroup}.  Let  $ V $  be a  $ \bk $--supermodule  which is free and whose rank, i.e.~the pair  $ \, \text{\sl rk}(V) := \big( \text{\sl rk}(V_\zero), \text{\sl rk}(V_\uno) \big) \, $,  is finite,  i.e.~$ \, \text{\sl rk}(V_\zero), \text{\sl rk}(V_\uno) \in \N \, $.  For any  $ \bk $--superalgebra  $ A \, $,  let  $ \, \rGL(V)(A) := \rGL\big(V(A)\big) \, $  be the set of isomorphisms  $ \; V(A) \longrightarrow V(A) \; $.  If we fix a homogeneous basis for  $ V $  and we set  $ \, p := \text{\sl rk}(V_\zero) \, $,  $ \, q := \text{\sl rk}(V_\uno) \in \N \, $,  we have  $ \, V \cong \bk^{p|q} \; $:  then we also denote $ \, \rGL(V) \, $  with  $ \, \rGL_{p|q} \, $.  Now,  $ \rGL_{p|q}(A) $  is the group of invertible  $ (p\,,q) $--block  matrices   --- whose size is  $ (p+q) $  ---   with diagonal block entries in  $ A_\zero $  and off-diagonal block entries in $ A_\uno \, $.  It is known that the functor  $ \rGL(V) $  is representable, so  $ \rGL(V) $  is indeed an affine  $ \bk $--supergroup,  and it is also algebraic; see (e.g.),  \cite{vsv}, Ch.~3,  for further details.   \hfill  $ \blacklozenge $
\end{examples}

\medskip

\begin{definition}  \label{nat-funcs}  {\ }
 For any superset  $ \bk $--functor  $ \, X : \salg_\bk \longrightarrow \sets \, $,  we respectively set
  $$  \overline{X} := X \circ \cJ_{\alg_\bk}^{\salg_\bk} \circ \overline{(\ )} \;\; ,  \quad  X_\zero := X \circ \cJ_{\alg_\bk}^{\salg_\bk} \circ {(\ )}_\zero \;\; ,  \quad  X_\uno^{(n)} := X \circ {(\ )}_\uno^{(n)}  $$
to denote its composition with the functor
 $ \, \overline{(\ )} : \salg_\bk \longrightarrow \alg_\bk \, $,  $ \, {(\ )}_\zero : \salg_\bk
\longrightarrow \alg_\bk \, $  and  $ \, {(\ )}_\uno^{(n)} : \salg_\bk \longrightarrow \alg_\bk \; $
 ($ n \in \N $),  followed, in the first two cases, by
 $ \, \cJ_{\alg_\bk}^{\salg_\bk} : \alg_\bk \longrightarrow \salg_\bk \, $   --- see  \S \ref{superalgebras}.  Similar notation applies when  $ \, X = \, G \, $  is in fact a supergroup  $ \bk $--functor.
 \hfill   $ \diamondsuit $
\end{definition}

\medskip

  \subsection{Lie superalgebras}  \label{Lie-superalgebras} {\ }

\smallskip

   The notion of Lie superalgebra over a field is well known: in particular, it is entirely satisfactory when the characteristic of the ground field  $ \bk $  is neither 2 nor 3.  However, it is not as well satisfactory   --- in the standard formulation ---   when that characteristic is either 2 or 3.  This motivates one to introduce the following modified formulation, whose main feature is to describe a ``correct'' notion of Lie superalgebras as given by the standard notion enriched with an additional piece of structure, namely sort of a  ``$ 2 $--mapping''  that is a close analogue to the  $ p $--mapping  in a  $ p $--restricted  Lie algebra over a field of characteristic  $ p > 0 \, $.

\medskip

\begin{definition}  \label{def-Lie-salg}
 Let  $ \, \fg = \fg_\zero \oplus \fg_\uno \, $  be a  $ \bk $--supermodule.  We say that  $ \fg $  is a  {\sl Lie superalgebra\/}  if we have a  {\it (Lie super)bracket\/}  $ \; [\,\cdot\, , \cdot\, ] : \fg \times \fg \longrightarrow \fg \, $,  $ \; (x,y) \mapsto [x,y] \, $,  \, and a  {\it  $ 2 $--operation\/}  $ \; {(\,\cdot\,)}^{\langle 2 \rangle} : \fg_\uno \longrightarrow \fg_\zero \, $,  $ \; z \mapsto z^{\langle 2 \rangle} \, $,  which satisfy the following properties (for all  $ \, x , y \in \fg_\zero \cup \fg_\uno \, $,  $ \, w \in \fg_\zero \, $,  $ \, z, z_1, z_2 \in \fg_\uno $):
 \vskip5pt
 \hskip-15pt
   {\it (a)}  \qquad  $ [\,\cdot\, , \cdot\, ] \, $  is  $ \bk $--bilinear,
\qquad  $ [w,w] \; = \; 0  \quad ,  \qquad  \big[z,[z,z]\big] \; = \; 0  \quad $;
 \vskip5pt
 \hskip-15pt
   {\it (b)}  $ \qquad  [x,y] \, + \, {(-1)}^{|x| \, |y|}[y,x] \; = \; 0  \qquad\; $  {\sl (anti-symmetry)}\,;
 \vskip7pt
 \hskip-15pt
   {\it (c)}
 $ \quad  {(-\!1)}^{|x| \, |z|} [x,  [y,z]] + {(-\!1)}^{|y| \, |x|} [y , [z,x]] \, + \, {(-\!1)}^{|z| \, |y|} [z , [x,y]] \, = \, 0 \;\, $
 \vskip3pt
   \centerline{\sl (Jacobi identity);}
 \vskip6pt
 \hskip-15pt
   {\it (d)}  \qquad  $ {(\,\cdot\,)}^{\langle 2 \rangle} \, $  is  $ \bk $--quadratic, i.e.~$ \,\;  {(c\,z)}^{\langle 2 \rangle} = \, c^2 \, z^{\langle 2 \rangle} \;\; $  for all  $ \, c \in \bk \; $;
 \vskip6pt
 \hskip-15pt
   {\it (e)}  \qquad  $ {(z_1 \! + z_2)}^{\langle 2 \rangle}  \, = \;  z_1^{\langle 2 \rangle} + \, [z_1,z_2] \, + \, z_2^{\langle 2 \rangle} \quad $;
 \vskip6pt
 \hskip-15pt
   {\it (f)}  \qquad  $ \big[ z^{\langle 2 \rangle}, x \big]  \, = \;  \big[ z \, , [z,x] \big] \quad $.
 \vskip9pt
   All Lie  $ \bk $--superalgebras  form a category, denoted  $ \slie_\bk \, $,  whose morphisms are the  $ \bk $--linear,  graded maps preserving the bracket and the  $ 2 $--operation.
 \hfill   $ \diamondsuit $
\end{definition}

\smallskip

\begin{remark}
 The conditions in  Definition \ref{def-Lie-salg}  are somewhat redundant, and in some cases may be simplified: for instance, condition  {\it (e)\/}  yields  $ \; [z_1,z_2] \, = \, {(z_1 \! + z_2)}^{\langle 2 \rangle} - z_1^{\langle 2 \rangle} - z_2^{\langle 2 \rangle} \; $  so one could use this as a  {\sl definition\/}  of the Lie bracket on  $ \, \fg_\uno \times \fg_\uno \, $  in terms of the 2-operation.  Conversely, when  $ 2 $  is invertible in  $ \bk $  the  $ 2 $--operation  is recovered from the Lie bracket, via condition  {\it (e)},  as  $ \; z^{\langle 2 \rangle} = 2^{-1} \, [z,z] \; $.
\end{remark}

\smallskip

\begin{example}  \label{def-End(V)}
  Let  $ \, V = V_\zero \oplus V_\uno \, $  be a  {\sl free\/}  $ \bk $--supermodule,  and consider  $ \End(V) \, $,  the endomorphisms of  $ V $  as an ordinary  $ \bk $--module.  This is again a free  $ \bk $--supermo\-dule,  $ \; \End(V) = \End(V)_\zero \oplus \End(V)_\uno \, $,  \, where $ \End(V)_\zero $  are the morphisms which preserve the parity, while  $ \End(V)_\uno $  are the morphisms which reverse the parity.  If  $ V $  has finite rank, and we choose a basis for  $ V $  of homogeneous elements (writing first the even ones), then  $ \End(V)_\zero $  is the set of all diagonal block matrices, while  $ \End(V)_\uno $  is the set of all off-diagonal block matrices.  Thus  $ \End(V) $  is a Lie  $ \bk $--superalgebra  with bracket
 $ \,\; [A,B] \, := \, A B - {(-1)}^{|A||B|} \, B A \;\, $  for all homogeneous  $ \, A, B \in \End(V) \;\, $
and  $ 2 $--operation  $ \,\; C^{\langle 2 \rangle} := C \, C \;\, $
 for all odd  $ C \, $.
                                                           \par
   The standard example is $ \, V := \bk^{p|q} = \bk^p \oplus \bk^q \, $,  with  $ \, V_\zero := \bk^p \, $  and  $ \, V_\uno := \bk^q \, $.  In this case we also write  $ \; \End\big(\bk^{m|n}\big) \! := \End(V) \; $  or  $ \; \rgl_{\,p\,|q} := \End(V) \; $.   \hfill  $ \blacklozenge $
\end{example}

\medskip

\begin{free text}  \label{Lie-salg_funct}
 {\bf Functorial presentation of Lie superalgebras.}  \ Let  $ \salg_\bk $  be the category of commutative  $ \bk $--superalgebras  (see section  \ref{first_preliminaries})  and  $ \lie_\bk $  the category of Lie  $ \bk $--algebras.  Any Lie  $ \bk $--superalgebra  $ \, \fg \in \slie_\bk \, $  yields a functor
 $ \; \cL_\fg : \salg_\bk \longrightarrow \lie_\bk \; $,  \,
 which is given on objects by
 $ \; \cL_\fg(A) \, := \, \big( A \otimes \fg \,\big)_\zero \, = \, A_\zero \otimes \fg_\zero \,
\oplus \, A_\uno \otimes \fg_\uno \; $,  \, for all  $ \, A \in \salg_\bk \, $:
 indeed,  $ \, A \otimes \fg \, $  is a Lie superalgebra (in a suitable sense, on the  $ \bk $--superalgebra  $ A $)  on its own, its Lie bracket being defined canonically via sign rules by
 $ \; \big[\, a \otimes  X \, , \, a' \otimes X' \,\big] \, :=
\, {(-1)}^{|X|\,|a'|} \, a\,a' \otimes \big[X,X'\big] \; $,
 and  $ \cL_\fg(A) $  is its even part, hence it is a Lie algebra (everything is trivial to verify: see \cite{bmpz},  or  \cite{ccf},  Proposition 11.2.5, for details).  In particular, this applies to the Lie superalgebra  $ \, \fg := \End(V) \, $,  where  $ V $  is any free  $ \bk $--supermodule.  Note also that  $ \rGL(V) $  --- see  Example \ref{exs-supvecs}{\it (b)}  ---   is then a subfunctor of  $ \cL_{\text{\sl End}(V)} \, $.

\vskip5pt

   This ``functorial presentation'' of Lie superalgebras can be adapted to representations too.  Indeed, let  $ V $  be a  $ \fg $--module, for a Lie superalgebra  $ \fg \, $:  by definition,  $ V $  is a  $ \bk $--supermodule, and we have a Lie superalgebra morphism  $ \; \phi : \fg \longrightarrow \End(V) \; $  (the representation map).  Now, scalar extension induces a morphism  $ \; \text{\sl id}_A \otimes \phi : A \otimes \fg \longrightarrow A \otimes \End(V) \; $  for each  $ \, A \in \salg_\bk \, $,  whose restriction to the even part gives a morphism  $ \; \big( A \otimes \fg \,\big)_\zero \longrightarrow \big( A \otimes \End(V) \big)_\zero \; $,  that is a morphism  $ \; \cL_\fg(A) \longrightarrow \cL_{\text{\sl End}(V)}(A) \; $  in  $ \lie_\bk \, $.  The whole construction is natural in  $ A \, $,  hence it induces a natural transformation of functors  $ \, \cL_\fg \longrightarrow \cL_{\text{\sl End}(V)} \, $.

\vskip5pt

   In the sequel, we shall call  {\it quasi-representable\/}  any functor  $ \; \cL : \salg_\bk \! \longrightarrow \lie_\bk \; $  for which there exists a Lie  $ \bk $--superalgebra  $ \fg $  such that  $ \, \cL = \cL_\fg \; $.  Any such functor is even  {\sl representable\/}  (in the usual sense) as soon as the  $ \bk $--module  $ \fg $  is finitely generated projective: indeed, in this case  $ \fg $  is a  $ \bk $--direct  summand  of a finite rank free  $ \bk $--supermodule,  say  $ \mathfrak{f} = \fg \oplus \mathfrak{h} \, $,  thus  $ \mathfrak{f}^* \cong \fg^* \oplus \mathfrak{h}^* \, $  and  $ \cL_\fg $  is then represented by the commutative  $ \bk $--superalgebra  generated by  $ \fg^* $  inside  $ S(\mathfrak{f}^*) \; $.

\vskip5pt

   Finally, note that all this has a natural, non-super counterpart which is obtained by letting ``Lie algebras'' replace ``Lie superalgebras'' all over the place.
\end{free text}

\medskip

  \subsection{The tangent Lie superalgebra of a supergroup}  \label{Lie(G)} {\ }

\smallskip

   We now quickly recall how to associate a Lie superalgebra with a supergroup scheme.  Further details can be found in  \cite{ccf}, \S\S 11.2--5.

\smallskip

   Let  $ \, A \in \salg \, $  and let  $ \, A[\varepsilon] := A[x]\big/\big(x^2\big) \, $  be the  {\sl superalgebra  of dual numbers\/}  over  $ A \, $,  in which  $ \, \varepsilon := x \! \mod \! \big(x^2\big) \, $  is taken to be  {\it even}.  Then  $ \, A[\varepsilon] = A \oplus A \varepsilon \, $,  and there are two natural morphisms
 $ \; i_{{}_A} : A \longrightarrow A[\varepsilon] \, $,  $ \, a \;{\buildrel {\,i_{{}_{A_{\,}}}} \over \mapsto}\; a \, $,  and
 $ \; p_{{}_A} : A[\varepsilon] \longrightarrow A \, $,  $ \, \big( a + a'\varepsilon \big) \;{\buildrel {\,p_{{}_{A_{\,}}}} \over \mapsto}\; a \; $,
 \; such that  $ \; p_{{}_A} \! \circ i_{{}_A} = \, {\mathrm{id}}_A \; $.

\medskip

\begin{definition} \label{tangent_Lie_superalgebra}
   Given a supergroup  $ \bk $--functor  $ \, G : \salg_\bk \! \longrightarrow \! \grps \, $,  let  $ \; G(p_A) : G (A(\varepsilon)) \! \longrightarrow \! G(A) \; $  be the morphism associated with  $ \; p_A : A[\varepsilon] \! \longrightarrow \! A \; $.  There then exists a unique functor  $ \; \Lie(G) : \salg_\bk \!\longrightarrow \!\sets \; $  given on objects by  $ \; \Lie(G)(A) := \Ker\,\big(G(p)_A\big) \; $.
   \hfill $ \diamondsuit $
\end{definition}

\medskip

   The key fact is that when  $ G $  is a supergroup  $ \Lie(G) $  is a Lie algebra valued functor, i.e.~a functor  $ \; \Lie(G) : \salg_\bk \longrightarrow \lie_\bk \; $:  this is by no means evident, since the very definition only assures that that functor is group-valued.  In fact, stating that  $ \Lie(G) $  is actually Lie algebra valued requires a non-trivial proof (like in the classical case): we refer for this to  \cite{ccf},  Ch.\ 11 (with the few adaptations needed for the present setup), and restrict ourselves to quickly sketching here the main steps.
 \vskip5pt
   The Lie structure on any object  $ \, \Lie(G)(A) \, $  is introduced as follows.  First, define the  {\it adjoint action\/}  of  $ G $  on  $ \Lie(G) $  as given, for every  $ \, A \in \salg_\bk \, $,  by
 \vskip4pt
   \centerline{ $ \Ad : G(A)  \longrightarrow  \rGL\big(\Lie(G)(A)\big) \quad ,  \qquad  \Ad(g)(x) \, := \, G (i)(g) \cdot x \cdot {\big(G (i)(g)\big)}^{-1} $ }
 \vskip5pt
\noindent
 for all  $ \, g \in G(A) \, $,  $ \, x \in \Lie(G)(A) \, $.  Second, define the  {\it adjoint morphism\/}  $ \ad $  as
 \vskip4pt
   \centerline{ $ \ad \, := \, \Lie(\Ad) : \Lie(G) \longrightarrow \Lie(\rGL(\Lie(G))) := \End(\Lie(G)) $ }
\vskip5pt
\noindent
 and finally define  $ \; [x,y] := \ad(x)(y) \; $  for all  $ \, x,y \in \Lie(G)(A) \, $.  Then we have the following:

\medskip

\begin{proposition}  \label{Lie-funct_Lie(G)}
 Given  $ \, G \in \sgrps_\bk \, $,  let  $ \, \omega_e(G)  := {\mathcal{O}(G)}^+ \big/ {\big( {\mathcal{O}(G)}^+ \big)}^2 $  and  $ \, \fg := T_e(G) = {\omega_e(G)}^* = \Hom_\bk\big(\omega_e(G),\bk\big) \, $  be the  {\sl cotangent}  and  {\sl tangent}  supermodule to  $ G $  at the unit  $ \, e \in G \, $.
 \vskip5pt
   (a) \,  $ \Lie(G) \, $  with the bracket  $ \, [\,\cdot\, , \cdot\, ] \, $  above yields a Lie algebra valued functor
 \vskip3pt
   \centerline{ $ \, \Lie(G) : \salg_\bk \! \relbar\joinrel\relbar\joinrel\longrightarrow \lie_\bk $ }
 \vskip5pt
   (b) \, if  $ \Lie(G) $  is quasi-representable, namely it is of the form  $ \, \Lie(G) = \cL_{\mathfrak{p}} \, $  (see \S \ref{Lie-salg_funct}),  then  $ \mathfrak{p} $  identifies with  $ \fg $  and the latter is endowed with a canonical structure of Lie  $ \bk $--superalgebra.
 \vskip5pt
   (c) \,  $ \, \Lie(G) \, $  is quasi-representable if and only if  $ \, \omega_e(G) $  is finitely generated projective  (over  $ \bk $).  When this is the case,  $ \Lie(G) $  is actually representable.
\end{proposition}

\begin{proof}
  Claim  {\it (a)},  i.e.~the fact that  $ \Lie(G) $  with the bracket  $ [\,\cdot\, , \cdot\, ] $  considered above be a Lie algebra valued functor, is a well known fact: cf.\  \cite{ccf},  \S 11.4 (for instance) for further details.
 \vskip3pt
   As to claim  {\it (b)},  it is also standard (cf.\  \cite{ccf},  \S 11.2) that if  $ \, \Lie(G) = \cL_{\mathfrak{p}} \, $,  then  $ \mathfrak{p} $  necessarily identifies with  $ \, \fg := T_e(G) \, $,  and then the existence of a ``Lie structure'' on  $ \, \Lie(G) = \cL_\fg \, $  endows  $ \, \mathfrak{p} = \fg \, $  with a structure of ``Lie  $ \bk $--superalgebra''  in the usual ``weak sense'': i.e.,  $ \fg $  has a Lie super\-bracket for which conditions  {\it (a)-(b)-(c)\/}  in  Definition \ref{def-Lie-salg}  are fulfilled.  In addition, one has similar, canonical identifications  $ \, \Lie(G) = \cL_{\fg'} \, $  and  $ \, \Lie(G) = \cL_{\fg''} \, $  where  $ \, \fg' := \Der_\bk(\cO(G),\bk) \, $  is the $ \bk $--superalgebra  of  $ \bk $--valued  superderivations of  $ \cO(G) $  and  $ \, \fg'' := \Der_\bk^{\,\ell}(\cO(G)) \, $  is the $ \bk $--superalgebra  of left-invariant superderivations of  $ \cO(G) $  into itself.  Also, both  $ \fg' $  and  $ \fg'' $  bear structures of Lie  $ \bk $--superalgebras  which are isomorphic to that of  $ \fg $  (yielding the Lie algebra structure on each  $ \, \Lie(G)(A) \, $,  for  $ \, A \in \salg_\bk \, $)   --- see e.g.\  \cite{ccf}, \S\S 11.3--6;  there  $ G $  is assumed to be algebraic, but the arguments (taken from classical sources, such as  \cite{dg}, Ch.~II, \S 4)  only require our assumption in  {\it (b)}.
                                                         \par
   What we still need to fix is that, under the assumption in  {\it (b)},  $ \, \fg := T_e(G) \, $  is also endowed with a  $ 2 $--operation  such that  $ \fg $  is a Lie $ \bk $--superalgebra  in the sense of Definition \ref{def-Lie-salg}.  Actually, I introduce such a  $ 2 $--operation  on  $ \fg'' $  and then I use the previous isomorphism(s) to ``transfer'' such a structure onto  $ \fg $  (and onto  $ \fg' $)  as well.  Indeed, the Lie bracket in  $ \, \fg'' := \Der_\bk^{\,\ell}(\cO(G)) \, $  is given by  $ \, [X,Y] = X \!\circ Y - {(-1)}^{|X|\,|Y|} \, Y \!\circ X \, $;  in addition, looking  $ \, Z \in \fg_\uno \, $  as an (odd) left-invariant superderivation of  $ \cO(G) $  one sees at once that  $ \, Z^{\,2} = Z \circ Z \, $  is an even left-invariant superderivation, i.e.~$ \, Z^{\,2} \in \fg_\zero \; $.  Then  $ \, \fg_\uno \!\longrightarrow \fg_\zero \, , \, Z \mapsto Z^{\langle 2 \rangle} := Z^{\,2} \, $,  is well defined and yields a  $ 2 $--operation  in  $ \fg $  that along with  $ [\,\cdot\, , \cdot\, ] $  makes it into a Lie  $ \bk $--superalgebra  as desired (i.e.~in the sense of  Definition \ref{def-Lie-salg}).
 \vskip3pt
   For claim  {\it (c)\/}  the ``if\,'' part is well-known again: if  $ \omega_e(G) $  is finitely generated projective then the same holds true for  $ \, \fg = {\omega_e(G)}^* \, $,  hence (see  \S \ref{Lie-salg_funct})  the functor  $ \, \Lie(G) = \cL_\fg \, $  is representable.  As to the ``only if\,'' part, here is a proof (kindly suggested to the author by prof.\ Masuoka).
                                                         \par
   First, by definition of ``quasi-representable'' (see  \S \ref{Lie-salg_funct})  and by claim {\it (b)\/}  above we have that  $ \text{\sl Lie}(G) $  is quasi-representable if and only if there exist isomorphisms (natural in  $ \, R \in \salg_\Bbbk \, $)
  $$  \text{\sl Lie}(G)(R)  \, \cong \,  \mathcal{L}_\fg(R) := {(\fg \otimes_\Bbbk R)}_\zero = (\fg_\zero \otimes_\Bbbk R_\zero) \oplus (\fg_\uno \otimes_\Bbbk R_\uno)   \eqno (2.1)   $$
   \indent   On the other hand, definitions give  $ \; \text{\sl Lie}(G)(R) \, \cong \, \text{\sl Hom}_{\text{(smod)}_\Bbbk}\big(\, \omega_e(G) \, , R \, \big) \; $   ---
where  $ \, \text{(smod)}_\Bbbk \, $  denotes the category of  $ \Bbbk $--supermodules  ---   so that
  $$  \text{Hom}_{\text{(smod)}_\Bbbk}\big(\, \omega_e(G) \, , R \, \big)  \; \cong \;  \text{Hom}_{\text{(mod)}_\Bbbk}\big( {(\,\omega_e(G))}_\zero \, , R_\zero \, \big) \oplus \text{Hom}_{\text{(mod)}_\Bbbk}\big( {(\,\omega_e(G))}_\uno \, , R_\uno \, \big)  $$
where now  $ \, \text{(mod)}_\Bbbk \, $  denotes the category of  $ \Bbbk $--modules.  Thus (2.1) above reads (for all  $ R \, $, etc.)
  $$  \text{Hom}_{\text{(mod)}_\Bbbk}\!\big( {(\omega_e(G))}_\zero \, , R_\zero \big) \oplus \text{Hom}_{\text{(mod)}_\Bbbk}\!\big( {(\omega_e(G))}_\uno \, , R_\uno \big)  \, \cong \,  (\fg_\zero \otimes_{\Bbbk\!} R_\zero) \oplus (\fg_\uno \otimes_{\Bbbk\!} R_\uno)  \hskip7pt   \eqno (2.2)  $$
   \indent   Given  $ \, M \in \text{(mod)}_\Bbbk \, $  we associate with it a couple of (super)commutative  $ \Bbbk $--superalgebras  $ M_+ $  and  $ M_- $  defined as follows.  As  $ \Bbbk $--algebras  they both are the central extension of  $ \Bbbk $  by  $ M \, $  (that is  $ \, M_+ := \Bbbk \oplus M =: M_- \, $  with  $ \; m'\,m'' = 0 \; $  for  $ \, m', m'' \in M \, $),  {\sl but the  $ \Z_2 $--grading  is different},  namely
  $$  {(M_+)}_\zero \, := \Bbbk \oplus M \; ,  \quad  {(M_+)}_\uno \, := \{0\} \; ,  \quad \qquad  {(M_-)}_\zero \, := \Bbbk \; ,  \quad  {(M_-)}_\uno \, := M  $$
   \indent   Now assume that  $ \text{\sl Lie}(G) $  is quasi-representable, hence (2.2) holds true.  For  $ \, R := M_+ \, $  this gives
  $$  \text{Hom}_{\text{(mod)}_\Bbbk}\big( {(\,\omega_e(G))}_\zero \, , R_\zero \, \big)  \, = \,  \text{Hom}_{\text{(mod)}_\Bbbk}\big( {(\,\omega_e(G))}_\zero \, , \Bbbk \, \big) \oplus \, \text{Hom}_{\text{(mod)}_\Bbbk}\big( {(\,\omega_e(G))}_\zero \, , M \, \big)  $$
and
  $ \; \fg_\zero \otimes_\Bbbk R_\zero = \fg_\zero \otimes_\Bbbk (\Bbbk \oplus M) = \big( \fg_\zero \otimes_\Bbbk \Bbbk \big) \oplus \big( \fg_\zero \otimes_\Bbbk M \big) \; $,  \,
whereas
  $ \; \text{Hom}_{\text{(mod)}_\Bbbk}\big( {(\,\omega_e(G))}_\uno \, , R_\uno \, \big) = \{0\} \; $,  $ \; \fg_\uno \otimes_\Bbbk R_\uno = \{0\} \; $.
Therefore condition (2.2) reads
  $$  \text{Hom}_{\text{(mod)}_\Bbbk}\big( {(\,\omega_e(G))}_\zero \, , \Bbbk \, \big) \oplus \, \text{Hom}_{\text{(mod)}_\Bbbk}\big( {(\,\omega_e(G))}_\zero \, , M \, \big)  \, \cong \,  \big( \fg_\zero \otimes_\Bbbk \Bbbk \big) \oplus \big( \fg_\zero \otimes_\Bbbk M \big)  $$
and eventually (as  $ \; \text{Hom}_{\text{(mod)}_\Bbbk}\big( {(\,\omega_e(G))}_\zero \, , \Bbbk \, \big) \, = \, {\big({(\,\omega_e(G))}_\zero\big)}^* \, = \, {\big({(\,\omega_e(G))}^*\big)}_\zero \, = \, \fg_\zero \, = \, \fg_\zero \otimes_\Bbbk \Bbbk \; $)
  $$  \text{Hom}_{\text{(mod)}_\Bbbk}\big( {(\,\omega_e(G))}_\zero \, , M \, \big)  \,\; \cong \;\,  \fg_\zero \otimes_\Bbbk M   \eqno (2.3)  $$
This last condition is natural in  $ M \, $:  this together with the fact that the functor  $ \, M \mapsto \fg \otimes_\Bbbk M \, $  preserves surjections, implies that  $ {\big( \omega_e(G) \big)}_\zero $  is  $ \Bbbk $--projective.
  \vskip3pt
   For  $ \, R := M_- \, $  we can repeat the same argument.  We find
  $$  \text{Hom}_{\text{(mod)}_\Bbbk}\big( {(\,\omega_e(G))}_\zero \, , R_\zero \, \big)  \, = \,  \text{Hom}_{\text{(mod)}_\Bbbk}\big( {(\,\omega_e(G))}_\zero \, , \Bbbk \, \big)  $$
and
  $ \,\; \text{Hom}_{\text{(mod)}_\Bbbk}\big( {(\,\omega_e(G))}_\uno \, , R_\uno \, \big) = \text{Hom}_{\text{(mod)}_\Bbbk}\big( {(\,\omega_e(G))}_\uno \, , M \, \big) \; $,  \,
while
  $ \; \fg_\zero \otimes_\Bbbk R_\zero = \fg_\zero \otimes_\Bbbk \Bbbk \cong \fg_\zero \; $,  $ \; \fg_\uno \otimes_\Bbbk R_\uno = \fg_\uno \otimes_\Bbbk M \; $.
Thus condition (2.2) now reads
  $$  \text{Hom}_{\text{(mod)}_\Bbbk}\big( {(\,\omega_e(G))}_\zero \, , \Bbbk \, \big) \oplus \, \text{Hom}_{\text{(mod)}_\Bbbk}\big( {(\,\omega_e(G))}_\uno \, , M \, \big)  \, \cong \,  \big( \fg_\zero \otimes_\Bbbk \Bbbk \big) \oplus \big( \fg_\uno \otimes_\Bbbk M \big)  $$
and then eventually (like before)
  $$  \text{Hom}_{\text{(mod)}_\Bbbk}\big( {(\,\omega_e(G))}_\uno \, , M \, \big)  \,\; \cong \;\,  \fg_\uno \otimes_\Bbbk M   \eqno (2.4)  $$
As (2.4) is natural in  $ M \, $,  we can now argue like above to infer that  $ {\big( \omega_e(G) \big)}_\uno $  is  $ \Bbbk $--projective.
                                                                     \par
   The outcome is that  $ \, \omega_e(G) = {\big( \omega_e(G) \big)}_\zero \oplus {\big( \omega_e(G) \big)}_\uno \, $  is  $ \Bbbk $--projective,  q.e.d.
  \vskip3pt
   Let now  $ \, \pi : F = \oplus_{i \in I} \Bbbk \relbar\joinrel\relbar\joinrel\twoheadrightarrow {\big( \omega_e(G) \big)}_\zero \, $  be a  $ \Bbbk $--linear  surjection from some free  $ \Bbbk $--module  $ \, F = \oplus_{i \in I} \, $  of rank  $ |I| $  onto $ {\big( \omega_e(G) \big)}_\zero \, $.  By projectivity of  $ {\big( \omega_e(G) \big)}_\zero $  there exists a splitting  $ \, \sigma : {\big( \omega_e(G) \big)}_\zero \lhook\joinrel\relbar\joinrel\relbar\joinrel\rightarrow F \, $  of  $ \pi \, $.  Then  $ \, \sigma \in \text{Hom}_{\text{(mod)}_\Bbbk}\big(\! {(\,\omega_e(G))}_\zero \, , F \, \big) \, \cong \, \fg_\zero \otimes_\Bbbk F \, \cong \, \fg_\zero \otimes_\Bbbk \big(\! \oplus_{i \in I} \Bbbk \big) \, $   --- by (2.3) with  $ \, M := F \, $  ---   hence there exists some  {\sl finite\/}  index subset  $ \, J \subseteq I \, $  such that  $ \sigma $  actually belongs to  $ \,  \fg_\zero \otimes_\Bbbk \big(\! \oplus_{i \in J} \Bbbk \big) \, $,  which means that the image of  $ \sigma $  is contained in  $ \, F' := \oplus_{i \in J} \Bbbk \, $.  But then the restriction of  $ \pi $  to  $ F' $  is still surjective, hence  $ {\big( \omega_e(G) \big)}_\zero $  is finitely generated.
                                                                     \par
   An entirely similar analysis shows that  $ {\big( \omega_e(G) \big)}_\zero $  is finitely generated as well.
\end{proof}

\medskip

   In the following we are interested in affine  $ \bk $--supergroups  of a specific class, characterized in terms of  $ \Lie(G) \, $,  as the following definition (not very restrictive, indeed) specifies:

%
%
%

\medskip

%
%

\begin{definition}  \label{def_fine-sgroups}
 We call  {\it fine\/}  any affine  $ \bk $--supergroup  $ \, G \in \sgrps_\bk \, $  whose associated functor  $ \Lie(G) $  is quasi-representable, say  $ \, \Lie(G) = \cL_\fg \, $,  for some Lie  $ \bk $--superalgebra  $ \fg $  whose odd part  $ \fg_\uno $  is  {\sl free\/}  of  {\sl finite rank\/}  as a  $ \bk $--supermodule.  We denote by  $ \fsgrps_\bk $  the full subcategory of  $ \sgrps_\bk $  whose objects are all  {\sl fine\/}  $ \bk $--supergroups.
 \hfill   $ \diamondsuit $
\end{definition}

%
%
%

\bigskip

\section{Splittings}  \label{splittings}

\smallskip

   {\ } \;\;   In this section we consider the notion of ``global splitting''   --- roughly, a ``separation of variables'' property ---   for superalgebras, Hopf superalgebras, (affine) superschemes and supergroups.  We shall see that if  $ \bk $  is a field then all (affine)  $ \bk $--supergroups  do admit ``global splittings'': this is essentially due to a result by Masuoka on the splitting of commutative Hopf superalgebras over a field.
                                              \par
   We shall also introduce some other (easy, yet interesting) ``splitting results'' for the  $ A $--points  of a  $ \bk $--supergroup  when  $ A $  ranges in special subcategories of  $ \salg_\bk \, $.

\bigskip

 \subsection{Augmentations and split superalgebras}  \label{split-salg} {\ }

\smallskip

   In the following, we shall think of  $ \bk $  as being a totally even superalgebra, i.e.~we identify  $ \bk $  with  $ \cJ_{\alg_\bk}^{\salg_\bk}(\bk) \, $   --- see  \S \ref{superalgebras}.

\vskip9pt

\begin{free text}  \label{aug/spl/salg}
 {\bf Augmentations for superalgebras and related constructions.}  For any superalgebra  $ \, A \in \salg_\bk \, $,  we call  {\it augmentation\/}  of  $ A $  any morphism of  $ \, \bk $--superalge\-bras  $ \, \eps : A \longrightarrow \bk \, $.  We denote by  $ \asalg_\bk $  the category of ``augmented (commutative) superalgebras'': its objects are pairs  $ \big(\, A \, , \, \eps \,\big) $  where  $ \, A \in \salg_\bk \, $  and  $ \eps $  is an augmentation of  $ A \, $,  and its morphisms  $ \; \big(\, A' \, , \, \eps' \,\big) \longrightarrow \big(\, A'' \, , \, \eps'' \,\big) \; $  are given by morphisms  $ \phi : A' \longrightarrow A'' \, $  in  $ \salg_\bk $  such that  $ \, \eps'' \circ \phi = \eps \, $.  We also identify  $ \, \big( A \, , \, \eps \,\big) \cong A \, $.

\vskip7pt

   Given  $ \, \big(\, A \, , \, \eps \,\big) \in \asalg_\bk \, $  one has  $ \, \Ker\,(\eps\,) = A_\zero^+ \oplus A_\uno \, $  where  $ \, A_\zero^+ := \Ker\,(\eps\,) \cap A_\zero \, $.  Define  $ \; W^{\!A} := A_\uno \big/ A_\zero^+ A_\uno \, $  and let  $ \, \bigwedge W^{\!A} \, $  be the exterior  $ \bk $--algebra  of  $ W^{\!A} \, $;  recall also that  $ A_\uno^{(1)} $  is the unital  $ \bk $--subalgebra  of  $ A $  generated by  $ A_\uno $  (see \S~\ref{superalgebras}).  Then  $ \, \overline{A} := A \big/ J_A \, $,  $ A_\uno^{(1)} $  and  $ \bigwedge W^{\!A} $  all inherit from  $ \big(\, A \, , \, \eps \,\big) $  a natural structure of augmented $ \bk $--superalgebra.
                                                 \par
   It follows also that both  $ \, \overline{A} \otimes_\bk A_\uno^{(1)} \, $  and  $ \, \overline{A} \otimes_\bk \bigwedge W^{\!A} \, $  have a natural structure of commutative, unital, augmented  $ \bk $--superalgebra,  i.e.~$ \; \overline{A} \otimes_\bk A_\uno^{(1)} , \, \overline{A} \otimes_\bk \bigwedge W^{\!A} \, \in \asalg_\bk \;\, $.

\medskip

\begin{definition}  \label{wkspl-salg}
 Given any  $ \, A \in \asalg_\bk \, $,  we say that it is  {\it weakly split\/}  if there exists a section  $ \, \sigma_{\!{}_A} : \overline{A} \lhook\joinrel\relbar\joinrel\rightarrow A \, $  of the projection  $ \, \pi_{\!{}_A} : A \relbar\joinrel\twoheadrightarrow \overline{A} \, $   --- both being meant as morphisms in  $ \, \asalg_\bk \, $.
   All pairs  $ \big( A \, , \sigma_{\!{}_A} \big) $  as above form a category, denoted by  $ \, \wkspsalg_\bk \, $,  or just  $ \wkspsalg \, $,  where morphisms are all those in  $ \asalg_\bk $
   \hbox{which are compatible (in the obvious sense) with the sections.
 \hfill   $ \diamondsuit $}
%
\end{definition}

\smallskip

\begin{remark}
 Given  $ \, \big( A \, , \sigma_{\!{}_A} \big) \in \wkspsalg_\bk \, $,  write  $ \, \overline{A} = \sigma_{\!{}_A}\big(\,\overline{A}\,\big) \subseteq A \, $.  Then the multiplication map  $ m_{\!{}_A} $  in  $ A $  yields an  $ \overline{A} $--linear  projection  $ \, \overline{A} \otimes_\bk A_\uno^{(1)} {\buildrel {m'_A} \over {\relbar\joinrel\relbar\joinrel\twoheadrightarrow}}\, A \; $.
\end{remark}

\medskip

   The above remark shows that any weakly split superalgebra  $ A $  can be recovered as a quotient   --- in the category of $ \overline{A} $--modules,  via the multiplication map ---   of  $ \, \overline{A} \otimes_\bk A_\uno^{(1)} \, $.  This invites us to consider those cases when this description is ``optimal'', which leads naturally to next definition:
  \eject

\begin{definition}  \label{def-stsplit-salg} {\ }
 \vskip5pt
   {\it (a)}  Any  $ \, \big( A \, , \sigma_{\!{}_A} \big) \in \wkspsalg_\bk \, $  is said to be  {\it split\/}  if the natural  $ \overline{A} $--linear  morphism (see above)  $ \, \overline{A} \otimes_\bk A_\uno^{(1)} {\buildrel {m'_A} \over {\relbar\joinrel\relbar\joinrel\twoheadrightarrow}}\, A \, $  induced by multiplication in  $ A $  is an isomorphism in  $ \wkspsalg_\bk \, $.
 We denote by  $ \, \splsalg_\bk \, $  the full subcategory of  $ \wkspsalg_\bk $  of all split  $ \bk $--superalgebras.
 \vskip5pt
   {\it (b)}  Any augmented superalgebra  $ \, \big( A \, , \, \eps \,\big) \in \asalg_\bk \, $  is said to be  {\it strongly split\/}  if there exists an isomorphism  $ \, \zeta :  A \,{\buildrel \cong \over {\lhook\joinrel\relbar\joinrel\relbar\joinrel\twoheadrightarrow}}\, \overline{A} \otimes_\bk \bigwedge W^{\!A} \; $  in  $ \asalg_\bk \, $.  We denote by  $ \, \stspsalg_\bk \, $  the full subcategory of  $ \asalg_\bk $  given by all strongly split  $ \bk $--superalgebras.
 \hfill   $ \diamondsuit $
\end{definition}

\medskip

   We introduce now another special subclass of (commutative) superalgebras.
 \vskip5pt
   Let  $ \, \mathcal{A} \in \alg_\bk \, $  and let  $ M $  be an  $ \mathcal{A} $--module.  Then  $ \, A_{\mathcal{A},M} := \mathcal{A} \oplus M \, $  has a natural structure of unital, commutative  $ \bk $--superalgebras  defined as follows: the  $ \Z_2 $--splitting  is given by
  $  \; {\big( A_{\mathcal{A},M} \big)}_\zero := \mathcal{A} \, $,  $ \; {\big( A_{\mathcal{A},M} \big)}_\uno := M \, $, \, and the  $ \bk $--algebra  structure is the unique one such that  $ \mathcal{A} $  is a  $ \bk $--subalgebra,  $ \, M \cdot M := \{0\} \, $  and  $ \, \alpha \cdot m := \alpha.m =: m \cdot \alpha \, $  for all  $ \, \alpha \in \mathcal{A} \, $,  $ \, m \in M \, $,  where  $ \, \alpha.m \, $  is given by the  $ \mathcal{A} $--action  on  $ M \, $.  In a formula,
  $$  (\alpha + m) \cdot \big( \alpha' \! + m' \big)  \; := \;  \alpha \, \alpha' \, + \big( \alpha.m' \! + \alpha'.m \big)   \eqno \forall \;\; \alpha, \alpha' \in \mathcal{A} \, , \; m, m' \in M  \quad  $$
   \indent   By construction, the  $ \bk $--superalgebra  $ \, A := A_{\mathcal{A},M} \, $  has the property that  $ \, A_\uno^{\,2} = \{0\} \, $.  Conversely, let  $ \, A \in \salg_\bk \, $  be such that  $ \, A_\uno^{\,2} = \{0\} \, $:  then  $ A $  is of the previous form, namely  $ \, A := A_{\mathcal{A},M} \, $  for  $ \, \mathcal{A} := A_\zero \, $  and  $ \, M := A_\uno \, $.

\smallskip

\begin{definition}  \label{cext-salg}
 We call  {\it augmented central extension  (\/$ \bk $--superalgebra)},  or simply  {\it central extension},  any  $ \, \big( A \, , \epsilon \big) \in \asalg_\bk \, $  which (as a superalgebra) is of the form  $ \, A = A_{\mathcal{A},M} \, $  as above with  $ \, \epsilon(M) = 0 \, $   --- in other words, such that  $ \, A_\uno^{\,2} = \{0\} \, $  and  $ \, \epsilon\big(A_\uno\big) = \{0\} \, $.  We denote by  $ \, \cexsalg_\bk \, $  the full subcategory of  $ \salg_\bk $  whose objects are all the central extension  $ \bk $--superalgebras.
 \hfill   $ \diamondsuit $
\end{definition}

\smallskip

   Next (easy) result shows the links between these special subcategories of  $ \salg_\bk \, $:
\end{free text}

\smallskip

\begin{proposition}  \label{wsp->spl / cex->spl}  {\ }
 \vskip2pt
   (a) \,  The category  $ \stspsalg_\bk $  identifies in a natural way with a subcategory of  $ \splsalg_\bk \, $,  and similarly  $ \splsalg_\bk $  identifies with a subcategory of  $ \wkspsalg_\bk \, $.  In other words, all  {\sl strongly split}  $ \bk $--superalgebras  are
 {\sl split},  and all  {\sl split}  $ \bk $--superalge\-bras  are  {\sl weakly split}.
 \vskip3pt
   (b) \,  The category  $ \cexsalg_\bk $  identifies in a natural way with a subcategory of  $ \splsalg_\bk \, $:  i.e., all  {\sl central extension}  $ \bk $--superalgebras  are (naturally)  {\sl split}.
\end{proposition}

\begin{proof}
 {\it (a)} \,  Given  $ \, A \in \stspsalg_\bk \, $,  by the isomorphism  $ \; \zeta : A \cong \overline{A} \otimes_\bk \bigwedge W^{\!A} \; $
one has  $ \, A_\uno = \overline{A} \otimes_\bk \Big( \mathop\oplus\limits_{\text{odd} \!\ s} {\big( W^{\!A} \big)}^{\wedge s} \Big) \, $  hence  $ \, A_\uno^{(1)} = \bk \oplus \Big(\, \overline{A} \otimes_\bk {\big( \bigwedge W^{\!A} \big)}^+ \Big) \, $,  and then an easy calculation shows that  $ \, A \in \splsalg_\bk \, $,  as expected.  Moreover, if  $ \, A \in \splsalg_\bk \, $  then the monomorphism  $ \; \overline{A} \lhook\joinrel\longrightarrow \overline{A} \otimes_\bk  A_\uno^{(1)} \, \big(\, \overline{a} \mapsto \overline{a} \otimes 1 \big) \, $  composed with the isomorphism  $ \, \overline{A} \otimes_\bk  A_\uno^{(1)} \lhook\joinrel\relbar\joinrel\twoheadrightarrow A \, $  (inverse to the one given by definition) yields a section  $ \, \overline{A} \lhook\joinrel\longrightarrow A \, $  of  $ \, \pi_{\!{}_A} : A \relbar\joinrel\twoheadrightarrow \overline{A} \, $,  so that  $ \, A \in \wkspsalg_\bk \, $,  \, q.e.d.
 \vskip5pt
   {\it (b)} \,  Let  $ \, A = A_\zero \oplus A_\uno \in \cexsalg_\bk \, $,  so that  $ \, A_\uno^{\,2} = \{0\} \, $.  Then definitions give  $ \, J_A := \big( A_\uno \big) = A_\uno^{\,2} \oplus A_\uno = A_\uno \, $  hence  $ \, \overline{A} := A \big/ J_A = A \big/ A_\uno = A_\zero \, $,  and  $ \, A_\uno^{(1)} = \bk \oplus A_\uno \, $.  Thus  $ \, A = A_\zero \oplus A_\uno = \overline{A} \oplus A_\uno \cong \overline{A} \otimes_\bk A_\uno^{(1)} \, $,  so that  $ \, A \in \splsalg_\bk \; $,  \, q.e.d.
\end{proof}

\medskip

\begin{free text}  \label{aug/spl/Hsalg}
 {\bf Strongly split Hopf superalgebras.}  Any commutative Hopf  $ \bk $--superalge\-bra,  say  $ \, H \in \Hsalg_\bk \, $,  is naturally augmented, in the sense of  \S \ref{aug/spl/salg},  its augmentation being the counit: so  $ \Hsalg_\bk $  naturally identifies with a subcategory of  $ \asalg_\bk \, $,  and all constructions therein make sense for Hopf algebras.  In addition, there are now some extra features.

\medskip

   First, in the Hopf setup  $ \, J_H := \big( H_\uno \big) \, $  is in fact a  {\it Hopf ideal\/}  of  $ H $.  Therefore,  $ \, \overline{H} := H \big/ J_H \, $  {\it is a classical (i.e.~super but  {\sl totally even})  commutative Hopf algebra}.
                                                                         \par
   Second, the coproduct of  $ H $  induces also a structure of super left  $ \overline{H} $--comodule  on  $ H \, $  (via the projection  $ \, H \relbar\joinrel\twoheadrightarrow \overline{H} \; $),  such that  {\it  $ H $  is a counital super left  $ \overline{H} $--comodule  $ \bk $--algebra}.
                                                                         \par
   Third, letting  $ \, \epsilon : H \longrightarrow \bk \, $  be the counit map of  $ H $,  let  $ \, H^+ := \Ker\,(\epsilon) \, $,  $ \, H_\zero^+ := H_\zero \cap H^+ \, $,  $ \, W^H := H_\uno \big/ H_\zero^+ H_\uno \, $  and  $ \, \bigwedge W^H \, $  as in  \S \ref{aug/spl/salg}.

\smallskip

   As in  \S \ref{aug/spl/salg},  $ \; \overline{H} \otimes \bigwedge W^H \; $  has
%
%
 a natural structure of a
commutative superalgebra, endowed with a counit map; moreover, the coproduct of  $ \overline{H} $  induces on  $ \; \overline{H} \otimes \bigwedge W^H \; $  a super left  $ \overline{H} $--comodule  structure, so that  {\it  $ \; \overline{H} \otimes \bigwedge W^H \, $  is a super counital left  $ \overline{H} $--comodule  $ \bk $--algebra}.  The notion of  ``split'' (commutative) Hopf superalgebra   --- introduced by Masuoka ---   then reads as follows:
\end{free text}

\vskip4pt

\begin{definition}  \label{stsplit-Hsalg}
 Any  $ \, H \in \Hsalg_\bk \, $  is said to be  {\it strongly split\/}  if  $ W^H $  is  $ \bk $--free  and there is an isomorphism  $ \, \zeta : H \,{\buildrel \cong \over {\lhook\joinrel\relbar\joinrel\relbar\joinrel\twoheadrightarrow}}\, \overline{H} \otimes_\bk \bigwedge \! W^H \; $  of super counital left  $ \overline{H} $--comodule  $ \bk $--algebras.
 \hfill   $ \diamondsuit $
\end{definition}

\smallskip

\begin{remark}  \label{barH-coaction_W}
 The right coadjoint coaction of  $ H $  canonically induces a right  $ \overline{H} $--coaction  onto  $ H \, $;  it is easy to see that this induces a  $ \overline{H} $--coaction  onto  $ W^H \, $,  hence on  $ \bigwedge W^H $  too.
                                                                       \par
   Now assume  $ H $  is strongly split and  $ \, \zeta :  H \,{\buildrel \cong \over {\lhook\joinrel\relbar\joinrel\relbar\joinrel\twoheadrightarrow}}\, \overline{H} \otimes_\bk \bigwedge \! W^H \; $  is a splitting map as in  Definition \ref{stsplit-Hsalg},  we can endow then  $ \, \overline{H} \otimes_\bk \bigwedge \! W^H \, $  with the push-forward (via  $ \zeta \, $)  of the coproduct of  $ H \, $:  thus  $ \, H_\sigma := \overline{H} \otimes_\bk \bigwedge \! W^H \, $  itself is a Hopf superalgebra (isomorphic to  $ H $)  such that  $ \, \overline{H_\sigma} \cong \overline{H} \, $  and  $ \, W^{\!H_\sigma} \! \cong W^H \, $  in a canonical way.  Now, the right coadjoint coaction of  $ H_\sigma $  induces again a right coaction of  $ \overline{H_\sigma} $  onto  $ \bigwedge W^{\!H_\sigma} $  (as above): it is then immediate to see that   --- via the identifications  $ \, \overline{H_\sigma} \cong \overline{H} \, $  and  $ \, \bigwedge W^{\!H_\sigma} \! \cong \bigwedge W^H $  ---   {\sl this coaction is the same as the one of  $ \, \overline{H} $  on  $  \bigwedge W^H \, $}.

\end{remark}

\smallskip

   The following result, due to Masuoka, ensures that  $ H $  {\sl is\/}  strongly split (yet he omits the ``strong'') when  {\sl the ground ring  $ \bk $
  is a field with
%
%
 $ \, \text{\sl char}(\bk) \not= 2 \, $}  (see also  \cite{bos1},  Theorem 1):

\medskip

\begin{theorem}  \label{Masuoka}
 {\sl  (cf.~\cite{ms1},  Theorem 4.5)}
                                         \par
   If\/  $ \bk $  is a field with  $ \, \text{\sl char}(\bk) \not= 2 \, $,  then each commutative Hopf\/  $ \bk $--superalgebra  is strongly split.
\end{theorem}

\smallskip

\begin{free text}  \label{ex-split}
 {\bf Examples and counterexamples.}
 \vskip5pt
   {\it (a)} \,  Consider on  $ \, \bk\big[\, \underline{x} \, , \underline{\xi} \,\big] := \bk\big[\,x_1,\dots,x_n,\xi_1,\dots,\xi_m\big] \in \salg_\bk \, $  the standard augmenta\-tion given by  $ \, \epsilon(x_i) := 0 \, $,  $ \, \epsilon(\xi_j) := 0 \, $.  Then  $ \, \big( \bk\big[\, \underline{x} \, , \underline{\xi} \,\big] \, , \, \epsilon \,\big) \in \stspsalg_\bk \, $.
                                                         \par
   If in addition we consider on  $ \, \bk\big[\, \underline{x} \, , \underline{\xi} \,\big] \, $  the (standard) Hopf superalgebra structure given by  $ \, \Delta(x_i) := x_i \otimes 1 + 1 \otimes x_i \, $,  $ \, \Delta(\xi_j) := \xi_j \otimes 1 + 1 \otimes \xi_j \, $,  $ \, \epsilon(x_i) := 0 \, $,  $ \, \epsilon(\xi_j) := 0 \, $,  $ \, S(x_i) := -x_i \, $,  $ \, S(\xi_j) := -\xi_j \, $,  then  $ \bk\big[\, \underline{x} \, , \underline{\xi} \,\big] $  is even a strongly split Hopf superalgebra.
 \vskip5pt
   {\it (b)} \,  For any invertible  $ \, u \in \bk^\star \, $,  let  $ \; \mathbb{F}_{2;u} := \bk\big[\,x,y,\xi,\eta\big] \Big/ \! \big( x\,y + \xi\,\eta - u \big) \; $  with the unique augmentation  $ \epsilon $  given by  $ \, \epsilon(x) = 1 \, $,  $ \, \epsilon(y) = u \, $  and  $ \, \epsilon(\xi) = \epsilon(\eta) = 0 \,  $.  Then we have  $ \, \overline{\,\mathbb{F}_{2;u}} \cong \bk\big[\,t\,,t^{-1}\big] \, $  (via  $ \, \overline{x} \mapsto t \, $,  $ \, \overline{y} \mapsto u \, t^{-1} \, $),  and moreover there exists a section  $ \, \sigma_{\!{}_{\mathbb{F}_{2;u}}} \!\! : \overline{\,\mathbb{F}_{2;u}} \lhook\joinrel\relbar\joinrel\rightarrow \mathbb{F}_{2;u} \, $  of the projection  $ \, \pi_{\!{}_{\mathbb{F}_{2;u}}} \! : \mathbb{F}_{2;u} \relbar\joinrel\twoheadrightarrow \overline{\,\mathbb{F}_{2;u}} \, $  given by  $ \, \sigma_{\!{}_{\mathbb{F}_{2;u}}}\!(t) := x \, $,  $ \, \sigma_{\!{}_{\mathbb{F}_{2;u}}}\!\big(t^{-1}\big) := u^{-1} \, y \, (1+\xi\,\eta) \, $;  thus  $ \, \mathbb{F}_{2;u} \in \wkspsalg_\bk \, $.  Furthermore, we have  $ \, {\big( \mathbb{F}_{2;u} \big)}_\uno \cong \bk\big[\,\xi,\eta\big] \, $  and one easily sees that  $ \; \mathbb{F}_{2;u} \, \cong \, \sigma_{\!{}_{\mathbb{F}_{2;u}}}\!\big( \overline{\,\mathbb{F}_{2;u}} \,\big) \otimes_\bk \bk\big[\,\xi,\eta\big] \; $  so that  $ \, \mathbb{F}_{2;u} \in \stspsalg_\bk \, $.
 \vskip4pt
   {\it (c)} \,  Again for any  $ \, u \in \bk^\star \, $,  let  $ \; \mathbb{G}_{2;u} := \bk\big[\,x,y,\xi,\eta,\vartheta,\chi\big] \Big/ \! \big( x\,y + \xi\,\eta - u \, , \, \vartheta \, \chi \big) \; $  with the unique augmentation  $ \epsilon $  given by  $ \, \epsilon(x) = 1 \, $,  $ \, \epsilon(y) = u \, $  and  $ \, \epsilon(\xi) = \epsilon(\eta) = \epsilon(\vartheta) = \epsilon(\chi) = 0 \,  $.  Then acting like in  {\it (c)\/}  one finds that  $ \, \mathbb{G}_{2;u} \in \splsalg_\bk \, $  but  $ \, \mathbb{G}_{2;u} \not\in \stspsalg_\bk \, $.
 \vskip4pt
   {\it (d)} \,  Let  $ \; \mathbb{F}_{2;0} := \bk\big[\,x,y,\xi,\eta\big] \Big/ \! \big( x\,y + \xi\,\eta \big) \; $  with the unique augmentation  $ \epsilon $  given by  $ \, \epsilon(x) = \epsilon(y) = \epsilon(\xi) = \epsilon(\eta) = 0 \,  $.  Again like in  {\it (c)\/}  one finds easily that  $ \, \mathbb{F}_{2;0} \not\in \wkspsalg_\bk \, $.
 \vskip4pt
   {\it (e)} \,  Let  $ \; \mathbb{E}_{n,m} := \bk\big[\,x_1,\dots,x_n,\xi_1,\dots,\xi_m\big] \Big/ \! \big( {\{\, \xi_j \, \xi_\ell \,\}}_{j, \ell = 1, \dots, m} \big) \; $  with the obvious augmentation induced by that of  $ \bk\big[\, \underline{x} \, , \underline{\xi} \,\big] $  in  {\it (a)\/}  above.  Then $ \, \mathbb{E}_{n,m} \in \cexsalg_\bk \, $.
 \vskip7pt
   {\it (f)  ---  ``$ \underline{\text{Super numbers}} $''  on classical algebras:} \,  Let  $ \; \mathcal{A} \in \alg_\bk \, $  be a commutative  $ \bk $--algebra.  If  $ \eta $  is a formal  {\sl odd\/}  variable, then  $ \, \mathcal{A}[\eta] := \mathcal{A} \otimes_\bk \bk[\eta] \, $  can naturally be thought of as an object in  $ \salg_\bk \, $;  if we assume  $ \mathcal{A} $  to have some augmentation, and extend it to  $ \mathcal{A}[\eta] $  by setting  $ \, \epsilon(\eta) := 0 \, $,  then clearly  $ \mathcal{A}[\eta] $  is even a(n augmented) central extension.  Letting  $ \aalg_\bk $  denote the category of commutative  $ \bk $--algebras  with augmentation, all this yields a faithful functor  $ \, \aalg_\bk \longrightarrow \cexsalg_\bk \, $.
\end{free text}

\bigskip

 \subsection{Global splittings of superschemes and supergroups}  \label{glob-split_ssch-sgrp} {\ }

\smallskip

   We begin with a general discussion about ``global splittings'' for affine superschemes; later on we shall look at the case of affine supergroups.

\medskip

   We resume notation as in  \S \ref{superalgebras}.
   Let  $ \, X \in \assch_\bk \, $  be an affine  $ \bk $--superscheme, and  $ \, R := \cO(X) \, $  the commutative (unital)  $ \bk $--superalgebra  representing it; then let  $ \; J_R := \big( R_\uno \big) = R_\uno^{\,[2]} \oplus A_\uno \; $  and  $ \; \overline{R} := R \big/ J_R = R_\zero \big/ R_\uno^{\,[2]} \, $.

\medskip

\begin{definition}  \label{def-ass-class-scheme}
 The affine superscheme  $ \, X_{\text{\it ev}} := h_{\overline{R}}
 \, $  represented by  $ \, \overline{R} = \overline{\cO(X)} \, $,  \, i.e.\ such that  $ \, \cO(X_{\text{\it ev}}) = \overline{\cO(X)} \, $,  is called  {\it the classical scheme associated with  $ X \, $}.
 \hfill   $ \diamondsuit $
\end{definition}

\medskip

   The terminology just introduced is motivated by the following result:

\medskip

\begin{proposition}  \label{ass-class-scheme}
 For any  $ \, X \in \assch_\bk \, $,  we have:
 \vskip4pt
  \! (a)  $ \, X_{\text{\it ev}} = X_\zero \, $  (see \S \ref{superalgebras}),  and it can be thought of as a representable functor from  $ \alg_\bk $  to  $ \sets \, $,  i.e.~as a classical   --- that is, ``non super'' ---   affine  $ \bk $--scheme;
 \vskip4pt
  \! (b) \,  the  $ \bk $--superscheme  $ \, X_{\text{\it ev}} = \! X_\zero \, $  identifies with a closed supersubscheme of  $ X \, $;  moreover, every closed supersubscheme\/  $ X $  which is  {\sl classical}  is a closed subscheme of  $ \, X_{\text{\it ev}} = X_\zero \; $.
\end{proposition}

\begin{proof}
 {\it (a)}  Let  $ \, R = \cO(X) \in \salg_\bk \, $,  so that  $ \, X = h_R \, $  and  $ \, X_{\text{\it ev}} := h_{\overline{R}} \; $.  For any  $ \, A \in \salg_\bk \, $  we have  $ \, X_{\text{\it ev}}(A) = h_{\overline{R}}(A) = \Hom_{\salg_\bk}\big(\,\overline{R} \, , A \,\big) = \Hom_{\alg_\bk}\big(\,\overline{R} \, , A_\zero \big) \, $  because  $ \overline{R} $  is totally even.  In addition,  $ \, \Hom_{\alg_\bk}\big(\,\overline{R} \, , A_\zero \big) = \Hom_{\salg_\bk}\big(\, R \, , A_\zero \big) = h_R(A_\zero) = X(A_\zero) = X_\zero(A) \, $,  in that, letting  $ \, \pi : R \relbar\joinrel\twoheadrightarrow \overline{R} \; $,  \, for every  $ \, \varphi \in \Hom_{\alg_\bk}\big(\,\overline{R} \, , A_\zero \big) \, $  we have  $ \, \varphi \circ \pi \in \Hom_{\alg_\bk}\big(\, R \, , A_\zero \big) \, $,  and conversely every  $ \, \phi \in \Hom_{\alg_\bk}\big(\, R \, , A_\zero \big) \, $  kills  $ J_R $  and thus it factors through  $ \overline{R} \; $.  Therefore  $ X_{\text{\it ev}} $  and  $ X_\zero $  coincide on objects, and similarly they do on morphisms.
 \vskip4pt
   {\it (b)} \,  The identification  $ \, X_{\text{\it ev}} = \! X_\zero \, $  is a sheer consequence of the (well known) definition of  {\sl closed supersubscheme},  see for instance  \cite{ccf},  \S 10.1.  Moreover, let  $ Y $  be any (closed) supersubgroup of  $ X $  which is also classical: then the commutative superalgebra  $ \cO(Y) $  representing  $ Y $  is a quotient of  $ \, R := \cO(G) \, $,  and it is totally even, i.e.~it is a (commutative) algebra.  Now, any quotient superalgebra of  $ R $  which is totally even is actually a quotient of  $ \, \overline{R} =: \cO(X_{\text{\it ev}}) = \cO(X_\zero) \, $,  \,\! by construction; applying this to  $ \cO(Y) \, $  we get  $ \, Y \subseteq X_{\text{\it ev}} = \! X_\zero \, $  as a closed subscheme.
\end{proof}

\medskip

   As a consequence of the previous Proposition,  {\sl we shall hereafter denote the classical scheme associated with  $ X $  by  $ X_\zero $  rather than  $ X_{\text{\it ev}} \; $}.
 \vskip7pt
   We can now introduce the notion of ``global splitting'' for an affine superscheme:

\medskip

\begin{definition}  \label{glob-split_sscheme}
 Let  $ \, X \in \assch_\bk \, $  be an affine  $ \bk $--superscheme.   We say that  $ X $  is  {\it globally split\/}  (or  ``{\it g-split\/}'')  if there is a superscheme isomorphim  $ \; X \cong \, X_\zero \times X_{\text{\it odd}} \; $  for some totally odd affine  $ \bk $--superscheme  $  X_{\text{\it odd}} $  having only one  $ \bk $--point.  In addition, we say that  $ X $  is  {\it globally strongly split\/}  (or  ``{\it gs-split\/}'')  if  $ \, X_{\text{\it odd}} = \mathbb{A}_\bk^{\text{\it odd}} \, $  is indeed a (totally odd) affine  $ \bk $--superspace.
 \hfill   $ \diamondsuit $
\end{definition}

\medskip

   The following algebraic characterization is entirely straightforward:

\medskip

\begin{proposition}  \label{char_glob-spl-ssch}
 Let  $ \, X \in \assch_\bk \, $.  Then  $ X $  is respectively  {\it globally split\/}  (=``{\it g-split\/}'')  or  {\it globally strongly split\/}  (=``{\it gs-split\/}'')  if and only if  $ \, \cO(X) \, $  is split or strongly split.
%
\end{proposition}

\medskip

   Proposition \ref{ass-class-scheme}  of course applies also to every (affine) supergroup  $ G \, $,  as a superscheme itself.  In addition, we have the following (more or less well-known) result:

\medskip

\begin{proposition}  \label{ass-class-grscheme}
 Let  $ G $  be an affine supergroup, over a ring\/  $ \bk \, $,  and let  $ \, H := \cO(G) \, $  be its representing (commutative Hopf)  $ \bk $--superalgebra.  Then every closed supersubgroup of\/  $ G $  which is  {\sl classical}  is a closed subgroup of  $ \, G_\zero = G_{\text{\it ev}} \; $.
\end{proposition}

\begin{proof}
 If  $ \mathbf{K} $  is any (closed) supersubgroup of\/  $ G $  which is also classical, then the commutative Hopf superalgebra  $ \cO(\mathbf{K}) $  representing  $ \mathbf{K} $  is a quotient of  $ \, H := \cO(G) \, $,  and it is totally even, i.e.~it is a (commutative) Hopf algebra.  Now, any quotient Hopf superalgebra of  $ H $  which is totally even is actually a quotient of  $ \, \overline{H} := \cO(G_\zero) \, $,  \,\! by construction; applying this to  $ \cO(\mathbf{K}) \, $  we get  $ \, \mathbf{K} \leq G_\zero \, $.
\end{proof}

\bigskip

   In sight of Masuoka's factorization result for commutative Hopf superalgebras (over fields) in  Theorem \ref{Masuoka}  the notion of ``globally strongly split'' (or  ``gs-split'') for supergroups deserves to be made more precise than in the superscheme setting.

\smallskip

   Let  $ G $  be an (affine)  $ \bk $--supergroup,  and let  $ \, H := \cO(G) \, $  be the supercommutative Hopf  $ \bk $--superalgebra  representing it.  The coproduct map  $ \; \cO(G) =: H \! \longrightarrow \! H \otimes H = \cO(G) \otimes \cO(G) = \cO(G \times G) \; $  corresponds to the multiplication map  $ \; G \! \times \! G \! \longrightarrow \! G \; $.  Similarly, the quotient map  $ \; \pi : H := \cO(G) \! \relbar\joinrel\twoheadrightarrow \cO(G_\zero) = \overline{H} \; $  corresponds to a cano\-nical embedding  $ \, j : G_{\text{\it ev}\!} = G_\zero \! \lhook\joinrel\longrightarrow G \, $.  Via this, the  $ \overline{H} $--coaction
  $$  \cO(G) \, =: \, H \, \longrightarrow \, \overline{H} \otimes H \, = \, \cO(G_\zero) \otimes \cO(G) \, = \, \cO(G_\zero \times G)  $$
%
%
corresponds to a left action  $ \; G_\zero \times G \longrightarrow G \; $  of  $ \, G_\zero $  onto  $ G \, $,  given by restriction of the (left) action of  $ G $  onto  $ G $  by left multiplication: so  $ G $  is a left  $ G_\zero $--scheme.  In addition,  $ G $  has a special point, the unit element   --- corresponding to the counit map for  $ \, H := \cO(G) \, $  ---   so that  $ G $  itself can be thought of as a  {\it pointed superscheme}.

\smallskip

   On the other hand, if  $ \mathbb{A}_\bk^{0|d_-} $  is any totally odd affine  $ \bk $--superspace  (with $ \, d_- \in \N \, $),  then the direct product  $ \, G_\zero \times \mathbb{A}_\bk^{0|d_-} \, $  has a left  $ G_\zero $--action,  given by left multiplication in  $ G_\zero \, $:  this makes  $ \, G_\zero \times \mathbb{A}_\bk^{0|d_-} \, $  into a left  $ G_\zero $--scheme.  In addition,  $ G_\zero $  is also a pointed superscheme, whose special point is the unit element   --- corresponding to the counit map for  $ \, \overline{H} = \cO(G_\zero) \, $.  But also  $ \mathbb{A}_\bk^{0|d_-} $  is a pointed superscheme, the special point being the zero of  $ \mathbb{A}_\bk^{0|d_-} $  as a free supermodule   --- this corresponds again to the counit map of the Hopf superalgebra repre\-senting  $ \mathbb{A}_\bk^{0|d_-} \, $.  Therefore,  {\sl the direct product  $ \, \bG_\zero \times \mathbb{A}_\bk^{0|d_-} \, $  is a}  {\it pointed superscheme\/}  as well.

\medskip

   All this lead us to strengthen the notion of ``globally split'' concerning supergroups:

\medskip

\begin{definition}  \label{gl-str-split_sgroup-Def}
 Let  $ G $  be an affine  $ \bk $--supergroup.  Assume that there exists a closed subsupercheme  $ G_\uno $  of  $ G $  such that
 \vskip5pt
   {\it (a)}\,  $ \; 1_{{}_G} \! \in G_\uno \, $,  hence we look at  $ G_\uno $  as a  {\sl pointed superscheme\/};
 \vskip0pt
   {\it (b)}\,  the product in  $ G $  restricts to an isomorphism  $ \; G_\zero \times G_\uno \;{\buildrel \cong \over {\lhook\joinrel\relbar\joinrel\relbar\joinrel\twoheadrightarrow}}\; G \; $  of pointed left  $ G_\zero $--superschemes;
 \vskip3pt
   {\it (c)}\,  $ G_\uno $  is (isomorphic to) a totally odd affine superscheme  $ \mathbb{A}_\bk^{0|d_-} $,  as a pointed superscheme.
 \vskip5pt
\noindent
 When all this holds, we say that  $ G $  is  {\it globally strongly split},  or in short that it is  {\it gs-split}.
 \vskip3pt
   We shall then denote by  $ \gsssgrps_\bk $  and  $ \gssfsgrps_\bk \, $,  respectively, the full subcategories of  $ \sgrps_\bk $  and  $ \fsgrps_\bk \, $,  respectively, whose objects are all supergroups and all fine supergroups, respectively, over  $ \bk $  which in addition are globally strongly split.
 \hfill   $ \diamondsuit $
\end{definition}

\smallskip

\begin{remark}  \label{Lie-repr_->_(gs-split_->_fine)}
 Let  $ G $  be an affine  $ \bk $--supergroup  for which  $ \Lie(G) $  is quasi-represen\-table.  Then if  $ G $  is globally strongly split, it is also clearly  {\sl fine}   --- in the sense of  Definition \ref{def_fine-sgroups}.
\end{remark}

\medskip

   The main facts to take into account at this stage are the following:

\medskip

\begin{theorem}  \label{gl-str-split_sgroup-Th}
 Let  $ G $  be an affine supergroup, defined over a ring\/  $ \bk \, $,  and let  $ \, H := \cO(G) \, $  be its representing (commutative Hopf)  $ \bk $--superalgebra.  Then  $ G $  is globally strongly split if and only if the Hopf superalgebra  $ \cO(G) $  is strongly split.  In particular, if\/  $ \bk $  is a  {\sl field}  whose characteristic is not 2, then  $ G $  is globally strongly split.
\end{theorem}

\begin{proof}
 It is clear that the very last part of the statement is a sheer consequence of Masuoka's  Theorem \ref{Masuoka}.  We just need to prove the rest.

\vskip7pt

   In one direction, the proof is obvious.  Indeed, assume that  $ G $  is  globally strongly split, so  $ \, G = G_\zero \cdot G_\uno \cong G_\zero \times G_\uno \cong G_\zero \times \mathbb{A}_\bk^{0|d_-} \, $  (see  Definition \ref{gl-str-split_sgroup-Def}):  then we have isomorphisms
  $$  \cO(G)  \; \cong \;  \cO\big( G_\zero \times G_\uno \big)  \; \cong \;  \cO(G_\zero) \otimes_\bk \cO(G_\uno)  \; \cong \;  \cO(G_\zero) \otimes_\bk \cO\big(\mathbb{A}_\bk^{0|d_-}\big)  $$
of  $ \bk $--algebras.  Moreover, the fact that  $ \, G = G_\zero \cdot G_\uno \cong G_\zero \times G_\uno \, $  as  {\sl pointed\/}  superschemes implies that the resulting isomorphism  $ \; \cO(G) \, \cong \, \cO(G_\zero) \otimes_\bk \cO\big(\mathbb{A}_\bk^{0|d_-}\big) \; $  is also  {\sl counital}.   Finally, as the isomorphism  $ \, G_\zero \times \mathbb{A}_\bk^{0|d_-} \cong G_\zero \cdot G_\uno = G \, $  is one of  $ G_\zero $--superspaces,  one sees in addition that the isomorphism  $ \; \cO(G) \, \cong \, \cO(G_\zero) \otimes_\bk \cO\big(\mathbb{A}_\bk^{0|d_-}\big) \; $  is also left  $ \cO(G_\zero) $--{\sl coinvariant}.
 \vskip5pt
   The converse step almost entirely follows from definitions (along with  Remark \ref{barH-coaction_W})  and Masuoka's result  (Theorem \ref{Masuoka}).  Indeed, assume that the Hopf superalgebra  $ \, H := \cO(G) \, $  is strongly split, i.e.~there exists an isomorphism  $ \; \cO(G) =: H \,{\buildrel \cong \over {\lhook\joinrel\relbar\joinrel\relbar\joinrel\twoheadrightarrow}}\, \overline{H} \otimes_\bk \bigwedge \! W^H \; $  of counital  $ \overline{H} $--comodule  superalgebras, with  $ \, \overline{H} = \overline{\cO(G)} = \cO(G_0) \, $:  using this isomorphism we can identify both  $ \, \overline{H} = \overline{\cO(G)} = \cO(G_0) \, $  and   $ \, \bigwedge \! W^H \, $  with subalgebras of  $ \, \cO(G) = H \, $  whose product is all of  $ H $  itself.  Taking superspectra, this yields an isomorphism  $ \; G_\zero \times G_\uno \, \cong \, G \; $   --- as pointed superschemes with left  $ G_\zero $--action  ---   for some closed subsuperscheme  $ G_\uno $  in  $ G $  such that  $ \, \cO(G_\uno) = \bigwedge \! W^H \, $.  To put it in down-to-earth terms, we look pointwise: if  $ \, A \! \in \! \salg_\bk \, $  one has
  $$  \displaylines{
   G(A) \, := \, \Hom_{\salg_\bk}\big( H \, , A \,\big) \;\; ,  \qquad  G_\zero(A) \, := \, \Hom_{\salg_\bk}\big(\, \overline{H} \, , A \,\big)  \cr
   G_\uno(A) \, := \, \Hom_{\salg_\bk}\big(\, {\textstyle \bigwedge} W^H , A \,\big)  }  $$
and the isomorphism then is given by  ($ m_{{}_A} \! $  being the multiplication in  $ A \, $)
  $$  G_\zero(A) \times G_\uno(A) \longrightarrow G(A) \; ,  \qquad  (\varphi_0\,,\,\varphi_1) \, \mapsto \, m_{{}_A} \! \circ (\varphi_0 \otimes \varphi_1)  $$
in particular,  $ G_\zero(A) $  and  $ G_\uno(A) $  as subsets of  $ G(A) $  are characterized as
  $$  \displaylines{
   G_\zero(A)  \; := \;  \big\{\, \phi \in G(A) \;\big|\; \phi\big( 1 \otimes \omega \big) = \epsilon_{{}_{\scriptscriptstyle \bigwedge \! W^H}}(\omega) \;\; \forall \; \omega \in {\textstyle \bigwedge} W^H \;\big\}  \cr
   G_\uno(A)  \; := \;  \big\{\, \phi \in G(A) \;\big|\; \phi\big(\, \overline{h} \otimes 1 \big) = \epsilon_{{}_{\scriptscriptstyle \overline{H}}}\big(\,\overline{h}\,\big) \,\; \forall \; \overline{h} \in \overline{H} \;\big\}  } $$
   \indent   Thus, the only non-trivial point which is left out is that this isomorphism actually is realized through restriction of the product in  $ G \, $:  we check it now pointwise.
 \vskip3pt
   The standard embedding  $ \; \sigma_{\!{}_H} : \cO(G_\zero) = \overline{\cO(G)} = \overline{H} \lhook\joinrel\relbar\joinrel\relbar\joinrel\rightarrow \overline{H} \otimes_\bk \bigwedge \! W^H \cong H =: \cO(G) \; $  induces a natural supergroup morphism  $ \, G \relbar\joinrel\twoheadrightarrow G_\zero \, $,  which is a retraction of the embedding  $ \, G_\zero \lhook\joinrel\relbar\joinrel\rightarrow G \, $  because  $ \sigma_{\!{}_H} $  itself is a section of the canonical projection  $ \, \pi_{\!{}_H} \! : H \! \relbar\joinrel\twoheadrightarrow \overline{H} \, $.  Putting it in down-to-earth terms, for  $ \, A \in \salg_\bk \, $  the above mentioned retraction reads
  $$  G(A) := \Hom_{\salg_\bk}\big( \cO(G) , A \big) \! \twoheadrightarrow G_\zero(A) := \Hom_{\salg_\bk}\big( \cO(G_\zero) , A \big) \; ,  \;\  g \mapsto g_0 := g \circ \sigma  $$
for all  $ \, g \in G(A) \, $.  In other words, if  $ \, g \in G(A) \, $  then  $ \, g_0 := g \circ \sigma \, $  can be seen as the restriction of  $ g $  (which is a morphism from  $ \, H := \cO(G) \, $  to  $ A $)  to  $ \, \overline{H} = \overline{\cO(G)} \, $  embedded in  $ \, H = \cO(G) \, $  via  $ \sigma \, $.  Notice also that  $ \, g_0 := g \circ \sigma \, $  belongs to  $ G_\zero(A) \, $,  and when we embed the latter (canonically) into  $ G(A) $  the image of  $ g_0 $  is  $ \, g_0 \circ \pi_{{}_{\!H}} \, $:  in other words, when thinking of  $ g_0 $  as an element of  $ G(A) $  we realize that we are actually taking  $ \, \hat{g}_0 := g_0 \circ \pi_{{}_{\!H}} \, $.
                                                          \par
   Now for each  $ \, g \in G(A) \, $  consider  $ \, g_0 := g \circ \sigma \in G_\zero(A) \, $  and  $ \, \hat{g}_0 := g_0 \circ \pi_{{}_{\!H}} \in G(A) \, $  as above; then define  $ \; g_1 := \hat{g}_0^{\,-1} \! \cdot g \; \big(\! \in G(A) \big) \, $.  By construction we have  $ \, g = \hat{g}_0 \cdot g_1 \, $   --- a product inside  $ G(A) $  ---   so we are only left to prove that actually  $ \; g_1 \, \in \, G_\uno(A) = \Hom_{\salg_\bk}\big( \bigwedge \! W^H , A \,\big) \; $.  To this end, we recall that the product in  $ \, G(A) = \Hom_{\salg_\bk}\big( H , A \big) \, $  is given by  {\sl convolution}
  $$  \quad   g' \cdot g'' \, := \,  m_{{}_{\!A}} \! \circ \big(\, g' \otimes g'' \big) \circ \Delta_{{}_H}   \qquad  \forall \;\; g' , g'' \in G(A) = \Hom_{\salg_\bk}\big( H , A \big)   \eqno (3.1)  $$
   By the characterization of  $ G_\uno(A) $  given above we have that  $ \, g_1 \in G_\uno(A) \, $  if and only if
  $$  g_1\big(\, \overline{h} \otimes 1 \big)  \, = \,  1_{{}_{G(A)}}\big(\,\overline{h} \otimes 1\big)  \qquad  \text{for all}  \quad  \overline{h} \in \overline{H}   \eqno (3.2)  $$
so this is our goal.  Definitions along with (3.1) give
  $$  g_1  \, := \,  \hat{g}_0^{\,-1} \! \cdot g  \, = \,  m_{{}_{\!A}} \! \circ \big(\, \hat{g}_0^{\,-1} \! \otimes g \big) \circ \Delta_{{}_H}  = \,  m_{{}_{\!A}} \! \circ \big(\, g_0^{-1} \! \otimes g \big) \circ \big(\, \pi_{{}_{\!H}} \otimes {\text{\sl id}}_{{}_H} \big) \circ \Delta_{{}_H}   \eqno (3.3)  $$
Recall that, by assumption, the isomorphism  $ \; H \,{\buildrel \cong \over {\lhook\joinrel\relbar\joinrel\relbar\joinrel\twoheadrightarrow}}\, \overline{H} \otimes_\bk \bigwedge \! W^H \; $  of counital superalgebras is also left  $ \overline{H} $--covariant:  this means that
  $$  \big( \big(\, \pi_{{}_{\!H}} \otimes {\text{\sl id}}_{{}_H} \big) \circ \Delta_{{}_H} \big) \big(\, \overline{h} \otimes 1 \big)  \, = \,  \Big( \Delta_{{}_{\overline{H}}} \otimes {\text{\sl id}}_{{}_{\scriptscriptstyle \bigwedge \! W^H}} \Big)\big(\, \overline{h} \otimes 1 \big)  \, = \,  \overline{h}_{(1)} \otimes \overline{h}_{(2)} \otimes 1  \qquad  \forall \;\; \overline{h} \in \overline{H}  $$
hence (3.3) and the identity  $ \, \overline{h}_{(2)} \otimes 1 = \sigma\big(\, \overline{h}_{(2)} \big) \, $  together give
  $$  \displaylines{
   g_1\big(\, \overline{h} \otimes 1 \big)  \; = \;  \big( m_{{}_{\!A}} \! \circ \big(\, g_0^{-1} \! \otimes g \big) \circ \big(\, \pi_{{}_{\!H}} \otimes {\text{\sl id}}_{{}_{\!H}} \big) \circ \Delta_{{}_H} \big)\big(\, \overline{h} \otimes 1 \big)  \; =   \hfill  \cr
   = \;  g_0^{-1}\big(\,\overline{h}_{(1)}\big) \; g\big(\, \overline{h}_{(2)} \otimes 1 \big)  \; = \;  g_0^{-1}\big(\overline{h}_{(1)}\!\big) \, (g \circ \sigma)\big( \overline{h}_{(2)} \!\big)  \; = \;  g_0^{-1}\big(\overline{h}_{(1)}\!\big) \; g_0\big(\overline{h}_{(2)}\!\big)  \; =  \cr
   \hfill   = \;  \big( g_0^{-1} \! \cdot g_0 \big)\big(\,\overline{h}\,\big)  \; = \;  1_{{}_{G_\zero(A)}}\!\big(\,\overline{h}\,\big)  \; = \;  1_{{}_{G(A)}}\!\big(\, \overline{h} \otimes \! 1 \big)  }  $$
for all  $ \, \overline{h} \in \overline{H} \, $,  so that (3.2) is proved.
\end{proof}

\smallskip

   The above characterization of gs-split supergroups yields an interesting consequence:

\smallskip

\begin{corollary}
 Let  $ G $  be a globally strongly split supergroup.  Then (with notation of  Definition \ref{gl-str-split_sgroup-Def})  $ G_\uno $  is stable by the adjoint  $ G_\zero $--action.
\end{corollary}

\begin{proof}
 Applying  Remark \ref{barH-coaction_W}  to  $ \, H := \cO(G) \, $  we find that  $ \, \bigwedge W^H = \cO(G_\uno) \, $  is stable by the right coadjoint  $ \cO(G_\zero) $--coaction.  But at the superscheme level this implies exactly that  $ G_\uno $  is stable by the adjoint  $ G_\zero $--action,  q.e.d.
\end{proof}

\smallskip

\begin{remark}
 To be complete, we mention that in Boseck's approach (see  \cite{bos1, bos3})  each  {\sl affine algebraic supergroup\/}  is assumed to be ``globally strongly split'' (in our sense)  {\sl by definition}.
\end{remark}

\smallskip

\begin{free text}  \label{consist-splitt.s}
 {\bf Consistent splittings for embeddings of gs-split supergroups.}  If we consider a (closed) supersubgroup within a supergroup, and both are gs-split, we can ask whether there exist ``splittings'' of both supergroups which are ``consistent'' with each other, by which we mean that they are ``compatible'', in some natural sense, with the embedding of the first supergroup inside the second one. We shall now make this rough idea more precise and find some significant results.
 \vskip7pt
   Let us start with an affine, fine supergroup  $ G $  over  $ \bk \, $,  with  $ \, \fg = \text{\sl Lie}(G) \, $.  Set
%
%
 $ \; A \, := \, \cO(G) \; $,  $ \; \overline{A} \, := \, A \Big/ \big( A_\uno \big) \; $,  $ \; W^A \, := \, A_\uno \Big/ A_\zero^+ A_\uno \; $.
 Recall that  $ \, \overline{A} = \overline{\cO(G)} = \cO\big( G_{\text{\it ev}} \big) \, $,  the commutative Hopf algebra which represents the affine group  $ G_{\text{\it ev}} $  associated with  $ G \, $,  and  $ \, {\big( W^A \big)}^* = \fg_\uno \; $.  By assumption  $ G $  is fine, so  $ \fg_\uno $  is  $ \bk $--finite  free.
%
%
 Recall also that  $ \fg $  is endowed with a  $ 2 $--operation
%
%
 and
the universal enveloping (super)algebra  $ U(\fg) $  involves the relations  $ \, v^2 = v^{\langle 2 \rangle} \, $,  for  $ \, v \in \fg_\uno \, $.  Just as in  \cite{mas-shi},  Lemma 4.24, there exists a canonical Hopf pairing  $ \; \langle\ \,,\ \rangle : U(\fg) \times A \longrightarrow \bk \; $  which gives rise to the superalgebra map
%
%
 $ \; \kappa : A \relbar\joinrel\longrightarrow {U(\fg)}^* \; $  defined by  $ \; \kappa(a) \, := \, \langle\ \,, a \rangle \; $.
                                                                \par
   Choose a totally ordered  $ \bk $--free basis  $ X $  of  $ \fg_\uno \, $,  and define a unit-preserving super-coalgebra map  $ \; \iota_{{}_{\scriptstyle X}} : \wedge \fg_\uno \longrightarrow U(\fg) \; $  by
%
%
 $ \; \iota_{{}_{\scriptstyle X}}(x_1 \wedge \cdots \wedge x_n) \, := \, x_1 \cdots x_n \; $
for  $ \, n \geq 0 \, $  and  $ \, x_1 < \cdots < x_n \, $  in  $ X \, $.  Define also  $ \; \rho_{{}_{\scriptstyle X}} : A \longrightarrow \wedge W^A \; $  to be the composition
%
%
 $ \; \rho_{{}_{\scriptstyle X}} \, : \, A \,{\buildrel \kappa \over {\relbar\joinrel\longrightarrow}}\, {U(\fg)}^* \,{\buildrel {\iota_{{}_X}^{\,*}} \over {\relbar\joinrel\longrightarrow}}\, {\big(\! \wedge \fg_\uno \,\big)}^* \,{\buildrel \cong \over {\relbar\joinrel\longrightarrow}}\, \wedge W^A \; $
where the last arrow denotes the canonical isomorphism; this  $ \rho_{{}_{\scriptstyle X}} $  is a counit-preserving superalgebra map: see  \cite{ms1},  page 301.  Now, assume that the Hopf superalgebra  $ A $  is split (or equivalently  $ G $  is gs-split): thus there exists a counit-preserving isomorphism  $ \; A \,{\buildrel \cong \over {\relbar\joinrel\longrightarrow}}\, \overline{A} \otimes_\bk \wedge W^A \; $  of left  $ A $--comodule  superalgebras.  The following result proves that one can choose a particular such splitting:
\end{free text}

\vskip8pt

\begin{lemma}  \label{split-isom-alg}
 {\sl (Masuoka)}  The map
%
%
 $ \; \psi_{{}_{\scriptstyle X}} : \, A \,{\relbar\joinrel\relbar\joinrel\longrightarrow}\, \overline{A} \otimes_\bk \wedge W^A  \; $,  $ \; a \mapsto \psi_{{}_{\scriptstyle X}}(a) := \overline{a_{(1)}} \otimes \rho_{{}_{\scriptstyle X}}\big(a_{(2)}\big) \; $,
where  $ \, a \mapsto \overline{a} \, $  denotes the natural projection  $ \, A \relbar\joinrel\twoheadrightarrow \overline{A} \, $,  is bijective.  In fact, it is a counit-preserving isomorphism of left  $ \overline{A} $--comodule  superalgebras.
\end{lemma}

\begin{proof}
 The claim follows from  \cite{mas-shi}, Lemma 4.27;  note in particular that the cited Lemma actually holds over an arbitrary ring  $ \bk \, $  (moreover,  $ \overline{A} $  may not be finitely generated, as long as  $ W^A $  is  $ \bk $--finite  free: see the proof of  Theorem A.10 in  \cite{mas-shi}).  Since the map  $ \rho_{{}_{\scriptstyle X}} \, $,  composed with the natural projection  $ \; \wedge W^A \relbar\joinrel\relbar\joinrel\twoheadrightarrow\, W^A \, $,  coincides with the natural composite
\vskip-7pt
  $$  A \, = \, \bk \oplus A^+ \relbar\joinrel\relbar\joinrel\twoheadrightarrow\, A^+ \Big/ {(A^+)}^2 \relbar\joinrel\relbar\joinrel\twoheadrightarrow\, {\Big( A^+ \Big/ {(A^+)}^2 \Big)}_\uno \, = \, W^A  $$
\vskip-2pt
\noindent
 we conclude that  $ \rho_{{}_{\scriptstyle X}} $  satisfies the assumption required by the cited lemma.
\end{proof}

\vskip6pt

   We can now state and prove the promised result about the existence of ``consistent'' splittings for a closed embedding between gs-split supergroups:

\vskip10pt

\begin{theorem}  \label{cons-splittings}
 ({\sl Gavarini and Masuoka})\,  Let  $ H $  and  $ K $  be (affine, fine) supergroups over  $ \bk \, $,  with  $ H $  being a closed subsupergroup of  $ \, K \, $.  Setting  $ \, \fh := \Lie\,(H) \, $,  $ \, \fk := \Lie\,(K) \, $,  assume that:
 \vskip2pt
%
%
 \quad   (a)\,  the quotient  $ \bk $--supermodule  $ \, \mathfrak{k}_\uno \big/ \mathfrak{h}_\uno $  is free;
 \vskip1pt
 \quad   (b)\,  $ H $  and  $ K $  are globally strongly split.
 \vskip3pt
   Then, for a given splitting  $ \, H = H_\zero \times H_\uno \, $  of  $ \; H \, $,  there exists a similar splitting  $ \, K = K_\zero \times K_\uno \, $  of  $ \, K $  which is  {\sl consistent}  with that of  $ \, H \, $,  that is one has  $ \, H_\zero \subseteq K_\zero \, $,  $ \, H_\uno \subseteq K_\uno \, $  and the diagram
\vskip-7pt
  $$  \xymatrix{
   \; H_\zero \times_{\phantom{.}} \! H_\uno  \ar@{_{(}->}[d] \ar[rr]^{\ \hskip11pt \cong}_{\ \hskip13pt \mu_{{}_H}}  &  &  \hskip5pt H_{\phantom{.}}  \ar@{^{(}->}[d]  \\
   \;  K_\zero \times K_\uno  \ar[rr]^{\ \hskip13pt \cong}_{\ \hskip15pt \mu_{{}_K}}  &  &  \hskip3pt  K  }  $$
\vskip-3pt
\noindent
 ($ \, \mu_{{}_H} \! $  being an isomorphism of pointed affine  $ H_\zero $--superschemes,  and similarly  $ \mu_{{}_K} $) is commutative.
\end{theorem}

\begin{proof}
 Set  $ \, A := \cO(K) \, $  and  $ \, B := \cO(H) \, $.  By assumption  $ \fk_\uno $  and  $ \fh_\uno $  are  $ \bk $--finite  free, so the same holds true for their linear dual  $ \, A_\uno \big/ A_\zero^+ A_\uno = \fk_\uno^{\,*} \, $  and  $ \, B_\uno \big/ B_\zero^+ B_\uno = \fh_\uno^{\,*} \, $.  Moreover,  {\it (a)\/}  ensures that there exists a free  $ \bk $--subsupermodule  $ \mathfrak{q} $  of  $ \, \mathfrak{k}_\uno $  such that  $ \, \mathfrak{k}_\uno = \mathfrak{h}_\uno \oplus \mathfrak{q} \, $.  Finally, let  $ K_\uno $  and  $ H_\uno $  be the pointed (affine algebraic) superschemes represented by  $ \, \wedge \big( \fg_\uno^{\,*} \big) \, $  and  $ \, \wedge \big( \fh_\uno^{\,*} \big) \, $  respectively.
%
%
 Since  $ H $  is a closed subsupergroup of  $ K \, $,  it follows that  $ \, H_\zero = H_{\text{\it ev}} \, $  can be seen as a closed (classical) subgroup of  $  \, K_\zero = K_{\text{\it ev}} \, $;  similarly,  $ H_\uno $  can be seen as a closed pointed subsuperscheme of  $ K_\uno \, $.
                                                             \par
   If  $ X_\fk $  is an ordered basis of  $ \fk_\uno $  and  $ X_\fh $  is one of  $ \fh_\uno $  we can construct splitting isomorphisms  $ \psi_{{}_{\scriptstyle{X_\fk}}} $  and  $ \psi_{{}_{\scriptstyle{X_\fh}}} $  as in  Lemma \ref{split-isom-alg} for  $ \, A' := \cO(K) \, $  and  $ \, A'' := \cO(H) \, $  respectively.  Then moving backwards from Hopf superalgebras to supergroups we find splitting maps  $ \, \mu_{{}_K} : K_\zero \times K_\uno \,{\buildrel \cong \over {\relbar\joinrel\longrightarrow}}\, K \, $  and  $ \, \mu_{{}_H} : H_\zero \times H_\uno \,{\buildrel \cong \over {\relbar\joinrel\longrightarrow}}\, H \, $  as in  \S \ref{consist-splitt.s}  for  $ H $  and  $ K \, $,  with  $ \, \cO\big(_{\,}\mu_{{}_K}\big) = \psi_{{}_{\scriptstyle{X_\fk}}} \, $,  $ \, \cO\big(_{\,}\mu_{{}_H}\big) = \psi_{{}_{\scriptstyle{X_\fh}}} \, $.
                                                             \par
  Now, by assumption the natural injection  $ \, \fh_\uno \lhook\joinrel\longrightarrow \fk_\uno \, $  is  $ \bk $--linearly  split, as  $ \, \fk_\uno = \fh_\uno \oplus \mathfrak{q} \, $;  moreover,  $ \mathfrak{q} $ is  $ \bk $--free:  therefore we can choose  $ X_\fh $  and  $ X_\fk $  as above so that the former is an ordered subset of the latter.  For such a choice, the diagram
  $$  \xymatrix{
   \, \wedge\,\fh_\uno  \ar@{_{(}->}[d] \ar[rr]^{\iota_{{}_{X_\fh}} \hskip1pt}  &  &  \hskip5pt U(\fh)_{\phantom{|}}  \ar@{^{(}->}[d]  \\
   \,  \wedge\,\fk_\uno  \ar[rr]_{\iota_{{}_{X_\fk}} \hskip1pt}  &  &  \hskip3pt U(\fk)  }  $$
is commutative, which in turn implies that the diagram
  $$  \xymatrix{
   \overline{\cO(H)} \otimes_\bk \wedge \big( \fh_\uno^{\,*} \big)  \ar@{<<-}[d] \,  &  &  \ar[ll]_{\hskip31pt \psi_{{}_{X_\fh}}}  \hskip3pt \cO(H)_{\phantom{|}}  \ar@{<<-}[d]  \\
   \overline{\cO(K)} \otimes_\bk \wedge \big( \fk_\uno^{\,*} \big) \,  &  &  \ar[ll]^{\hskip31pt \psi_{{}_{X_\fk}}} \hskip3pt \cO(^{\phantom{\textstyle |}}\hskip-4pt K)  }  $$
is commutative too.  When turning back from superalgebras to superschemes, this implies that the diagram in the statement is commutative as well.
\end{proof}

\vskip11pt

\begin{remark}\,  The assumption  {\it (a)\/}  in  Theorem \ref{cons-splittings}  obviously holds true when  $ \bk $  is a PID, or   --- more in general ---   whenever every finitely generated projective  $ \bk $--module  is free.
\end{remark}

\medskip

 \subsection{Splittings on  $ A $--points}  \label{split_on_points} {\ }

\vskip-5pt

   {\ } \;\;   In this subsection we dwell upon some special splittings of supergroups which arise when we take their  $ A $--points  for some special superalgebras  $ A \, $,  i.e.~when we restrict them   --- as functors ---   on special subcategories of  $ \salg_\bk \, $.

\smallskip

\begin{definition}  \label{def-Ker(pi)}
 Let  $ \, G : \salg_\bk \longrightarrow \sgrps \, $  be a supergroup  $ \bk $--functor.  Then there exists a unique, well defined normal subgroup  $ \bk $--functor of  $ G $,  denoted  $ \, {\Ker(\pi)}_G \, $  and given on objects by  $ \, {\Ker(\pi)}_G(A) := \Ker\big(G(\pi_{\!{}_A})\big) \, $  for every  $ \, A \in \wkspsalg_\bk \, $.
 \hfill   $ \diamondsuit $
\end{definition}

\smallskip

   In general, in the study of a supergroup functor  $ G $  the normal subgroup functor  $ {\Ker(\pi)}_G $  is not of great use.  But restricting to  {\sl weakly split superalgebras},  next result shows that it splits into a semidirect product, in which  $ {\Ker(\pi)}_G $  is the normal factor.

\vskip9pt

\begin{proposition}  \label{split_G-wksplsalg}
 Let  $ G $  be a  $ \bk $--supergroup  functor and  $ \, A \in \wkspsalg_\bk \, $.
 \vskip1pt
   {\it (a)}  The group  $ \, G(A) \, $  splits into a semidirect product
 $ \; G(A) \, = \, \overline{G}(A) \ltimes {\Ker(\pi)}_G(A) \; $
 with (see  Definition \ref{nat-funcs})  $ \, \overline{G}(A) := G\big(\,\overline{A}\,\big) \, $  and  $ \, {\Ker(\pi)}_G(A) := \Ker\big(G(\pi_{\!{}_A})\big) \, $.  Therefore, denoting by  $ \dot{F} $  the restriction to  $ \wkspsalg_\bk $  of any superfunctor  $ F $,  we have that  $ \, \dot{G} : \wkspsalg_\bk \!\longrightarrow \grps \, $  splits into a semidirect product
 $ \,\; \dot{G} \, = \, \dot{\overline{G}} \ltimes \! \dot{\Ker(\pi)}_G \;\, $.
 \vskip1pt
   {\it (b)}  The group  $ \, G_\zero(A) \, $  splits into a semidirect product
 $ \; G_\zero(A) \, = \, \overline{G}(A) \ltimes {\Ker(\pi)}_{G_\zero}(A) \; $
 with (see  Definition \ref{nat-funcs})  $ \, {\Ker(\pi)}_{G,\zero}(A) := {\Ker(\pi)}_G(A) \,\cap\, G_\zero(A) \, $.  Thus the  $ \bk $--supergroup  functor  $ \, \dot{G}_\zero : \wkspsalg_\bk \!\longrightarrow \grps \, $  splits into a semidirect product
 $ \,\; \dot{G}_\zero \, = \, \dot{\overline{G}} \ltimes \! \dot{\Ker(\pi)}_{G_\zero} \;\, $.
\end{proposition}

\begin{proof}
 As  $ \, A \in \wkspsalg_\bk \, $,  we have  a projection  $ \, \pi_{\!{}_A} \! : A \relbar\joinrel\twoheadrightarrow \overline{A} \, $  with section  $ \, \sigma_{\!{}_A} \! : \overline{A} \lhook\joinrel\relbar\joinrel\rightarrow A \, $  within  $ \salg_\bk \, $.  Applying  $ G $  we get that  $ \, G(\sigma_{\!{}_A}) : G\big(\,\overline{A}\,\big) \! \relbar\joinrel\rightarrow G(A) \, $  is a section of  $ \, G(\pi_{\!{}_A}) : G(A) \relbar\joinrel\rightarrow G\big(\,\overline{A}\,\big) \, $,  so  $ G(\sigma_{\!{}_A}) $  is a monomorphism and  $ G(\pi_{\!{}_A}) $  an epimorphism.  In turn, this yields then a semidirect product factorization of  $ G(A) \, $,  namely  $ \; G(A) = \text{\sl Im}\big(G(\sigma_{\!{}_A})\big) \ltimes \Ker\big(G(\pi_{\!{}_A})\big) \; $.  Looking at definitions one finds  $ \,\text{\sl Im}\big(G(\sigma_{\!{}_A})\big) \cong G\big(\,\overline{A}\,\big) \, $  whence claim  {\it (a)\/}  follows.
                                                                \par
   As to claim  {\it (b)},  one can repeat the previous argument: just replace  $ G $  with  $ \, G_\zero $   --- such that  $ \, A \mapsto G(A_\zero) \, $  ---   and  $ {\Ker(\pi)}_G $  with  $ {\Ker(\pi)}_{G_\zero} $  wherever they occur.
\end{proof}

\vskip2pt

   The previous result reads better when applied to  {\sl split\/}  superalgebras.

\vskip9pt

\begin{proposition}  \label{split_G-stsplsalg}
 Let  $ G $  be a  $ \bk $--supergroup  functor  and  $ \, A \in \splsalg_\bk \, $.
 \vskip1pt
   {\it (a)}  The group  $ \, G(A) \, $  splits into a semidirect product
 $ \; G(A) \, = \, \overline{G}(A) \ltimes G_\uno^{(1)}(A) \; $
 with (see  Definition \ref{nat-funcs})  $ \, \overline{G}(A) := G\big(\,\overline{A}\,\big) \, $  and  $ \, G_\uno^{(1)}(A) := G\big(A_\uno^{(1)}\big) \, $.  Therefore, denoting by  $ \hat{F} $  the restriction to  $ \splsalg_\bk $  of any superfunctor  $ F $,  we have that  $ \, \hat{G} : \splsalg_\bk \!\longrightarrow \grps \, $  splits into a semidirect product
 $ \,\; \hat{G} \, = \, \hat{\overline{G}} \ltimes \! {\hat{G}}_\uno^{(1)} \;\, $.
 \vskip1pt
   {\it (b)}  The group  $ \, G_\zero(A) \, $  splits into a semidirect product
 $ \; G_\zero(A) \, = \, \overline{G}(A) \ltimes G_\uno^{(2)}(A) \; $
 with (see  Definition \ref{nat-funcs})  $ \, G_\uno^{(2)}(A) := G\big(A_\uno^{(2)}\big) \, $.  Thus, with notation as in (a),  $ \, \hat{G}_\zero : \splsalg_\bk \!\longrightarrow \grps \, $  splits into a semidirect product
 $ \,\; \hat{G}_\zero \, = \, \hat{\overline{G}} \ltimes \! {\hat{G}}_\uno^{(2)} \;\, $.
\end{proposition}

\begin{proof}
 From the natural embedding  $ \, A_\uno^{(1)} \lhook\joinrel\longrightarrow \overline{A} \otimes_\bk A_\uno^{(1)} = A \, $  in  $ \wkspsalg_\bk $  we get a group morphism  $ \, G_\uno^{(1)}(A) := G\big(A_\uno^{(1)}\big) \longrightarrow G(A) \, $.  Directly from definitions, one finds that the latter too is an embedding and moreover  $ \, {\Ker(\pi)}_G(A) = G_\uno^{(1)}(A) \, $  for  $ \, A \in \splsalg_\bk \, $;  then claims  {\it (a)\/}  and  {\it (b)\/}  follow from this and  Proposition \ref{split_G-wksplsalg}  right above.
\end{proof}

\vskip2pt

   Next result still improves the previous one when we restrict to  {\sl central extension\/}  algebras:

\vskip9pt

\begin{proposition}  \label{split_G-cexsalg}
 Let notation be as in  Proposition \ref{split_G-wksplsalg}.  Let  $ \, G \in \sgrps_\bk \, $  be a  $ \bk $--supergroup  and  $ \, A \in \cexsalg_\bk \, $.  Then  $ \, G(A) \, $  splits into a semidirect product  $ \; G(A) \, = \, \overline{G}(A) \ltimes G_\uno^{(1)}(A) \; $  with  $ \, \overline{G}(A) = G_\zero(A) \, $.  Thus, letting  $ \check{F} $  be the restriction to  $ \cexsalg_\bk $  of any superfunctor  $ F \, $,  we have that  $ \, \check{G} : \cexsalg_\bk \!\longrightarrow \grps \, $  splits into a semidirect product  $ \,\; \check{G} \, = \, \overline{G} \ltimes \! {\check{G}}_\uno^{(1)} \, = \, \check{G}_\zero \ltimes \! {\check{G}}_\uno^{(1)} \;\, $.
\end{proposition}

\begin{proof}
 As  $ \, A \in \cexsalg_\bk \, $  we have  $ \, \overline{A} = A_\zero \, $,  whence everything follows.
\end{proof}

\medskip

 \subsection{Examples and applications}  \label{exs-appls} {\ }

\smallskip

   We provide some examples to illustrate the previously explained ideas.  Besides their intrinsic interest, these will also be useful in the sequel.

\medskip

\begin{free text}
 {\bf Supergroups on ``super-numbers'' as (classical) groups of ``super-points''.}  Let  $ \, \mathcal{A} \in \alg_\bk \, $  be a commutative  $ \bk $--algebra.  Like in  Example \ref{ex-split}{\it (c)},  we consider the associated central extension  $ \, \mathcal{A}[\eta] = \mathcal{A} \oplus \mathcal{A}\,\eta \in \cexsalg_\bk\, $.  Loosely inspired by the similar construction of ``dual numbers''   --- either in the non-super or the super framework ---   we call its elements ``super  $ \mathcal{A} $--numbers'',  thinking at those in  $ \mathcal{A} $  itself as  ``{\sl even\/}  super-numbers'' and those in  $ \, \mathcal{A}\,\eta \, $  as  ``{\sl odd\/}  super-numbers''.  Now, as  $ \, \mathcal{A}[\eta] \in \cexsalg_\bk \, $  we have  $ \, {\big( \mathcal{A}[\eta] \big)}_\zero = \mathcal{A} \, $,  $ \, {\big( \mathcal{A}[\eta] \big)}_\uno^{(1)} = \bk \oplus \mathcal{A}\,\eta \, $.  Then for any  $ \, G \in \sgrps_\bk \, $  Proposition \ref{split_G-cexsalg}  above yields
  $$  G\big(\mathcal{A}[\eta]\big)  \,\; = \;\,  G_\zero(\mathcal{A}) \, \ltimes G_\uno^{(1)\!}\big(\, \bk \oplus \mathcal{A} \, \eta \big)   \eqno (3.4)  $$
   Now, the left hand factor  $ G_\zero(\mathcal{A}) $  of  $ G\big(\mathcal{A}[\eta]\big) $  is the group of  $ \mathcal{A} $--points  of the classical (= non-super) affine group-scheme  $ G_\zero \, $,  hence its elements are nothing but classical (= non-super) points of a classical group-scheme.  For this reason, we suggest to think of these as being  ``{\sl even\/}  $ \mathcal{A} $--superpoints'' of  $ G \, $,  and similarly to think of the elements of the right hand factor  $ G_\uno^{(1)\!}\big(\, \bk \oplus \mathcal{A} \, \eta \big) $  of  $ G\big(\mathcal{A}[\eta]\big) $  as being  ``{\sl odd\/}  $ \mathcal{A} $--superpoints''  of  $ G \, $.
 \vskip4pt
   Now assume in addition that  $ G $  is strongly split, say  $ \, G = G_{\text{\it ev}} \times G_{\text{\it odd}} = G_\zero \times G_{\text{\it odd}} \, $  with  $ \, G_{\text{\it odd}} \cong \mathbb{A}_\bk^{\text{\it odd}} \, $  (notation of  Subsection \ref{glob-split_ssch-sgrp}).  Then we have
  $$  G_\uno^{(1)\!}\big(\, \bk \oplus \mathcal{A} \, \eta \big)  \; = \;  G_{\text{\it odd}}\big(\, \bk \oplus \mathcal{A} \, \eta \big)  \; = \;  \mathbb{A}_\bk^{\text{\it odd}}\big(\, \bk \oplus \mathcal{A} \, \eta \big)  \; = \;  \mathcal{A}^{\,0\,|I}  \, = \;  \mathcal{A}^I   \eqno (3.5)  $$
where  $ \, 0\,|I \, $  is the (possibly infinite) super-dimension of  $ \mathbb{A}_\bk^{\text{\it odd}} $  and $ \mathcal{A}^I $  is taken with odd parity.  Thus  $ \, G_\uno^{(1)\!}\big(\, \bk \oplus \mathcal{A} \, \eta \big) = \mathcal{A}^I \, $  identifies with the set of  $ \mathcal{A} $--points  of the classical (= totally even, or non-super) affine scheme  $ \, \mathcal{A}^I \equiv \mathcal{A}^{I|\,0} \, $.  Therefore, by (3.4) and (3.5) together we conclude that computing the  $ \bk $--superscheme  $ G $  on the central extensions given by ``super-numbers'' on classical algebras   --- e.g., on  $ \mathcal{A}[\eta] \, $,  say ---   is the same as computing the (classical!)  $ \bk $--scheme  $ \, G_\zero \times \mathbb{A}_\bk^I \, $  on classical algebras   --- namely on  $ \mathcal{A} \, $,  say.
\end{free text}

\smallskip

\begin{free text}  \label{examples-sgroups_Grassmann}
 {\bf Splittings on ``Grassmann-points''.}  Let  $ G $  be a  $ \bk $--supergroup,  and  $ \, \Lambda = \bk \big[ {\{\xi_i\}}_{i \in I} \big] \, $  any Grassmann algebra, possibly infinite-dimensional.  Obviously  $ \, \Lambda \in \stspsalg_\bk \, $   --- for a unique, canonical augmentation ---   hence we have a splitting of the group  $ G(\Lambda) $  of  $ \Lambda $--points  of  $ G $  as in  Proposition \ref{split_G-stsplsalg}{\it (a)}.  In particular, this is exactly the splitting mentioned by Boseck in  \cite{bos1}, \S 2,  where indeed only  $ \Lambda $--points  of supergroups are considered.
\end{free text}

\smallskip

\begin{free text}  \label{examples-linear_sgroups}
 {\bf Global splittings of general linear supergroups.}  Let  $ \, \rGL(V) \, $  be a  {\sl linear supergroup\/}  as in Example \ref{exs-supvecs}{\it (b)},  defined over some ground ring  $ \bk \, $.  Letting  $ \; p\,|\,q := \text{\sl rk}(V_\zero) \,\big| \text{\sl rk}(V_\uno) \; $  be the (finite, by assumption) superdimension of  $ V $,  we shall also write  $ \, \rGL_{p|q} := \rGL(V) \, $.  In particular, this means that each element of  $ \, \rGL_{p|q}(A) := \big(\rGL(V)\big)(A) \, $  can be written as a block matrix
 $ \, \displaystyle \bigg(\! {{{\;\;\hskip0,3pt a \;|\; \beta \;\,} \over {\,\; \gamma \;|\; d \;\,}}} \!\bigg)_{\phantom|} $
where  $ a \, $,  $ \beta \, $,  $ \gamma \, $,  $ d \, $  are matrices of size  $ p \times p \, $,  $ p \times q \, $,  $ q \times p \, $,  $ q \times q \, $  respectively, and whose entries respectively belong to  $ \, A_\zero \, $,  $ \, A_\uno \, $,  $ \, A_\uno $  and  $ \, A_\zero \, $.
                                                                \par
   {\sl The condition that such a block matrix in  $ \rgl_{p|q}(A) $  belong to  $ \rGL_{p|q}(A) $   --- i.e., that it be invertible ---   amounts to  $ a $ and  $ d $  being invertible on their own\/}  (see  \cite{ccf},  \S 1.5).
                                                                \par
   Note also that  $ {\big( \rGL_{p|q} \big)}_{\text{\it \!ev}}(A) $  has a neat description: it is the subgroup of all those block matrices for which (in the previous notation)  $ \, \beta = 0 = \gamma \, $.
                                                                \par
   We shall now show that  {\sl  $ \rGL(V) $  is strongly split\/}:  note that this does not depend on the nature of the ground ring  $ \bk $   --- in particular, we do not need it to be a field.
 \vskip5pt
   Define  $ \; {\big( \rGL_{p|q} \big)}_{\text{\it \!odd}_{\phantom{!}}} \! := I + {\big( \rgl_{p|q} \big)}_\uno \; $,  where  $ \, I := I_{p+q} \, $  is the identity (block) matrix of size  $ \, (p+q) \times (p+q) \, $.  This is clearly a totally odd affine superspace, which is stable by the adjoint action of  $ {\big( \rGL_{p|q} \big)}_{\text{\it \!ev}} $   --- both being considered embedded inside  $ \rGL_{p|q} \, $.  Now, a direct check shows that any
 $ \; \displaystyle \bigg(\! {{{\;\; a \;|\; \beta \;\,} \over {\,\; \gamma \;|\; d \;\,}}} \!\bigg)^{\phantom{\big|}\!} \! \in \, \rGL_{p|q}(A) \; $
admits a unique factorization, w.~r.~to the matrix product, as
 \vskip-1pt
  $$  \displaystyle \bigg(\! {{{\;\; a \;|\; \beta \;\,} \over {\,\; \gamma \;|\; d \;\,}}} \!\bigg)  \,\; = \;\,  \bigg(\! {{{\,\; a \;|\; 0 \;\,} \over {\,\; 0 \;|\; d \;\,}}} \!\bigg) \cdot \bigg(\! {{{\,\; \hskip9pt I_p \hskip8pt \;|\; a^{-1} \beta \;\,} \over {\,\; d^{-1} \gamma \;|\; \hskip6pt I_q \hskip6pt \;\,}}} \!\bigg)  $$
 \vskip5pt
\noindent
 This provides a map  $ \; \rGL_{p|q}(A) \longrightarrow {\big( \rGL_{p|q} \big)}_{\text{\it \!ev}}(A) \times {\big( \rGL_{p|q} \big)}_{\text{\it \!odd}_{\phantom{|}}}(A) \; $  which, for  $ A $  ranging in  $ \salg_\bk \, $,  eventually provides a global splitting as we were looking for.
 \vskip5pt
   Instead of the above geometric approach, one can follow an algebraic one.  For that, one simply has to notice that
  $$  \displaylines{
   \quad   \cO\big(\rGL_{p|q}\big)  \,\; = \;\,  \bk\Big[ \big\{ x'_{i,j} \, , \, x''_{h,k} \, , \, \xi'_{i,k} \, , \, \xi''_{h,j} \big\}_{i,j=1,\dots,p;}^{h,k=1,\dots,q;} \, , \, {\text{\sl det}(X')}^{-1} \, , \, {\text{\sl det}(X'')}^{-1} \Big]  \,\; =   \hfill  \cr
   \hfill   = \;\,  \bk\Big[ \big\{ x'_{i,j} \, , \, x''_{h,k} \big\}_{i,j=1,\dots,p;}^{h,k=1,\dots,q;} \, , \, {\text{\sl det}(X')}^{-1} , \, {\text{\sl det}(X'')}^{-1} \Big] \,\otimes_\bk\, \bk\Big[ \big\{ \xi'_{i,k} \, , \, \xi''_{h,j} \big\}_{i,j=1,\dots,p;}^{h,k=1,\dots,q;} \Big]  }  $$
with  $ \, X' := {\big( x'_{i,j} \big)}_{i,j=1,\dots,p;} \, $,  $ \, X'' := {\big( x''_{h,k} \big)}_{h,k=1,\dots,q;} \, $,  is strongly split as a Hopf superalgebra.
 \vskip9pt
   In any case, looking in detail we find that we have proved the following

\medskip

\begin{theorem}  \label{gl-str-split_GL(V)}
 Every general linear  $ \bk $--supergroup  $ \, \rGL(V) := \rGL_{p|q} \, $  is globally strongly split.
\end{theorem}

\end{free text}

\bigskip

\section{Supergroups and super Harish-Chandra pairs}  \label{alsgrps-sHCp}

\smallskip

   {\ } \;\;   Whether in a differential, analytic, or algebraic geometrical framework, with any given supergroup  $ G $  one can always associate, in a functorial way, its  {\sl super Harish-Chandra pair\/}  (or sHCp in short),  namely the pair  $ (G_\zero,\fg) $  formed by the classical (even)  $ G_\zero $  subgroup and the tangent Lie superalgebra  $ \, \fg := \Lie\,(G) \, $  of  $ G $  itself.  The key question is whether one can come back, and in the positive case what kind of (functorial) recipes one can explicitly provide to  reconstruct the original supergroup out of its sHCp.  In this section I present my own solutions to these problems, showing in particular that a positive answer is possible if and only if we restrict our attention to those (affine) supergroups which are globally strongly split   --- so fixing a link
   \hbox{with the first half of the paper.}
                                                       \par
   At first strike I shall deal with the linear case, i.e.\ with supergroups and sHCp's which are linearized.  This is presented as a sheer source of inspiration, after which I treat the general case, which indeed might as well dealt with independently.

\medskip

 \subsection{Super Harish-Chandra pairs}  \label{sHCp} {\ }

\smallskip

\begin{free text}  \label{superHCpairs}
 {\bf Super Harish-Chandra pairs.}  We introduce now the notion of  {\it super Harish-Chandra pair},  indeed a well known one.  Typically, it is considered in the framework of real or complex analytic super Lie groups (see  \cite{koszul}  and \cite{vis}  respectively): here instead we consider the corresponding version adapted to the setup of algebraic supergroups in algebraic supergeometry  (cf.~\cite{cf}, \S 3).
  \eject

\begin{definition}  \label{def-sHCp}
 We call  {\it super Harish-Chandra pair}  (={\it sHCp})  over  $ \bk $  any pair  $ \, (G_+ \, , \, \fg) \, $  such that
 \vskip4pt
   {\it (a)}  $ \quad G_+ $  is an affine  $ \bk $--group-scheme,  $ \,\; \fg \in \slie_\bk \; $,  \, and \,  $ \fg_\uno $  is a finite rank free  $ \bk $--module;
 \vskip3pt
   {\it (b)}  $ \quad \Lie\,(G_+) $  is quasi-representable and  $ \; \Lie(G_+) = \fg_\zero \; $;
 \vskip3pt
   {\it (c)}  there is a  $ G_+ $--action  on  $ \fg $  by automorphisms, denoted  $ \, \Ad : G_+ \longrightarrow \text{\sl Aut}(\fg) \, $,  such that its restriction to  $ \fg_\zero $  is the adjoint action of  $ G_+ $  on  $ \Lie(G_+) = \fg_\zero \, $  and the differential of this action is the Lie bracket of  $ \fg $  restricted to  $ \; \Lie(G_+) \times \fg \, = \, \fg_\zero \times \fg \; $.
 \vskip5pt
   All super Harish-Chandra pairs over  $ \bk $  form the objects of a category, denoted  $ \sHCp_\bk \, $.  The morphisms in  $ \sHCp_\bk $  are all pairs  $ \; (\Omega_+,\omega) : \big( G'_+ \, , \, \fg' \big) \longrightarrow \big( G''_+ \, , \, \fg'' \big) \; $  of a morphism  $ \, \Omega_+ : G'_+ \longrightarrow G''_+ \, $  of  $ \bk $--group  schemes and a morphism  $ \, \omega : \fg' \longrightarrow \fg'' \, $  in  $ \slie_\bk \, $  which are compatible with the additional sHCp structure, that is to say
  $$  \indent \text{\it (d)} \hskip27pt   \omega{\big|}_{\fg_\zero} \, = \, d\Omega_+  \quad ,  \qquad  \Ad\big(\Omega_+(g)\big) \circ \omega \, = \, \omega \circ \Ad(g)  \quad \;\forall \;\, g \in G_+   \hskip27pt   \eqno \diamondsuit  $$
\end{definition}

\medskip

   There is a natural, well-known way to attach a sHCp to any supergroup, which indeed motivates the very notion of sHCp.  In the present context   --- letting  $ \fsgrps_\bk $  be the category of  {\sl fine\/}  $ \bk $--supergroups,  see  Definition \ref{def_fine-sgroups}  ---   it reads as follows:

\medskip

\begin{proposition}  \label{sgrps-->sHCp}
 There exists a functor  $ \; \Phi : \fsgrps_\bk \longrightarrow \sHCp_\bk \; $  given on objects by  $ \; \Phi : G \mapsto \Phi(G) := \big(\, G_\zero \, , \Lie(G) \big) \, $,  and on morphisms by  $ \; \Phi : \varphi \mapsto \Phi(\varphi) := \big(\, \varphi{\big|}_{G_\zero} , \Lie(\varphi) \big) \, $.
\end{proposition}

\medskip

\begin{free text}  \label{inv-prob}
 {\bf The inversion problem for  $ \Phi \, $.}  The main question about the functor  $ \; \Phi : \fsgrps_\bk \!\longrightarrow \sHCp_\bk \; $  is whether it is an equivalence.  In down-to-earth terms, this amounts to asking: can one associate (backwards) a supergroup to any given sHCp, and can one reconstruct any supergroup from its associated sHCp (and conversely)?  In order to answer this question, one looks for a quasi-inverse (i.e., ``inverse up to isomorphism'') functor to  $ \Phi \, $.
                                                                       \par
   In the present, algebraic framework, a solution was given by Masuoka (see  \cite{ms2})  with the assumption that  $ \bk $  be any field of characteristic different from 2, using purely Hopf algebraic techniques.  A weaker result is due to Carmeli and Fioresi (see  \cite{cf}),  who apply Koszul's original method (cf.\  \cite{koszul})  to the context where the ground ring  $ \bk $  be a field of characteristic zero; the same approach was recently extended to  {\sl any\/}  commutative ring  $ \bk $  by Masuoka and Shibata in  \cite{mas-shi}.
                                                                       \par
   In the next two subsections, I present yet another, totally general solution.
\end{free text}

\medskip

 \subsection{The converse functor: linear case}  \label{conv-funct_lin} {\ }

\smallskip

   In this subsection I present my own approach to solve the inversion problem explained in  \S \ref{inv-prob}  above, with a (functorial) geometrical method.  The first approach that I follow is a representation-theoretical one: the basic ingredient to work with is a sHCp together with a faithful representation, which means that I restrict myself to  {\sl linear\/}  sHCp's and linear supergroups.  Later on, I adapt this construction to the general framework of all super Harish-Chandra pairs and all fine supergroups.
 \vskip5pt
   To start with, we define the notions of ``linear'' supergroups and super Harish-Chandra pairs:

\medskip

\begin{definition}  \label{def-lsgrps-lsHCp}  {\ }
 \vskip3pt
   {\it (a)} \,  We call {\it linear gs-split fine supergroup\/}  over  $ \bk $  any pair  $ (G,V) $  where  $ \, G \in \gssfsgrps_\bk \, $,  $ V $  is a finite rank faithful  $ G $--module  (that is,  $ V $  is a free  $ \bk $--supermodule  of finite rank such that  $ G $  embeds into  $ \rGL(V) $  as a closed $ \bk $--supersub\-group),  and
 $ \, {\rgl(V)}_\uno \Big/ \fg_\uno \, $  is  $ \bk $--free   --- or, what is the same,  $ \, {\rgl(V)}_\uno = \fg_\uno \! \oplus \mathfrak{q} \, $  for some  $ \bk $--free  submodule  $ \mathfrak{q} $  of  $ {\rgl(V)}_\uno \, $.
  \eject
   We denote by  $ \, \lgssfsgrps_\bk \, $  the category whose objects are linear supergroups over  $ \bk $  and whose morphisms  $ \, \big(G',V'\big) \longrightarrow \big(G'',V''\big) \, $  are given by pairs  $ \, \big( \varphi_g \, , \varPhi_v \big) \, $  where  $ \, \varphi_g : G' \relbar\joinrel\relbar\joinrel\longrightarrow G'' \, $  and  $ \, \varPhi_v : \rGL\big(V'\big) \relbar\joinrel\relbar\joinrel\longrightarrow \rGL\big(V''\big) \, $  are morphisms of supergroups which obey an obvious compatibility constraint (namely,  $ \varphi_g $  is induced by  $ \varPhi_v $  via restriction).
 \vskip1pt
   {\it (b)} \,  We call {\it linear super Harish-Chandra pair\/}  (over  $ \bk $)  any pair   $ \big( (G_+,\fg) \, , V \,\big) $  where  $ \, (G_+,\fg) \in \text{(sHCp)}_\bk \, $,
 the  $ V $  is a finite rank faithful  {\it  $ (G_+,\fg) $--module}   --- this means, by definition, that  $ V $  is a free  $ \bk $--supermodule  of finite rank with representation monomorphisms  $ \, \boldsymbol{r}_{\!+} \! : G_+ \lhook\joinrel\longrightarrow \rGL(V) \, $,  as  $ \bk $--supergroups  (in particular,  $ G_+ $  is closed in  $ \rGL(V) \, $),  and  $ \, \rho: \fg \lhook\joinrel\longrightarrow \rgl(V) \, $,  as Lie  $ \bk $--superalgebras,  such that  $ \,\; \rho{\big|}_{\fg_\zero} = \, d\boldsymbol{r}_{\!+} \; $,  $ \,\; \text{Ad}\big(\boldsymbol{r}_{\!+}(g)\big) \circ \rho \, = \, \rho \circ \text{Ad}(g) \;\, \big(\, \forall \; g \in G_+ \,\big) \, $  ---   and identifying  $ \fg_\uno $  with  $ \rho(\fg_\uno) \, $,
 we have that  $ \, {\rgl(V)}_\uno \Big/ \fg_\uno \, $  is  $ \bk $--free, or  $ \, {\rgl(V)}_\uno = \fg_\uno \oplus \mathfrak{q} \, $  for some finite  $ \bk $--free  submodule  $ \mathfrak{q} \, $.
                                                                     \par
   We denote by  $ \, \text{\rm (lsHCp)}_\bk \, $  the category whose objects are linear super Harish-Chandra pairs over  $ \bk $  and whose morphisms  $ \, \big( \big( G'_+ , \fg' \big) \, , V' \big) \relbar\joinrel\relbar\joinrel\longrightarrow \big( \big( G''_+ , \fg'' \big) \, , V'' \big) \, $  are given by pairs  $ \, \big( \phi_g \, , \Phi_v \big) \, $  where  $ \, \phi_g : \big( G'_+ , \fg' \big) \relbar\joinrel\relbar\joinrel\longrightarrow \big( G''_+ , \fg'' \big) \, $  is a morphism of super H-C pairs   --- i.e.\ in  $ \, \text{(sHCp)}_\bk \, $  ---   $ \, \Phi_v : \rGL\big(V'\big) \relbar\joinrel\relbar\joinrel\longrightarrow \rGL\big(V''\big) \, $  is a morphism of supergroups, and a natural (obvious) compatibility constraint linking  $ \phi_g $  and  $ \Phi_v $  holds.
 \hfill   $ \diamondsuit $
\end{definition}

\smallskip

\begin{remark}
 It is worth recalling that the constraint for a supergroup to be linear is not that restrictive: indeed, it is well known that any (finite dimensional) affine supergroup   $ G $  is linearizable   --- i.e., can be embedded inside some  $ \rGL(V) $  ---   if its ground ring  $ \bk $  is a field.  Even more, the same is true   --- essentially by the same arguments ---   also when  $ \bk $  is only a PID, under the additional assumption that  $ \cO(G) $  be free as a  $ \bk $--module.
\end{remark}

\smallskip

   It is easy to see from definitions that the functor  $ \; \Phi : \fsgrps_\bk \longrightarrow \sHCp_\bk \; $  considered in  Proposition \ref{sgrps-->sHCp}  above naturally induces a similar functor among the ``associated linear'' categories.  The precise claim reads as follows:

\medskip

\begin{proposition}  \label{lsgrps-->lsHCp}
 There is a unique functor  $ \; \Phi_\ell : \lgssfsgrps_\bk \longrightarrow \lsHCp_\bk \; $  which is given on objects by  $ \; \Phi_\ell : (G,V) \mapsto \Phi_\ell\big((G,V)\big) := \big( \big(\, G_\zero \, , \Lie\,(G) \big) \, , V \,\big) \; $.
\end{proposition}

\end{free text}

\smallskip

   We can now undertake the construction of a quasi-inverse functor to  $ \Phi_\ell \; $.

\medskip

\begin{free text}
 {\bf The functor  $ \; \Psi_\ell : \lsHCp_\bk \longrightarrow \lgssfsgrps_\bk \; $.}  Let us consider a linear sHCp over  $ \bk \, $,  say  $ \, \big( \big( G_+, \fg \big) \, , V \,\big) \in \lsHCp_\bk \, $.  As  $ G_+ $  embeds into  $ \rGL(V) \, $,  we identify  $ G_+ $  itself with its (closed) image inside  $ \rGL(V) \, $;  similarly, we identify  $ \fg $  with its image inside  $ \rgl(V) \, $.  The very definition of linear sHCp then tells us that the pair given by these two images do form a linear sHCp on its own.
%
%
%
 \vskip5pt
   We can now introduce the following definition:
\end{free text}

\smallskip

\begin{definition}  \label{def G_- / G_P - lin}
 Let  $ \, \cP := \big(\! \big( G_+ \, , \fg \big) \, , V \big) \in \lsHCp_\bk \; $.  Let  $ \, 1_{{}_V} \! \in \rgl(V) \, $  be the identity en\-domorphism, fix in  $ \, \fg_\uno \, $  (which is finite free) a  $ \bk $--basis  $ \, {\big\{ Y_i \big\}}_{i \in I} \, $   --- for some finite index set  $ I $  ---   and fix also a total order in  $ I \, $.  For all  $ \, A \in \salg_\bk \, $  consider in  $ \, \rGL(V)(A) \, $  the set  $ \; \big( 1_{{}_V} \! + A_\uno \, Y_i \big) \, := \, \big\{ (\,1_{{}_V} \! + \eta \, Y_i) \,\big|\, \eta \in A_\uno \big\} \; $  for all  $ \, i \in I \, $.  Then set
 \vskip-5pt
  $$  G_-^{\scriptscriptstyle \,<}(A)  \, := \,  {\textstyle \prod\limits_{i \in I}^\rightarrow} \big( 1_{{}_V} \! + A_\uno \, Y_i \big)  \, = \,  \bigg\{\, {\textstyle \prod\limits_{i \in I}^\rightarrow} \big( 1_{{}_V} \! + \eta_i \, Y_i \big) \;\bigg|\; \eta_i \in A_\uno \;\, \forall \; i \in I \,\bigg\}  $$
 \vskip-5pt
\noindent
 where  $ \, \prod\limits_{i \in I}^\rightarrow \, $  denotes an  {\sl ordered product\/}  (with respect to the fixed total order in  $ I \, $),  and
 \vskip-3pt
  $$  G_{{}_{\!\cP}}(A)  \; := \;  \Big\langle\, G_+(A) \;{\textstyle \bigcup}\; G_-^{\scriptscriptstyle \,<}(A) \,\Big\rangle  \; := \;  \Big\langle\, G_+(A) \;{\textstyle \bigcup}\; \big(\, {\textstyle \bigcup_{i \in I}} (\, 1_{{}_V} \! + A_\uno \, Y_i \,) \,\big) \,\Big\rangle  $$
 \vskip-1pt
\noindent
 the subgroup of  $ \rGL(V)(A) $  generated by the subset  $ \, G_+(A) \,\cup\, G_-^{\scriptscriptstyle \,<}(A) \, $,  or by  $ G_+(A) $  and the  $ \big( 1_{{}_V} \! + A_\uno \, Y_i \big) $'s   --- with  $ \, G_+(A) := G_+(A_\zero) \, $  by abuse of notation.
                                                             \par
   Finally, we denote by  $ \; G_-^{\scriptscriptstyle \,<} : \salg_\bk \!\relbar\joinrel\longrightarrow \sets \; $  and  $ \; G_{{}_{\!\cP}} : \salg_\bk \!\relbar\joinrel\longrightarrow \grps \; $  the  $ \bk $--functor  and the  $ \bk $--supergroup  functor defined by  $ \; A \mapsto G_-^{\scriptscriptstyle \,<}(A) \; $  and  $ \; A \mapsto G_{{}_{\!\cP}}(A) \; $   --- by the above recipes ---   on objects and in the obvious way on morphisms.
 \vskip4pt
   {\sl N.B.:}  by definition  $ G_-^{\scriptscriptstyle \,<} $  depends on the choice of the  $ \bk $--basis  $ \, {\big\{ Y_i \big\}}_{i \in I} \, $  of  $ \fg_\uno \, $.  On the other hand, we shall presently see that  $ G_{{}_{\!\cP}} $  instead is independent of such a choice.
 \hfill   $ \diamondsuit $
\end{definition}

\smallskip

\begin{lemma}  \label{tang-group}
 Let  $ \, G \in \sgrps_\bk \, $  be a\/  $ \bk $--supergroup  such that  $ \, \Lie\,(G) \, $  is quasi-re\-presentable, say  $ \, \Lie\,(G) = \cL_\fg \, $  for some  $ \, \fg \in \slie_\bk \, $.
   Let  $ \, A \in \salg_\bk \, $,  $ \, \eta, \eta', \eta'' \in A_\uno \, $,  $ \, c \in A_\zero \, $  such that  $ \, c^2 = 0 \, $,  $ \, Y, Y' \! \in \fg_\uno \, $,  $ \, X \! \in \fg_\zero \, $  and  $ \, g_0 \in G_0(A) \, $.  Then we have (notation of  Definition \ref{def-Lie-salg})
 \vskip5pt
   (a) \;  $ \, \big( 1 + \, c \, X \big) \in G_\zero(A) \; , \;  \big( 1 + \, \eta \, Y \big) \in G(A)  \;\, $;  \; in particular  $ \; \big( 1 + \, \eta \, \eta' \, [Y,Y'\,] \big) \in G_\zero(A) \; $;
 \vskip4pt
   (b) \;  $ \,\;\; (1 + \eta \, Y) \, g_0 \; = \; g_0 \, \big( 1 + \eta \, \Ad\big(g_0^{-1}\big)(Y) \big)  \quad $;
 \vskip4pt
   (c) \;  $ \,\;\; \big( 1 + \eta' \, Y' \big) \, \big( 1 + \eta'' \, Y'' \big) \; = \; \big( 1 + \eta'' \, \eta' \, [Y',Y''\,] \big) \, \big( 1 + \eta'' \, Y'' \big) \, \big( 1 + \eta' \, Y' \big)  \quad $;
 \vskip4pt
   (d) \;  $ \,\;\; \big( 1 + \eta \, Y' \big) \, \big( 1 + \eta \, Y'' \big)  \; = \;  \big( 1 + \eta \, (Y'+Y'') \big)  \; = \;  \big( 1 + \eta \, Y'' \big) \, \big( 1 + \eta \, Y' \big)  \quad $;
 \vskip4pt
   (e) \;  $ \,\;\; \big( 1 + \eta' \, Y \big) \, \big( 1 + \eta'' \, Y \big)  \; = \;  \big( 1 + \eta'' \, \eta' \, Y^{\langle 2 \rangle} \big) \, \big( 1 + (\eta'+\eta'') \, Y \big)  \quad $;
 \vskip4pt
   (f) \;  $ \,\;\; (1 \, + \, \eta \, Y) \, (1 \, + \, \eta' \eta'' X)  \; = \;  (1 \, + \, \eta' \eta'' X) \, \big( 1 \, + \, \eta \, \eta' \eta'' \, [Y,X] \big) \, (1 \, + \, \eta \, Y)  \; = $
                                                              \hfill{\ }\break
   \indent \hskip-1pt   \phantom{(f) \;  $ \,\;\; (1 \, + \, \eta \, Y) \, (1 \, + \, \eta' \eta'' X) $}  $ \; = \;  (1 \, + \, \eta' \eta'' X) \, (1 \, + \, \eta \, Y) \, \big( 1 \, + \, \eta \, \eta' \eta'' \, [Y,X] \big)  \quad $.
 \vskip4pt
   (g) \,  Let  $ \, (h,k) := h\,k\,h^{-1}k^{-1} $  be the commutator of elements  $ h $  and  $ k $  in a group.  Then
  $$  \displaylines{
   \big( \big( 1 + \eta \, Y \big) , \big( 1 + \eta' \, Y' \big) \big)  =  \big( 1 + \eta' \, \eta \, [Y,Y'\,] \big) \; ,  \;\;  \big( \big( 1 + \eta \, Y \big) , \big( 1 + \eta \, Y' \big) \big)  =  \big( 1 + \eta \, (Y+Y') \big)  \cr
   \big( \big( 1 + \eta' \, Y \big) \, , \big( 1 + \eta'' \, Y \big) \big) \, = \, {\big( 1 + \eta'' \, \eta' \, Y^{\langle 2 \rangle} \big)}^2 \, = \, \big( 1 + \eta'' \, \eta' \, 2 \, Y^{\langle 2 \rangle} \big) \, = \, \big( 1 + \eta'' \, \eta' \, [Y,Y] \big)  }  $$

(N.B.: taking the rightmost term in the last identity, the latter is a special case of the first).
\end{lemma}

\begin{proof}
 Recall that  $ \, G(A) = \Hom_{\salg_\bk}\big(\cO(G)\,,A\big) \subseteq \Hom_{{\text{(smod)}}_\bk}\big(\cO(G)\,,A\big) \, $,  the latter being the  $ \bk $--supermodule  of all morphisms between  $ \cO(G) $  and  $ A $  in the category of  $ \bk $--supermodules;  the sum in the formulas then is just the sum in the  $ \bk $--supermodule  $ \, \Hom_{{\text{(smod)}}_\bk}\big(\cO(G)\,,A\big) \, $.  Also, in those formulas the  ``$ 1 $''  stands for the unit element in  $ \, G(A) = \Hom_{\salg_\bk}\big(\cO(G)\,,A\big) \, $,  which is the map given by composition  $ \, 1 := 1_{\scriptscriptstyle G(A)} : \cO(G) \;{\buildrel {\epsilon_{\scriptscriptstyle \cO(G)}} \over {\relbar\joinrel\relbar\joinrel\relbar\joinrel\longrightarrow}}\; \bk \,{\buildrel {u_{\scriptscriptstyle A}} \over {\relbar\joinrel\relbar\joinrel\longrightarrow}}\, A \, $.
                                                               \par
   Once this is fixed, everything follows easily by straightforward calculations and from the identities  $ \, \big( 1 + \, \eta \, \eta' [Y,Y'] \big) \! = \exp\!\big( \eta \, \eta' [Y,Y'] \big) \, $  (see \cite{ccf}, \S 11.5,  for details),  $ \, Y^2 \! = \! Y^{\langle 2 \rangle} \, $  and  $ \, g_0^{-1} \, Y \, g_0 = \Ad\big(g_0^{-1}\big)(Y) \, $,  which do hold in any representation of  $ \fg \, $.  In particular, claim  {\it (g)\/}  directly follows from the identities in  {\it (c)},   {\it (d)\/}  and  {\it (e)}.
 \vskip4pt
   It is possibly worth adding some details for claim  {\it (a)},  which holds by extending to the present super-context a standard trick for group-schemes.
                                                               \par
   Let  $ \, (a,Z) \, $  be  $ \, (c,X) \, $  or  $ \, (\eta,Y) \, $.  From  $ \, Z \in \fg \, $  we have
  $$  a\,Z \, := \, a \otimes Z \in {\big( A_\bk \otimes \fg \big)}_\zero \, =: \, \fg(A)  \;\; .  $$
 By the standard identification of  $ \, \fg(A) = \cL_\fg(A) \, $  with  $ \, \Lie\,(G)(A) := \Ker\big( {G(p)}_A \big) \, $   --- see  \S \ref{Lie(G)}  and references therein ---   we have that  $ \; \big(\, 1 + \, \varepsilon \, a \, Z \,\big) \in G\big(A[\varepsilon]\big) \; $.  But now, as  $ \, \varepsilon^2 = 0 \, $,  the fact that  $ \, \big(\, 1 + \varepsilon \, a \, Z \,\big) \in G\big(A[\varepsilon]\big) = \Hom_{\salg_\bk}\big(\cO(G)\,,A[\varepsilon]\big) \, $  be multiplication preserving is equivalent to the fact that  $ \, a \, Z \in \Hom_{{\text{(smod)}}_\bk}\big(\cO(G)\,,A\big) \, $  be an  $ A $--valued  $ \epsilon_{\scriptscriptstyle \cO(G)} $--derivation,  i.e.
  $$  a \, Z \big( f' \cdot f'' \big)  \,\; = \;\,  a \, Z \big(f'\big) \cdot \epsilon_{\scriptscriptstyle \cO(G)}\big(f''\big) \, + \, \epsilon_{\scriptscriptstyle \cO(G)}\big(f'\big) \cdot a \, Z \big(f''\big)   \eqno \forall\;\, f' , f'' \in \cO(G)  \quad  $$
But then, in turn, as  $ \, a \in \{c,\eta\} \, $  also satisfies  $ \, a^2 = 0 \, $,  we have that  $ \, \big(\, 1 + \, a \, Z \,\big) \, $  is multiplication preserving too, so that  $ \; \big(\, 1 + a \, Z \,\big) \, \in \, \Hom_{\salg_\bk}\big(\cO(G)\,,A\big) =: G(A) \; $,  \, q.e.d.
\end{proof}

\medskip

   This lemma is the key to prove the next relevant result:

\medskip

\begin{proposition}  \label{fact-G_P}
 For any  $ \, A \in {\text{\rm (salg)}}_\bk \, $,  there exist group-theoretic factorizations
  $$  G_{{}_{\!\mathcal{P}}}(A) \; = \; G_+(A) \cdot G_-^{\scriptscriptstyle \,<}(A)  \quad ,  \qquad  G_{{}_{\!\mathcal{P}}}(A) \; = \; G_-^{\scriptscriptstyle \,<}(A) \cdot G_+(A)  $$
   Moreover, the group  $ G_{{}_{\!\mathcal{P}}}(A) $  is independent of the choice of an ordered\/  $ \bk $--basis  $ {\big\{ Y_i \big\}}_{i \in I_{\phantom{|}}} \! $  of\/  $ \fg_{\mathbf{1}} $  used for its definition; the same holds true for the whole functor  $ G_{{}_{\!\cP}} \, $.  Similarly, the sets  $ \, G_+\big(A_{\mathbf{1}}^{(2)}\big) \, G_-^{\scriptscriptstyle \,<}(A) \, $  and  $ \, G_-^{\scriptscriptstyle \,<}(A) \, G_+\big(A_{\mathbf{1}}^{(2)}\big) \, $   --- cf.\ \S 2.1.1 ---   both coincide with the subgroup of\/  $ G_{{}_{\!\mathcal{P}}}(A) $  generated by  $ G_+\big(A_{\mathbf{1}}^{(2)}\big) $  and  $ G_-^{\scriptscriptstyle \,<}(A) \, $,  and they are independent of the choice of an ordered\/  $ \bk $--basis  of\/  $ \fg_{\mathbf{1}} \, $.
\end{proposition}

\begin{proof}
 First of all, as we need it later on, we notice that the even part of  $ G_{{}_{\!\cP}} $  is, directly from definitions, nothing but  $ G_+ \, $,  i.e.~$ \, {\big( G_{{}_{\!\cP}} \big)}_\zero = G_+ \, $.  Definitions imply also that the inverse of any element  $ \, (1 + \eta_i \, Y_i) \in G_-^{\scriptscriptstyle \,<}(A) \, $  is nothing but  $ \, {(1 + \eta_i \, Y_i)}^{-1} = (1 - \eta_i \, Y_i) \in G_-^{\scriptscriptstyle \,<}(A) \, $.  Taking this into account, our goal amounts to showing that
  $$  g'_+ \, {\textstyle \prod\limits_{i \in I}^\rightarrow} \big( 1 + \eta'_i \, Y_i \big) \cdot g''_+ \, {\textstyle \prod\limits_{i \in I}^\rightarrow} \big( 1 + \eta''_i \, Y_i \big)  \; \in \;  G_+(A) \cdot G_-^{\scriptscriptstyle \,<}(A)   \eqno (4.1)  $$
for all  $ \, g'_+ \, , \, g''_+ \in G_+(A) \, $  and  $ \, \eta'_i \, , \, \eta''_i \! \in \! A_\uno \, $,  i.e.~we can re-write  $ \; g'_+ \, {\textstyle \prod\limits_{i \in I}^\rightarrow} \! \big( 1 + \eta'_i \, Y_i \big) \cdot g''_+ \, {\textstyle \prod\limits_{i \in I}^\rightarrow} \! \big( 1 + \eta''_i \, Y_i \big) \; $  as the product of an element in  $ G_+(A) $  times an ordered product of factors of type  $ \, (1 + \eta_\ell \, Y_\ell) \, $.
                                               \par
   First of all, claim  {\it (b)\/}  of  Lemma \ref{tang-group}  gives  (for all  $ \, i \in I \, $)
  $$  (1 + \eta'_i \, Y_i) \, g''_+ \; = \; g''_+ \, \big( 1 + \eta'_i \, \Ad\big({(g''_+)}^{-1}\big)(Y_i) \big) \; = \; g''_+ \, \big( 1 + \eta'_i \, {\textstyle \sum_{j \in I}} \, c_{i,j} \, Y_j \,\big)  $$
for some  $ \, c_{i,j} \in \bk \, $  ($ \, j \in I \, $).  But now the special case of claim  {\it (d)\/}  in  Lemma \ref{tang-group}  implies
  $$  \Big( 1 + \eta'_i \, {\textstyle \sum_{j \in I}} \, c_{i,j} \, Y_j \,\Big)  \, = \,  {\textstyle \prod\limits_{j \in I}} \big( 1 + \eta'_i \, c_{i,j} \, Y_j \,\big)  \, = \,  {\textstyle \prod\limits_{j \in I}^\rightarrow} \big( 1 + \eta'_i \, c_{i,j} \, Y_j \,\big)  \,\; \in \;\,  G_-^{\scriptscriptstyle \,<}(A) \;\;   \eqno (4.2)  $$
as in particular the factors in the product(s) do commute among themselves.
                                              \par
   Applying all this to  $ \; g'_+ \, {\textstyle \prod\limits_{i \in I}^\rightarrow} \! \big( 1 + \eta'_i \, Y_i \big) \cdot g''_+ \, {\textstyle \prod\limits_{i \in I}^\rightarrow} \! \big( 1 + \eta''_i \, Y_i \big) \; $  we eventually find
  $$  g'_+ \, {\textstyle \prod\limits_{i \in I}^\rightarrow} \! \big( 1 + \eta'_i \, Y_i \big) \cdot g''_+ \, {\textstyle \prod\limits_{i \in I}^\rightarrow} \! \big( 1 + \eta''_i \, Y_i \big)  \,\; = \;\,  \big( g'_+ \, g''_+ \big) \cdot \, {\textstyle \prod\limits_{i \in I}^\rightarrow} \, {\textstyle \prod\limits_{j \in I}^\rightarrow} \big( 1 + \eta'_i \, c_{i,j} \, Y_j \,\big) \, {\textstyle \prod\limits_{i \in I}^\rightarrow} \big( 1 + \eta''_i \, Y_i \big)  $$
in which the first factor  $ \, \big( g'_+ \, g''_+ \big) \, $  of the right-hand side does belong to  $ G_+(A) \, $.  Therefore, in order to prove (4.1) we are left to show that the following holds:

\vskip11pt

   {\sl  $ \underline{\text{\it Claim}} $:  Any (possibly unordered) product of the form  $ \, {\textstyle \prod_{k=1}^N} \big( 1 + \eta_k \, Y_{i_k} \big) \, $  can be ``re-ordered'', i.e.~it can be re-written as an element of  $ \, G_+\big(A_\uno^{(2)}\big) \cdot G_-^{\scriptscriptstyle \,<}(A) \, $.}

\vskip11pt

   In order to prove the  {\it Claim},  let  $ \fa $  be the (two-sided) ideal of  $ A $  generated by the  $ \eta_k $'s,  and denote by  $ \, \fa^n \, $  its  $ n $--th  power, for any  $ \, n \in \N \, $.  As the  $ \eta_k $'s  are finitely many  {\sl odd\/}  elements, we have  $ \, \fa^n = \{0\} \, $,  for all  $ \, n > N \, $.
                                                     \par
   Looking at the product  $ \, {\textstyle \prod_{k=1}^N} \big( 1 + \eta_k \, Y_{i_k} \big) \, $,  we define its  {\sl inversion number\/}  as being the number of occurrences of two consecutive indices  $ k_s $  and  $ k_{s+1} $  such that  $ \, i_{k_s} \npreceq i_{k_{s+1}} \, $:  the product itself then is  {\sl ordered\/}  iff its inversion number is zero.
                                                     \par
   Now assume the product  $ \, g := {\textstyle \prod_{k=1}^N} \big( 1 + \eta_k \, Y_{i_k} \big) \, $  is unordered: then there exists at least an inversion, say  $ \, i_{k_s} \npreceq i_{k_{s+1}} \, $,  i.e.~either  $ \, i_{k_s} \succ i_{k_{s+1}} \, $  or  $ \, i_{k_s} = i_{k_{s+1}} \, $.  Using claim  {\it (c)\/}  or  {\it (e)},  respectively, of  Lemma \ref{tang-group}  we can re-write the product  $ \; \big( 1 + \eta_{k_s} Y_{i_{k_s}} \big) \, \big( 1 + \eta_{k_{s+1}} Y_{i_{k_{s+1}}} \big) \; $  as
   $$  \displaylines{
    \big( 1 + \eta_{k_s} Y_{i_{k_s}} \big) \, \big( 1 + \eta_{k_{s+1}} Y_{i_{k_{s+1}}} \big)  \; = \;  \big( 1 + \eta_{k_{s+1}} \, \eta_{k_s} \, \big[Y_{i_{k_{s+1}}},Y_{i_{k_s}}\big] \big) \, \big( 1 + \eta_{k_{s+1}} Y_{i_{k_{s+1}}} \big) \, \big( 1 + \eta_{k_s} Y_{i_{k_s}} \big)  \cr
    \big( 1 + \eta_{k_s} Y_{i_{k_s}} \big) \, \big( 1 + \eta_{k_{s+1}} Y_{i_{k_s}} \big)  \,\; = \;  \Big( 1 + \eta_{k_{s+1}} \, \eta_{k_s} Y_{i_{k_s}}^{\,\langle 2 \rangle} \Big) \, \big( 1 + (\eta_{k_s} \! + \eta_{k_{s+1}}) \, Y_{i_{k_s}} \big)  }  $$
   Thus, re-writing in this way the product of the  $ k_s $--th  and the  $ k_{s+1} $--th  factor in the original product  $ \, g := {\textstyle \prod_{k=1}^N} \big( 1 + \eta_k \, Y_{i_k} \big) \, $,  we find another product expression in which we did eliminate one inversion, but we payed the price of inserting a  {\sl new factor}.  However, in both cases this new factor is of the form  $ \, \big( 1 + a \, X \big) \, $  for some $ \, X \in \fg_\zero \in G_+(A) \, $  and  $ \, a \in \fa^2 \, $.
                                                     \par
   By repeated use of  Lemma \ref{tang-group}{\it (f)\/}  we can shift this new factor  $ \, \big( 1 + a \, X \big) \, $  to the leftmost position in  $ g $  (now re-written once more in yet a different product form) up to paying the price of inserting several  {\sl new factors of the form\/}  $ \, \big( 1 + b_t \, Z_t \big) \, $  for some $ \, Z_t \in \fg_\uno \, $  and  $ \, b_t \in \fa^3 \, $.  Moreover, by  Lemma \ref{tang-group}{\it (d)\/}  each one of these new factors can be written as a product of factors of the form  $ \, \big( 1 + \eta'_h \, Y_{i'_h} \big) \, $  where  $ \, \eta'_h \in A_\uno \, $  is a multiple of some  $ \, b_t \, $,  hence  $ \, \eta'_h \in \fa^3 \, $  too.
                                                     \par
   Eventually, we find a new factorization of the original element  $ \, g := {\textstyle \prod_{k=1}^N} \big( 1 + \eta_k \, Y_{i_k} \big) \, $  in the new form  $ \, g := g'_0 \cdot {\textstyle \prod_{h=1}^{N'}} \big( 1 + \eta'_h \, Y_{i'_h} \big) \, $  where  $ \, g'_0 \in G_+(A) \, $  and the factors  $ \, \big( 1 + \eta'_h \, Y_{i'_h} \big) \, $  satisfy the following conditions:
 \vskip5pt
   {\it --- (a)}\,  each new factor  $ \, \big( 1 + \eta'_h \, Y_{i'_h} \big) \, $  is either one of the old factors  $ \, \big( 1 + \eta_k \, Y_{i_k} \big) \, $  or a truely new one;
 \vskip4pt
   {\it --- (b)}\,  for every (truely) new factor  $ \, \big( 1 + \eta'_h \, Y_{i'_h} \big) \, $  one has  $ \; \eta'_h \in \fa^3 \; $;
 \vskip4pt
   {\it --- (c)}\,  the number of inversions among factors  $ \, \big( 1 + \eta'_h \, Y_{i'_h} \big) = \big( 1 + \eta_k \, Y_{i_k} \big) \, $  of the old type is one less than before.
 \vskip7pt
   Iterating this procedure, after finitely many steps we can achieve a new factorization of the original element  $ \, g := {\textstyle \prod_{k=1}^N} \big( 1 + \eta_k \, Y_{i_k} \big) \, $  as a new product  $ \, g = g''_0 \cdot {\textstyle \prod_{h=1}^{N''}} \big( 1 + \eta''_h \, Y_{i''_h} \big) \, $  where  $ \, g''_0 \in G_+(A) \, $  and the factors  $ \, \big( 1 + \eta''_h \, Y_{i''_h} \big) \, $  enjoy properties  {\it (a)\/}  and  {\it (b)\/}  above plus the ``optimal version'' of  {\it (c)},  namely
 \vskip4pt
   {\it --- (c+)}\,  the number of inversions among factors
 of the old type is zero.
 \vskip7pt
   Now we apply the same ``reordering operation'' to the product  $ \, {\textstyle \prod_{h=1}^{N''}} \big( 1 + \eta''_h \, Y_{i''_h} \big) \, $.  By assumption, now an inversion can occur only among two factors of new type or among an old and a new factor.  But then, the two coefficients  $ \eta''_h $  involved by the inversion belong to  $ \fa $  and at least one of them belong to  $ \fa^3 \, $.  It follows that when one performs the ``reordering operation'' onto the pair of factors involved in the inversion the new factor which pops up necessarily involves a coefficient in  $ \fa^4 \, $.  As this applies for any possible inversion, at the end of the day we shall find a new factorization of  $ g $  of the form
  $$  g  \; = \;  g''_0 \cdot \widehat{g}_0 \cdot {\textstyle \prod_{t=1}^{\widehat{N}}} \big( 1 + \widehat{\eta}_t \, Y_{\widehat{i}_t} \big)  $$
in which  $ \, \widehat{g}_0 \in G_+(A) \, $  and the factors  $ \, \big( 1 + \widehat{\eta}_t \, Y_{\widehat{i}_t} \big) \, $  are either old factors  $ \, \big( 1 + \eta_k \, Y_{i_k} \big) \, $,  {\sl with no inversions among them},  or new factors for which  $ \, \widehat{\eta}_t \in \fa^5 \, $.
                                             \par
   The end of the story is clear.  We can iterate at will this procedure, and then   --- since  $ \, \fa^n = \{0\} \, $  for  $ \, n > N \, $  ---   after finitely many steps we have no longer any new factor popping out; thus, we eventually find a last factorization of  $ g $  of the form
  $$  g  \,\; = \;\,  \widetilde{g}_0 \cdot {\textstyle \prod_{\ell=1}^{\widetilde{N}}} \big( 1 + \widetilde{\eta}_\ell \, Y_{\widetilde{i}_\ell} \big)  \,\; = \;\,  \widetilde{g}_0 \cdot {\textstyle \prod\limits^\rightarrow} {\textstyle {}_{\ell=1}^{\widetilde{N}}} \big( 1 + \widetilde{\eta}_\ell \, Y_{\widetilde{i}_\ell} \big)  $$
in which  $ \, \widetilde{g}_0 \in G_+(A) \, $  and  $ \,\; {\textstyle \prod_{\ell=1}^{\widetilde{N}}} \big( 1 + \widetilde{\eta}_\ell \, Y_{\widetilde{i}_\ell} \big) \, = \, {\textstyle \prod\limits^\rightarrow} {}_{\ell=1}^{\widetilde{N}} \big( 1 + \widetilde{\eta}_\ell \, Y_{\widetilde{i}_\ell} \big) \; \in \; G_-^{\scriptscriptstyle \,<}(A) \;\, $  {\sl  is an ordered product},  as required, so that  $ \; g \, \in \, G_+(A) \cdot G_-^{\scriptscriptstyle \,<}(A) \, $,  \, q.e.d.

\vskip7pt

   For the last part of the main statement, let  $ \, {\big\{ Y_i \big\}}_{i \in I} $  and  $ \, {\big\{ Z_i \big\}}_{i \in I} $  be two (finite)  $ \bk $--bases  of  $ \fg_\uno \, $;  then  $ \, Z_j = \sum_{i \in I} c_i \, Y_i \, $  (with  $  \, c_i  \in I \, $)  for each  $ \, j \in I \, $.  The same argument proving (4.2) also yields
 \vskip-5pt
  $$  1 + \psi_j \, Z_j  \; = \;  1 + \psi_j \, {\textstyle \sum\limits_{i \in I}} \, c_i \, Y_i  \; = \;
{\textstyle \prod\limits_{i \in I}^\rightarrow} \big( 1 + \psi_j \, c_i \, Y_i \,\big)  \,\; \in \;\,  G_-^{\scriptscriptstyle \,<, Y}(A)   \eqno (4.3)  $$
 \vskip-2pt
\noindent
 where  $ G_-^{\scriptscriptstyle \,<, Y} $  is relative to the group  $ G_{{}_{\!\cP}}^{\,\scriptscriptstyle Y} $  defined as in  Definition \ref{def G_- / G_P - lin}  making use of the basis of the  $ \, {\big\{ Y_i \big\}}_{i \in I} \, $.  Letting  $ G_-^{\scriptscriptstyle \,<, Z} $  be the similar group defined via the basis of the  $ Z_j $'s,  formula (4.3) proves that  $ \, G_{{}_{\!\cP}}^{\,\scriptscriptstyle \,<, Z}(A) \subseteq G_+\big(A_{\mathbf{1}}^{(2)}\big) \, G_{{}_{\!\cP}}^{\,\scriptscriptstyle \,<, Y}(A) \, $,  \, so by symmetry we eventually get  $ \, G_+\big(A_{\mathbf{1}}^{(2)}\big) \, G_{{}_{\!\cP}}^{\,\scriptscriptstyle \,<, Z}(A) = G_+\big(A_{\mathbf{1}}^{(2)}\big) \, G_{{}_{\!\cP}}^{\,\scriptscriptstyle \,<, Y}(A) \, $.
\end{proof}

\medskip

   To improve the previous result, we need a couple of additional lemmas.

\medskip

\begin{lemma}  \label{lemma-prod}
 Let  $ \, A \in \salg_\bk \, $,  let $ \, \hat{\eta}_i \, , \check{\eta}_i \in A_\uno \, $  and let  $ \, \mathfrak{q} \, $  be an ideal of  $ A \, $  such that  $ \, \hat{\eta}_i , \check{\eta}_i \in \mathfrak{q} \, $  and  $ \, \alpha_i := \hat{\eta}_i - \check{\eta}_i \in \mathfrak{q}^n \, $  ($ \, i \in I \, $)  for some  $ \, n \in \N \, $. Then
 \vskip-5pt
  $$  {\textstyle \prod\limits_{i \in I}^\rightarrow} \big( 1 + {[ \hat{\eta}_i ]}_{n+1} Y_i \,\big) \cdot {\textstyle \prod\limits_{i \in I}^\leftarrow} \big( 1 - {[\check{\eta}_i]}_{n+1} Y_i \,\big)  \,\; = \;\,  {\textstyle \prod\limits_{i \in I}^\rightarrow} \big( 1 + {[\alpha_i]}_{n+1} Y_i \,\big)  \;\; \in \;\;  G_{{}_{\cP}}\big( A \big/ \mathfrak{q}^{n+1} \big)  $$
 \vskip-2pt
\noindent
 where  $ \prod\limits_{i \in I}^\rightarrow $  and  $ \prod\limits_{i \in I}^\leftarrow $  respectively denote an ordered and a reversely-ordered product (w.r.~to the given order in  $ I \, $)  and  $ \, {[a]}_{n+1} \in A \big/ \mathfrak{q}^{\,n+1} \, $  stands for the coset modulo  $ \mathfrak{q}^{n+1} $  of any  $ \, a \in A \, $.
\end{lemma}

\begin{proof}
 This is an easy, straightforward consequence of  claims  {\it (e)\/}  and  {\it (f)\/}  in  Lemma \ref{tang-group}.
\end{proof}

\medskip

\begin{lemma}  \label{lemma-triv_prod}
 For any given  $ \, A \in \salg_\bk \, $,  let  $ \, \zeta_i \in A_\uno \, $  ($ \, i \in I \, $)  be such that
                 \hfill\break
 $ \; g := {\textstyle \prod\limits_{i \in I}^\rightarrow} \big( 1 + \zeta_i \, Y_i \,\big) \, \in \, G_+(A\,) \bigcap G_-^{\scriptscriptstyle \,<}(A\,) \; $.  Then  $ \; \zeta_i = 0 \; $  for all  $ \, i \in I \, $.
\end{lemma}

\begin{proof}
 By our global assumptions we have  $ \, G_{{}_{\!\cP}}(A\,) \subseteq \rGL(V)(A\,) \subseteq \End_{\,\bk}(V)(A\,) \, $,  with  $ \End_{\,\bk}(V)(A\,) $  being a unital, associative  $ A_\zero $--algebra:  indeed, fixing a homogeneous  $ \bk $--bases  for  $ V $  we can read  $ \End_{\,\bk}(V)(A\,) $  as an algebra of block matrices, in which the diagonal blocks have entries in  $ A_\zero \, $  and the other ones have entries in  $ A_\uno \, $  (like in  Example \ref{exs-supvecs}{\it (b)\/}  and references therein).  Thus, inside  $ \, \End_{\,\bk}(V)(A\,) \, $  we can expand the product  $ \; g := {\textstyle \prod\limits_{i \in I}^\rightarrow} \big( 1 + \zeta_{{}_{\,\scriptstyle i\,}} Y_i \,\big) \, \in \, G_{{}_{\!\cP}}(A\,) \subseteq \End_{\,\bk}(V)(A\,) \; $  so to get
  $$  g  \; := \;  {\textstyle \prod\limits_{i \in I}^\rightarrow} \big( 1 + \zeta_{{}_{\,\scriptstyle i\,}} Y_i \,\big)  \; = \;  1 \, + {\textstyle \sum\limits_{n \in \N_+}} c_n\big(\, \underline{\zeta} \,\big)   \eqno (4.4)  $$
where each  $ c_n\big(\, \underline{\zeta} \,\big) $  denotes a (block) matrix in  $ \End_{\,\bk}(V)(A\,) $  whose entries are homogeneous polynomials in the  $ \zeta_{{}_{\,\scriptstyle i}} $'s  of degree  $ n \, $.
                                                    \par
   {\sl In particular we have  $ \, c_n\big(\, \underline{\zeta} \,\big) = 0 \, $  for all  $ \, n > |I| \, $},  and moreover  $ \; c_1\big(\, \underline{\zeta} \,\big) = \sum_{i \in I} \zeta_{{}_{\,\scriptstyle i\,}} Y_i \; $.
 \vskip5pt
   Now let  $ \, \mathfrak{a} := \big( {\big\{ \zeta_{{}_{\,\scriptstyle i}} \big\}}_{i \in I} \big) \, $  be the ideal of  $ A $  generated by all the  $ \zeta_{{}_{\,\scriptstyle i}} $'s.  For all  $ \, n \in \N \, $,  let  $ \; p_n : A \relbar\joinrel\relbar\joinrel\twoheadrightarrow A \,\big/ \mathfrak{a}^{\,n} =: {[A\,]}_n \; $  be the canonical quotient map, for which we write  $ \, {[a]}_n := p_n(a) \, $  for every  $ \, a \in A \, $.  Correspondingly, we let  $ \, G_{{}_{\!\cP}}(p_n) : G_{{}_{\!\cP}}(A\,) \relbar\joinrel\relbar\joinrel\twoheadrightarrow G_{{}_{\!\cP}}\big( A \,\big/ \mathfrak{a}^{\,n} \big) =: G_{{}_{\!\cP}}\big( {[A\,]}_n \big) \, $  the associated group morphism and we write  $ \, {[y]}_n := G_{{}_{\!\cP}}(p_n)(y) \, $  for every  $ \, y \in G_{{}_{\!\cP}}(A\,) \, $.
                                                    \par
   Applying this to (4.4) above we find
  $$  {[g]}_2  \; := \;  {\big[\, 1 \, + \, c_1\big(\, \underline{\zeta} \,\big) \,\big]}_2  \; = \;  1 \, + \, {\textstyle \sum\limits_{i \in I}} \, {\big[ \zeta_{{}_{\,\scriptstyle i\,}} \big]}_2 \, Y_i  \,\; \in \;\, G_{{}_{\!\cP}}\big({[A\,]}_2\big)   \eqno (4.5)  $$
On the other hand, the assumption  $ \, g \in G_+(A\,) \bigcap G_-^{\scriptscriptstyle \,<}(A\,) \, $  implies also
$ \, {[g]}_2 \in G_+\big({[A\,]}_2\big) \, $,  which means that all entries   --- belonging to  $ {[A\,]}_2 $  ---   of the matrix  $ {[g]}_2 $  actually belong to the  {\sl even\/} part of  $ {[A\,]}_2 \, $.  This together with (4.5) forces  $ \; {\textstyle \sum\limits_{i \in I}} \, {\big[ \zeta_{{}_{\,\scriptstyle i\,}} \big]}_2 \, Y_i \, = \, 0 \; $.  In turn, by the linear independence of the  $ Y_i $'s   --- inside  $ \, \fg_\uno \, $,  hence inside  $ \, A_\uno \cdot \fg_\uno \, \subseteq \, {\big( A \otimes_\bk \fg \big)}_\zero \, \subseteq \, \End(V)(A) \, $  ---   this implies  $ \, {\big[ \zeta_{{}_{\,\scriptstyle i\,}} \big]}_2 = {\big[ 0 \big]}_2 \in \, {[A\,]}_2 := A \,\big/ \mathfrak{a}^{\,2} \, $,  hence  $ \, \zeta_{{}_{\,\scriptstyle i\,}} \! \in \mathfrak{a}^{\,2} \, $  for all  $ \, i \in I \, $.
                                                    \par
   Now,  $ \, {\big\{ \zeta_{{}_{\,\scriptstyle i}} \big\}}_{i \in I} \! \subseteq \mathfrak{a}^{\,2} = {\big( {\big\{ \zeta_{{}_{\,\scriptstyle i}} \big\}}_{i \in I} \big)}^2 \, $  automatically entails  $ \, {\big\{ \zeta_{{}_{\,\scriptstyle i}} \big\}}_{i \in I} \! \subseteq \mathfrak{a}^{\,n} \, $  for all  $ \, n \in \N_+ \, $.  As  $ \, \mathfrak{a}^{\,n} = \{0\} \, $  for  $ \, n \gg 0 \, $,  we end up with  $ \, \zeta_{{}_{\,\scriptstyle i\,}} \! = 0 \; $  for all  $ \, i \in I \, $,  \, q.e.d.
\end{proof}

\medskip

\begin{remark}  \label{alt-proof}
 An alternative argument to finish the previous proof is the following.  Once we have found that  $ \, \zeta_{{}_{\,\scriptstyle i\,}} \! \in \mathfrak{a}^{\,2} \, $  for all  $ \, i \in I \, $,  we remark that this implies  $ \, c_n\big(\, \underline{\zeta} \,\big) \in \mathfrak{a}^{\,2n} \, $  for all  $ \, n \in \N_+ \, $.  Then (4.4) yields the analogue of (4.5), namely
  $$  {[g]}_4  \; :=  {\big[\, 1 \, + \, c_1\big(\, \underline{\zeta} \,\big) \,\big]}_4  \; = \;  1 \, + \, {\textstyle \sum\limits_{i \in I}} \, {\big[ \zeta_{{}_{\,\scriptstyle i\,}} \big]}_4 \, Y_i  \,\; \in \;\, G_{{}_{\!\cP}}\big({[A\,]}_4\big)  $$
and again, acting like above, by a parity argument in  $ \, G_{{}_{\!\cP}}\big({[A\,]}_4\big) \, $  along with the linear independence of the  $ Y_i $'s  we get  $ \, \zeta_{{}_{\,\scriptstyle i\,}} \! \in \mathfrak{a}^{\,4} \, $  for all  $ \, i \in I \, $.
                                                    \par
   We can now  {\sl iterate\/}  this procedure, thus finding  $ \, \zeta_{{}_{\,\scriptstyle i\,}} \! \in \mathfrak{a}^{\,2\,n} \, $  (for  $ \, i \in I \, $)  for all  $ \, n \in \N_+ \, $.  As  $ \, \mathfrak{a}^{\,2\,n} = \{0\} \, $  for  $ \, n \gg 0 \, $,  we end up with  $ \, \zeta_{{}_{\,\scriptstyle i\,}} \! = 0 \, $  for all  $ \, i \in I \, $,  \, q.e.d.
\end{remark}

\medskip

   Thanks to the previous lemmas, we can improve  Proposition \ref{fact-G_P}  as follows:

\medskip

\begin{proposition}  \label{dir-prod-fact-G_P - lin}  {\ }
 \vskip7pt
   {\it (a)} \,  The restriction of group multiplication in  $ G_{{}_{\!\cP}} $  provides  $ \bk $--superscheme  isomorphisms
  $$  G_+ \times G_-^{\scriptscriptstyle \,<} \; \cong \; G_{{}_{\!\cP}}  \quad ,  \qquad  G_-^{\scriptscriptstyle \,<} \times G_+ \; \cong \; G_{{}_{\!\cP}}  $$
 \vskip0pt
   {\it (b)} \,  There exist\/  $ \bk $--superscheme  isomorphisms  $ \, G_-^{\scriptscriptstyle \,<} \cong \mathop{\times}\limits_{i \in I} \mathbb{A}_\bk^{0|1} \cong \mathbb{A}_\bk^{0|d_-} \, $  with  $ \, d_- := |I| \; $.
\end{proposition}

\begin{proof}
 {\it (a)} \,  It is enough to prove the first identity, which amounts to showing the following: for any  $ \, A \in \salg_\bk \, $,  if  $ \; \hat{g}_+ \, \hat{g}_- = \check{g}_+ \, \check{g}_- \; $  for  $ \, \hat{g}_\pm \, , \check{g}_\pm \in G_\pm(A) \, $,  then  $ \, \hat{g}_+ = \check{g}_+ \, $  and  $ \, \hat{g}_- = \check{g}_- \; $.
                                                                   \par
   From the assumption  $ \; \hat{g}_+ \, \hat{g}_- \, = \, \check{g}_+ \, \check{g}_- \; $  we get  $ \; g \, := \, \hat{g}_- \, \check{g}_-^{-1} \, = \, \hat{g}_+^{-1} \, \check{g}_+ \, \in \, G_+(A) \, $,  as  $ G_+(A) $  is a subgroup in  $ G(A) \, $.  Now  $ \, \hat{g}_- \in \, G_-^{\scriptscriptstyle \,<}(A) \, $  has the form  $ \, \hat{g}_- = {\textstyle \prod\limits^\rightarrow}_{i \in I} \big(\, 1 + \hat{\eta}_i \, Y_i \,\big) \, $  and similarly  $ \, \check{g}_- = {\textstyle \prod\limits^\rightarrow}_{i \in I} \big( 1 + \check{\eta}_i Y_i \big) \, $  so that  $ \, \check{g}_-^{\,-1} = {\textstyle \prod\limits^\leftarrow}_{i \in I} \big( 1 - \check{\eta}_i Y_i \big) \, $,  where once more  $ {\textstyle \prod\limits^\rightarrow} $  and  $ {\textstyle \prod\limits^\leftarrow} $  respectively denote an ordered and a reversely-ordered product.  Therefore we have
  $$  g  \,\; := \;\,  \hat{g}_- \, \check{g}_-^{-1}  \,\; = \;\,  {\textstyle \prod\limits_{i \in I}^\rightarrow} \big( 1 + \hat{\eta}_i Y_i \big) \, {\textstyle \prod\limits_{i \in I}^\leftarrow} \big( 1 - \check{\eta}_i Y_i \big)  \;\; \in \;\;  G_+(A) \; \subseteq \; G_{{}_{\!\cP}}(A)   \eqno (4.6)  $$
   \indent   We define  $ \, \fa := \big( {\big\{\, \hat{\eta}_{{}_{\,\scriptstyle i}} , \check{\eta}_{{}_{\,\scriptstyle i}} \,\big\}}_{i \in I} \big) \, $  the ideal of  $ A $  generated by all the  $ \hat{\eta}_{{}_{\,\scriptstyle i}} $'s  and the  $ \check{\eta}_{{}_{\,\scriptstyle i}} $'s.
   Like in the proof of  Lemma \ref{lemma-triv_prod},  for  $ \, n \in \N \, $  we write  $ \; p_n : A \relbar\joinrel\relbar\joinrel\twoheadrightarrow A \,\big/ \fa^{\,n} =: {[A]}_n \; $  for the canonical quotient map and  $ \, {[a\,]}_n := p_n(a) \, $  for every  $ \, a \in A \, $,  and then also, correspondingly,  $ \, G_{{}_{\!\cP}}(p_n) : G_{{}_{\!\cP}}(A) \relbar\joinrel\relbar\joinrel\twoheadrightarrow G_{{}_{\!\cP}}\big( A \,\big/ \fa^{\,n} \big) =: G_{{}_{\!\cP}}\big( {[A]}_n \big) \, $  for the associated group morphism and  $ \, {[y]}_n := G_{{}_{\!\cP}}(p_n)(y) \, $  for every  $ \, y \in G_{{}_{\!\cP}}(A) \, $.  Now (4.6) along with  Lemma \ref{lemma-prod}  gives
  $$  {[g]}_2  \; = \;  {\textstyle \prod\limits_{i \in I}^\rightarrow} \big( 1 + {[\hat{\eta}_i]}_2 \, Y_i \big) \, {\textstyle \prod\limits_{i \in I}^\leftarrow} \big( 1 - {[\check{\eta}_i]}_2 \, Y_i \big)  \; = \;  {\textstyle \prod\limits_{i \in I}^\rightarrow} \big( 1 + {[\alpha_i]}_2 \, Y_i \big)  \;\; \in \;\;  G_{{}_{\!\cP}}\big({[A]}_2\big)  $$
with  $ \; \alpha_i := \hat{\eta}_i - \check{\eta}_i \, \in \fa \; $  for all  $ i \, $.  Since it is also  $ \; {[g]}_2 \, \in \, G_+\big({[A]}_2\big) \, \bigcap \, G_-^{\scriptscriptstyle \,<}\big({[A]}_2\big) \, $,  \, we can apply  Lemma \ref{lemma-triv_prod},  with  $ {[A]}_2 $  playing the r{\^o}le of  $ A \, $,  \, giving  $ \, {[\alpha_i]}_2 = {[0]}_2 \in {[A]}_2 \, $,  that is  $ \, \alpha_i \in \fa^{\,2} \, $,  for all  $ \, i \in I \, $.  But now  Lemma \ref{lemma-triv_prod}  applies again, with  $ {[A]}_3 $  playing the r{\^o}le of  $ A \, $, yielding  $ \, {[\alpha_i]}_3 = {[0]}_3 \in {[A]}_3 \, $  hence  $ \, \alpha_i \in \fa^{\,3} \, $,  for all  $ \, i \in I \, $.  {\sl Then we iterate},  finding by induction that  $ \, \alpha_i \in \fa^{\,n} \, $  (for  $ \, i \in I \, $)  for all  $ \, n \in \N_+ \, $;  as  $ \, \fa^{\,n} = \{0\} \, $  for  $ \, n \gg 0 \, $  we end up with  $ \, \hat{\eta}_i - \check{\eta}_i =: \alpha_i = 0 \, $,  i.e.~$ \, \hat{\eta}_i = \check{\eta}_i \, $,  for all  $ \, i \in I \, $.  This yields  $ \, \hat{g}_- = \check{g}_- \, $,  and from this we get also  $ \, \hat{g}_+ = \check{g}_+ \, $,  \, q.e.d.
%
%
%
 \vskip9pt
   {\it (b)} \,  By definition there exists a  $ \bk $--superscheme  epimorphism  $ \, \Theta : \mathbb{A}_\bk^{0|d_-} \!\!\relbar\joinrel\longrightarrow G_-^{\scriptscriptstyle \,<} \, $  which is given on every single  $ \, A \in \salg_\bk \, $  by
  $$  \Theta_{\!A} \, : \, \mathbb{A}_\bk^{0|d_-}(A) := A_\uno^{\,\times d_-} \!\! \relbar\joinrel\longrightarrow\, G_-^{\scriptscriptstyle \,<}(A) \;\; ,  \quad  {\big( \eta_i \big)}_{i \in I} \, \mapsto \, \Theta_{\!A}\big(\! {\big( \eta_i \big)}_{i \in I} \big) := {\textstyle \prod\limits_{i \in I}^\rightarrow} \big(\, 1 + \eta_i \, Y_i \,\big)  $$
We prove now that all these  $ \Theta_{\!A} $'s  are injective, so that  $ \Theta $  is indeed an isomorphism.
 \vskip7pt
   Let  $ \, {\big( \hat{\eta}_i \big)}_{i \in I} \, , {\big( \check{\eta}_i \big)}_{i \in I} \in A_\uno^{\,\times d_-} \, $  be such that  $ \, \Theta_{\!A}\big(\! {\big( \hat{\eta}_i \big)}_{i \in I} \big) = \Theta_{\!A}\big(\! {\big( \check{\eta}_i \big)}_{i \in I} \big) \; $,  that is  $ \; {\textstyle \prod\limits_{i \in I}^\rightarrow} \big(\, 1 + \hat{\eta}_i \, Y_i \,\big) = {\textstyle \prod\limits_{i \in I}^\rightarrow} \big(\, 1 + \check{\eta}_i \, Y_i \,\big)  \; $.  Then  {\sl we can replay the proof of claim  {\it (a)}},  now taking  $ \; \hat{g}_+ := 1 =: \check{g}_+ \; $:  the outcome will be again  $ \, \hat{\eta}_i = \check{\eta}_i \, $  for all  $ \, i \in I \, $,  i.e.~$ \, {\big( \hat{\eta}_i \big)}_{i \in I} = {\big( \check{\eta}_i \big)}_{i \in I} \; $.
\end{proof}

\medskip

   The first key consequence of the previous results is the following

\medskip

\begin{corollary}
 The supergroup  $ \bk $--functor  $ G_{{}_{\!\cP}} \! $  considered above is representable, hence it is an affine $ \bk $--supergroup.
\end{corollary}

\begin{proof}
 By  Proposition \ref {dir-prod-fact-G_P - lin}  one has an isomorphism  $ \, G_{{}_{\!\cP}} \cong G_+ \times G_-^{\scriptscriptstyle \,<} \, $  as functors, hence as  $ \bk $--superschemes.  As  $ \, G_+ \, $  is representable by assumption, and  $ \, G_-^{\scriptscriptstyle \,<} \cong \mathbb{A}_\bk^{0|d_-} \, $  is representable too (by  Proposition \ref{dir-prod-fact-G_P - lin}{\it (b)\/}  above), we get that  $ \, G_{{}_{\!\cP}} \cong G_+ \times G_-^{\scriptscriptstyle \,<} \, $  is representable as well.
\end{proof}

\medskip

   With next step we fix some further details, so to see that the assignment  $ \, \cP \mapsto G_{{}_{\!\cP}} \, $  eventually yields a functor of the type we are looking for.

\medskip

\begin{proposition}  \label{G_P functor - lin}
 For every  $ \, \cP \in \lsHCp_\bk \, $,  let  $ G_{{}_{\!\cP}} \! $  be defined as above.  Then:
 \vskip3pt
   {\it (a)} \,  $ G_{{}_{\!\cP}} \! $  is globally strongly split;
 \vskip3pt
   {\it (b)} \,  the defining embedding of\/  $ G_{{}_{\!\cP}} \! $  inside  $ \rGL(V) $  is  {\sl closed},  so that  $ \, G_{{}_{\!\cP}} \! $  identifies with a closed subgroup of\/  $ \rGL(V) \, $;
 \vskip3pt
   {\it (c)} \,  the above construction of  $ \, G_{{}_{\!\cP}} \! $  naturally extends to morphisms, so to provide a functor  $ \; \Psi_\ell : \lsHCp_\bk \!\longrightarrow \lgssfsgrps_\bk \; $.
\end{proposition}

\begin{proof}
 {\it (a)} \, We already noticed that, by the very construction, one has  $ \, {\big( G_{{}_{\!\cP}} \big)}_\zero = G_+ \, $;  this together with  $ \, G_-^{\scriptscriptstyle \,<} \cong \mathbb{A}_\bk^{0|d_-} \, $  yields  $ \, G_{{}_{\!\cP}} \cong G_+ \times G_-^{\scriptscriptstyle \,<} \cong G_+ \times\mathbb{A}_\bk^{0|d_-} \, $  (see above).  Furthermore, by construction and  Lemma \ref{tang-group}{\it (b)\/}  we have that  $ \, G_-^{\scriptscriptstyle \,<} \cong \mathbb{A}_\bk^{0|d_-} \, $  is stable by the adjoint action of  $ \, {\big( G_{{}_{\!\cP}} \big)}_\zero = G_+ \, $.  Eventually, all this means that that  $ G_{{}_{\!\cP}} $  is globally strongly split, q.e.d.

\vskip5pt

   {\it (b)} \,  Due to   Proposition \ref{fact-G_P}  and  Proposition \ref{dir-prod-fact-G_P - lin},  it is enough to prove that both  $ G_+ $  and  $ G_-^{\scriptscriptstyle \,<} $  are closed subsuperschemes in  $ \rGL(V) \, $.  The first property holds by the definition of a lsHCp, so we are left to cope with the second.

\vskip3pt

   In the proof of  Proposition \ref{dir-prod-fact-G_P - lin}{\it (b)\/}  we saw that there exists a  $ \bk $--superscheme  isomorphism  $ \, \Theta : \mathbb{A}_\bk^{0|d_-} \!\!\relbar\joinrel\longrightarrow G_-^{\scriptscriptstyle \,<} \, $  given by
  $$  {\big( \eta_i \big)}_{i \in I} \, \mapsto \, \Theta_{\!A}\big(\! {\big( \eta_i \big)}_{i \in I} \big) \, := \, {\textstyle \prod_{i \in I}} \big( 1 + \eta_i Y_i \big) \, =: \, \mathcal{Y}  \qquad  \hbox{for any  $ \, A \in \salg_\bk \, $.}  $$
Expanding the last product   --- inside  $ \End_{\,\bk}(V)(A) \, $,  say ---   yields
  $$  \mathcal{Y}  \,\; := \;\,  \Theta_{\!A}\big(\! {\big( \eta_i \big)}_{i \in I} \big)  \,\; = \;\,  1 \, + \, {\textstyle \sum_{k=1}^{d_-} \prod_{i_1 < \cdots < i_k}} \, \eta_{i_1} \cdots \eta_{i_k} \, Y_{i_1} \cdots Y_{i_k}  \,\; = \;\,  \mathcal{Y}'_0 \, + \mathcal{Y}'_1  $$
where we set  $ \; \mathcal{Y}'_0 := \! \sum\limits_{k \, \text{\sl even}} \hskip-9pt {\phantom{\big|}}_{k=1}^{d_-} {(-1)}^{k \choose 2} \prod\limits_{i_1 < \cdots < i_k} \hskip-3pt \eta_{i_1} \cdots \eta_{i_k} \, Y_{i_1} \cdots Y_{i_k} \, \in \, {\End_{\,\bk}(V)}_\zero(A) \; $  and similarly  $ \; \mathcal{Y}'_1  := \! \sum\limits_{\,k \ \text{\sl odd}^{\phantom{|}}} \hskip-8pt {\phantom{\big|}}_{\hskip-4pt k=1}^{\hskip-4pt d_-} {(-1)}^{{k-1} \choose 2} \prod\limits_{i_1 < \cdots < i_k} \hskip-3pt \eta_{i_1} \cdots \eta_{i_k} \, Y_{i_1} \cdots Y_{i_k} \, \in \, {\End_{\,\bk}(V)}_\uno(A) \; $.
 \vskip5pt
   Now recall that  $ \rgl(V) $  is  $ \bk $--free  and  $ \, {\rgl(V)}_\uno = \fg_\uno \oplus \mathfrak{q} \, $  with both  $ \fg_\uno $  and  $ \mathfrak{q} $  being  $ \bk $--free.  Then the given expansions of  $ \mathcal{Y}'_0 $  and  $ \mathcal{Y}'_1 $  prove that, with respect to some  $ \bk $--basis  $ B $  of  $ \, \End_{\,\bk}(V) = \rgl(V) \, $  extending that of  $ \fg_\uno $  given by the  $ Y_i $'s,  the coefficients  $ c_b $  ($ \, b \! \in \! B \, $)  of both  $ \mathcal{Y}'_0 $  and  $ \mathcal{Y}'_1 $  are polynomials in the  $ \eta_j $'s.  In particular, the coefficients in  $ \mathcal{Y}'_1 $  of each basis element  $ Y_i $  is of the form  $ \, \eta_i + \cO_i(3) \, $,  where  $ \cO_i(3) $  is some polynomial in the  $ \eta_j $'s  in which only monomials of degree odd and at least 3 can occur.  These polynomials yield a  $ \bk $--superscheme  endomorphism  $ \, \Lambda \, $  of  $ \mathbb{A}_\bk^{0|d_-} $  given on  $ A $--points by  $ \; \Lambda_A : {\big( \eta_i \big)}_{i=1,\dots,d_-} \!\! \mapsto {\big(\, \eta_i + \cO_i(3) \big)}_{i=1,\dots,d_-} \; $
%
%
 which is automa\-tically an isomorphism (exploiting the fact that the variables  $ \eta_i $ are nilpotent).  Setting  $ \, \widetilde{\eta}_i := \Lambda_A^{-1}(\eta_i) \, $,  all this implies that the previously mentioned coefficients  $ c_b $  ($ \, b \! \in \! B \, $)  are also polynomials in the  $ \widetilde{\eta}_j $'s,  say  $ \, c_b = P_b\big( {\big\{ \widetilde{\eta}_j \big\}}_{j=1,\dots,d_-} \big) \, $;  in particular when  $ \, b = Y_i \, $  (for some  $ i \, $)  we have  $ \, c_i := c_{{}_{Y_i}} = \widetilde{\eta}_i \, $.  Therefore, our  $ G_-^{\scriptscriptstyle \,<} \, $  is the set of zeroes (in the superscheme-theoretical sense) of the ideal  $ \, \Big( {\big\{\, c_b - P_b\big( {\big\{ c_i \big\}}_{i=1,\dots,d_-} \big) \,\big\}}_{b \in B\,} \Big) \, $  of the  $ \bk $--algebra  $ \, \cO\big( \End_{\,\bk}(V) \big) \, $:  here we think of the  $ c_b $'s  as being elements of  $ \cO\big( \End_{\,\bk}(V) \big) \, $,  which clearly makes sense in that they are defined as ``coordinate functions''.  Therefore  $ G_-^{\scriptscriptstyle \,<} $  is closed in  $ \End_{\,\bk}(V) \, $,  hence also in  $ \rGL(V) \, $,  so that it is a closed  $ \bk $--subsuperscheme of  $ \rGL(V) \, $,  q.e.d.

\vskip5pt

   {\it (c)} \,  The previous claims ensure that  $ G_{{}_{\!\cP}} $  is a  $ \bk $--supergroup,  actually a  {\sl linear\/}  one; moreover, we also remarked that  $ \, {\big( G_{{}_{\!\cP}} \big)}_\zero = G_+ \, $.  In addition, again by the very construction and by  Proposition \ref{dir-prod-fact-G_P - lin}  we find that  $ \, \Lie\,\big(G_{{}_{\!\cP}}\!\big) =\fg \, $:  in particular, by the assumptions on  $ \fg $  this implies that the supergroup  $ G_{{}_{\!\cP}} $  is  {\sl fine}.  Overall this means that  $ \, G_{{}_{\!\cP}} \! \in \lgssfsgrps_\bk \, $.
                                                         \par
   In order to have a functor  $ \, \Psi_\ell : \lsHCp_\bk \longrightarrow \lgssfsgrps_\bk \, $  we still need to define  $ \Psi_\ell $  on morphisms of  $ \lsHCp_\bk \, $.  Letting  $ \; (\Omega_+,\omega) : \cP' := \big( G'_+ \, , \, \fg' \big) \longrightarrow \big( G''_+ \, , \, \fg'' \big) =: \cP'' \; $  be a morphism in  $ \lsHCp_\bk \, $,  we define  $ \, \Psi_\ell\big( (\Omega_+,\omega) \big) \, $  on  $ A $--points   --- for any  $ \, A \in \salg_\bk $  ---   as follows.  Given  $ \, g' \in \Psi_\ell\big( \cP' \big) \, $,  let  $ \, g' = g'_+ \cdot \prod_{i \in I'} \big( 1 + \eta'_i Y_i \big) \, $  be its unique factorization after the factorization  $ \, G_{\!{}_{\cP'}} = G'_+ \times G'_- \, $  of  $ \, G_{\!{}_{\cP'}} := \Psi_\ell\big( \cP' \big) \, $  as in  Proposition \ref{dir-prod-fact-G_P - lin}{\it (a)\/}:  then set
  $$  {\Psi_\ell\big( (\Omega_+,\omega) \big)}_A\big(g'\big) \, := \, \Omega_+\big(g'_+\big) \cdot {\textstyle \prod_{i \in I'}} \big( 1 + \eta'_i \, \omega(Y_i) \big)  $$
It is then a bookkeeping matter to check that this map is actually a group morphism, and that all properties required for that to yield a functor, as desired, are indeed satisfied.
\end{proof}

\bigskip

   In the end, our main result is that the  $ \Psi_\ell $  above is a quasi-inverse such as we were looking for:

\medskip

\begin{theorem}  \label{Psi_ell-inverse_Phi_ell}
 The functor  $ \, \Psi_\ell : \text{\rm (lsHCp)}_\bk \! \relbar\joinrel\relbar\joinrel\longrightarrow \text{\rm (lgss-fsgroups)}_\bk \, $  is inverse, up to a natural isomorphism, to the functor  $ \, \Phi_\ell : \text{\rm (lgss-fsgroups)}_\bk \! \relbar\joinrel\relbar\joinrel\longrightarrow \text{\rm (lsHCp)}_\bk \; $.  In other words, these two are category equivalences, quasi-inverse to each other.
\end{theorem}

\begin{proof}
 The previous results altogether show that, for any  $ \, \cP \in \lsHCp_\bk \, $,  the sHCp associated with  $ \, \Psi_\ell(\cP) := G_{{}_{\!\cP}} \, $  is nothing but  $ \cP $  itself, up to isomorphism: in other words, we have  $ \, \Phi_\ell\,\big( \Psi_\ell\,(\cP) \big) = \Phi_\ell\,\big( G_{{}_{\!\cP}} \big) \, \cong \, \cP \, $.  Moreover, tracking the whole construction one realizes at once that it is  {\sl natural},  i.e.~all these isomorphisms match together as to give  $ \, \Phi_\ell \circ \Psi_\ell \cong \text{\sl Id}_{{}_{\lsHCp_\bk}} \, $.
  \vskip4pt
   As to the composition  $ \, \Psi_\ell \, \circ \, \Phi_\ell \, $,  \, let  $ \, G \in \text{\rm (lgss-fsgroups)}_\bk \, $  and  $ \, \cP := \Phi_\ell\,(G\,) = \big( G_\zero \, , \fg \big) \, \in \, \text{\rm (lsHCp)}_\bk \, $   --- with  $ \, \fg = \Lie\,(G) \, $  ---   and  $ \, G_{{}_{\!\cP}} := \Psi_\ell\,(\cP) = \Psi_\ell\,\big( \Phi_\ell\,(G\,) \big) \, $:  clearly, everything amounts to proving that  $ \, G_{{}_{\!\cP}} = G \, $   --- as closed subgroups inside  $ \rGL(V) \, $,  with  $ V $  as in  Definition \ref{def-lsgrps-lsHCp}{\it (a)\/}  above.  Note that the inclusion  $ \, G_{{}_{\!\cP}} \subseteq G \, $  holds true by construction, since all generators of  $ G_{{}_{\!\cP}}(A) $  belong to  $ G(A) \, $.
 \vskip2pt
   First, by  Proposition \ref{G_P functor - lin}{\it (b)\/}  we have that  $ G_{{}_{\!\cP}} $  is a  {\sl closed\/}  subgroup of  $ \rGL(V) \, $;  the same holds for  $ G \, $,  by assumption.  As  $ \, G_{{}_{\!\cP}} \subseteq G \, $,  we can argue that  $ G_{{}_{\!\cP}} $  is closed also inside  $ G \, $.
 \vskip2pt
   Then we apply  Theorem \ref{cons-splittings}  to  $ \, H := G_{{}_{\!\cP}} \, $  and  $ \, K := G \, $,  finding global splittings  $ \; {\big(G_{{}_{\!\cP}}\big)}_\zero \times {\big(G_{{}_{\!\cP}}\big)}_\uno \cong G_{{}_{\!\cP}} \; $  and  $ \; G_\zero \times G_\uno \cong G \; $  which are consistent with each other, as in the cited Theorem, in particular  $ \, {\big(G_{{}_{\!\cP}}\big)}_\zero \subseteq G_\zero \, $  and  $ \, {\big(G_{{}_{\!\cP}}\big)}_\uno \subseteq G_\uno \, $.
 \vskip2pt
   Third, the inclusion  $ \, {\big(G_{{}_{\!\cP}}\big)}_\zero \subseteq G_\zero \, $  is an identity by the very construction of  $ G_{{}_{\!\cP}} \, $.
%
%
 So we are only left to prove that the inclusion  $ \, {\big(G_{{}_{\!\cP}}\big)}_\uno \subseteq G_\uno \, $,  provided in the second step above, is an equality too.
                                                                     \par
   Now, the fact that the splittings of  $ G_{{}_{\!\cP}} $  and  $ G $  are compatible is equivalent to the fact that the projection
 \vskip-2pt
   \centerline{ $ \, \pi : \, \cO(G) \, = \, \cO\big(G_\zero\big) \otimes \cO\big(G_\uno\big) \, \relbar\joinrel\relbar\joinrel\twoheadrightarrow \, \cO\big({\big( G_{{}_{\!\cP}} \big)}_{\zero\,}\big) \otimes \cO\big({\big( G_{{}_{\!\cP}} \big)}_{\uno\,}\big) \, = \, \cO\big(G_{{}_{\!\cP}}\big) \, $ }
 \vskip2pt
\noindent
 is of the form  $ \, \pi = \pi_\zero \otimes \pi_\uno \, $,  where  $ \, \pi_\zero : \cO\big(G_{\zero\,}\big) := \cO\big(G_{\text{\it ev}\,}\big) \relbar\joinrel\relbar\joinrel\twoheadrightarrow \cO\big({\big( G_{{}_{\!\cP}} \big)}_{\!\text{\it ev}\,}\big) =: \cO\big({\big( G_{{}_{\!\cP}} \big)}_{\zero\,}\big) \, $  is canonically defined and  $ \, \pi_\uno : \cO\big(G_\uno\big) \relbar\joinrel\relbar\joinrel\twoheadrightarrow \cO\big(\! {\big( G_{{}_{\!\cP}} \big)}_{\uno\,} \big) \, $  is a suitable morphism.  Now, by construction  $ \, \cO\big(G_\uno\big) \, $  is a Grassmann algebra, namely  $ \, \cO\big(G_\uno\big) = \bigwedge {\text{\sl Lie\/}\big(G_\uno\big)}^* \, $,  and  $ \, \cO\big(G_\uno\big) = \bigwedge {\text{\sl Lie\/}\big(\! {\big( G_{{}_{\!\cP}} \big)}_{\uno\,} \big)}^* \, $  by similar reasons.  But still by construction we have  $ \, \text{\sl Lie\/}\big(G_\uno\big) \! = \fg_\uno \! = \text{\sl Lie\/}\big(\! {\big( G_{{}_{\!\cP}} \big)}_\uno \big) \, $  inside  $ \, \text{\sl Lie}\big( \rGL(V) \big) \! = \rgl(V) \, $  so the inclusion map  $ \; \text{\sl Lie\/}\big( {\big( G_{{}_{\!\cP}} \big)}_\uno \big) = \bigwedge \fg_\uno \lhook\joinrel\longrightarrow \bigwedge \fg_\uno = \text{\sl Lie\/}\big(G_\uno\big) \; $  is just the identity; hence its dual, namely the projection map  $ \, \pi_\uno : \cO\big(G_\uno\big) \relbar\joinrel\relbar\joinrel\twoheadrightarrow \cO\big({\big( G_{{}_{\!\cP}} \big)}_\uno \big) \, $,  is the identity from  $ \, \cO\big(G_\uno\big) = \bigwedge \fg_\uno^{\,*} \, $  to  $ \, \cO\big( {\big( G_{{}_{\!\cP}} \big)}_\uno \big) = \bigwedge \fg_\uno^{\,*} \, $.  In turn, the inclusion  $ \, {\big( G_{{}_{\!\cP}} \big)}_\uno \subseteq G_\uno \, $  is necessarily the identity too.
\end{proof}

\medskip

 \subsection{The converse functor: general case}  \label{conv-funct-gen} {\ }

\smallskip

   We shall now face the task of providing a quasi-inverse to the functor  $ \, \Phi : \fsgrps_\bk \! \longrightarrow \sHCp_\bk \, $  in greater generality.  In the end, it will turn out that this will be successful only if we bound ourselves to deal with fine supergroups which are globally strongly split: in other words, a fine supergroup  $ G $  can be ``reconstructed'' starting from its associated sHCp if and only if it is globally strongly split, i.e.~only if  $ \, G \in \gssfsgrps_\bk \, $    --- notation of  Definition \ref{gl-str-split_sgroup-Def}.  Therefore, for sheer notational purposes we introduce the following

\smallskip

\begin{definition}  \label{gss-sgrps_cat}
 We denote by  $ \, \Phi_g : \gssfsgrps_\bk \!\longrightarrow\! \sHCp_\bk \, $  the restriction to the subcatego\-ry  $ \gssfsgrps_\bk $  of the functor  $ \, \Phi : \fsgrps_\bk \!\longrightarrow\! \sHCp_\bk \, $  considered in  Proposition \ref{sgrps-->sHCp}.   \hfill   $ \diamondsuit $
\end{definition}

\smallskip

   By the way, note that by  Remark \ref{Lie-repr_->_(gs-split_->_fine)}  if  $ \Lie\,(G) $  is representable for a supergroup  $ G \, $,  then asking  $ G $  to be fine and gs-split actually amounts to asking that  $ G $  be gs-split only.

\smallskip

   We are ready to go and construct a quasi-inverse functor to  $ \Phi_g \, $.  As we shall presently see, the very construction is modeled on that of  $ \Phi_\ell \, $,  and also many arguments used in the proofs are essentially the same, up to technical modifications.  The key difference with the linear case is the following.  Roughly speaking, in that setup having an embedding of  $ \cP $  inside  $ \rgl(V) $  allowed us to construct  $ G_{{}_{\!\cP}} $  as a subsupergroup of  $ \rGL(V) \, $.  Also, we could investigate the properties of such a group, hence proving all our results, just exploiting this ``native'' embedding of  $ G_{{}_{\!\cP}} $  into  $ \rGL(V) $  and then into  $ \End_{\,\bk}(V) $  too.  In the general case such a linearization is not available: nevertheless, we can achieve a ``partial linearization'', which will still be enough for our purposes.
                                                              \par
   Indeed, first we construct our candidate for  $ G_{{}_{\!\cP}} $  by bare hands, in the form of a  $ \bk $--supergroup  functor.  Then we find a suitable representation of  $ \cP \, $,  and we show that this naturally ``integrate'' to a representation of  $ G_{{}_{\!\cP}} \, $:  this representation, though not faithful, is still ``faithful enough'' to make it possible to apply again the arguments we used in the linear case.  Thus we can replicate,  {\it mutatis mutandis},  the process we followed in that case, and eventually find that our candidate for  $ G_{{}_{\!\cP}} $  actually does the job, namely  $ \, \cP \mapsto G_{{}_{\!\cP}} \, $  yields the converse functor we were looking for.

\smallskip

   As a first step, we start with the definition of  $ G_{{}_{\!\cP}} \, $:

\medskip

\begin{definition}  \label{def G_- / G_P - gen}
 Let  $ \, \cP := \big( G_+ \, , \fg \big) \in \sHCp_\bk \, $  be a sHCp over  $ \bk \, $.  We fix in  $ \, \fg_\uno \, $,  which is  $ \bk $--free,  a  $ \bk $--basis  $ \, {\big\{ Y_i \big\}}_{i \in I} \, $  (for some index set  $ I \, $)  and a total order in  $ I \, $.
 \vskip7pt
   {\it (a)}\,  We introduce a  $ \bk $--supergroup  functor  $ \; G_{{}_{\!\cP}} : \salg_\bk \!\relbar\joinrel\longrightarrow \grps \; $  as follows.  For any given  $ \, A \in \salg_\bk \, $,  consider a formal element  $ \, \big( 1 + \eta_i \, Y_i \big) \, $  for each pair  $ \, (i,\eta_i) \in I \! \times \! A_\uno \, $.
                                                                      \par
   {\sl We define  $ G_{{}_{\!\cP}}(A) $  by generators and relations\/}:  the set of  {\it generators\/}  is
  $$  \Gamma_{\!A}  \,\; := \;\,  \big\{\, g_+ \, , \big( 1 + \eta_i \, Y_i \big) \,\big|\, g_+ \in G_+(A) \, , \, (i,\eta_i) \in I \times A_\uno \,\big\}  \,\; = \;\,  G_+(A) \,{\textstyle \bigcup}\, {\big\{ (1 + \eta_i \, Y_i) \big\}}_{(i,\eta_i\!) \, \in \, I \times A_\uno}  $$
(where  $ \, G_+(A) := G_+(A_0) \, $,  by abuse of notation) and the set of  {\it relations\/}  is
 \vskip-13pt
  $$  \displaylines{
   \qquad \qquad   g'_+ \cdot\, g''_+  \,\; = \;\,  g'_+ \,\cdot_{\!\!\!{}_{G_+}} g''_+   \phantom{{}_{\big|}}  \hfill  \forall \;\; g'_+ \, , g''_+ \in G_+(A)  \qquad  \cr
   \quad   \big( 1 + \eta_i \, Y_i \big) \cdot g_+  \,\; = \;\,  g_+ \cdot \big( 1 + c_{j_1} \eta_i \, Y_{j_1} \big)  \cdot \cdots \cdot \big( 1 + c_{j_k} \eta_i \, Y_{j_k} \big)   \hfill  \cr
   \hfill   \forall \;\; (i,\eta_i) \in I \! \times \! A_\uno \, , \; g_+ \in G_+(A) \, ,  \;\quad  \text{with}  \quad  \Ad\big(g_+^{-1}\big)(Y_i) \, = \, c_{j_1} \, Y_{j_1} + \cdots + c_{j_k} \, Y_{j_k}   \phantom{{}_{\big|}}  \cr
   \qquad   \big( 1 + \eta'_i \, Y_i \big) \cdot \big( 1 + \eta''_i \, Y_i \big)  \; = \;  \Big( 1_{{}_{G_+}} \!\! + \, \eta''_i \, \eta'_i \, Y_i^{\langle 2 \rangle} \Big)_{\!\!{}_{G_+}} \!\!\cdot \big( 1 + \big( \eta'_i + \eta''_i \big) \, Y_i \big) \phantom{{}_{\big|}}   \hfill \hskip15pt   \forall \;\; i \in I  \qquad  \cr
   \hskip2pt   \big( 1 + \eta_j \, Y_j \big) \cdot \big( 1 + \eta_i \, Y_i \big)  \; = \;  \Big( 1_{{}_{G_+}} \!\! + \, \eta_i \, \eta_j \, [Y_i,Y_j] \Big)_{\!\!{}_{G_+}} \!\!\cdot \big( 1 + \eta_i \, Y_i \big) \cdot \big( 1 + \eta_j \, Y_j \big)   \phantom{{}_{.}}   \hfill \hskip6pt   \forall \; j > i \;\, (\in I)  \;\;  \cr
   \qquad \qquad \qquad \qquad   \big( 1 + 0_{\scriptscriptstyle A} \, Y_i \big)  \,\; = \;\,  1   \phantom{{}_{\big|}}  \hfill  \forall \;\; i \in I  \quad \qquad \qquad  }  $$
 \vskip-1pt
\noindent
 where the first line just means that for generators chosen in  $ G_+(A) $  their product   --- denoted with  ``$ \, \cdot \, $''  ---   inside  $ G_{{}_{\!\cP}}(A) $  is the same as in  $ G_+(A) $   --- where it is denoted with  ``$ \,\; \cdot_{\!\!\!{}_{G_+}} $'';  moreover, notation like  $ \, \Big(\, 1_{{}_{G_+}} \! + \, \eta''_i \, \eta'_i \, Y_i^{\langle 2 \rangle} \Big)_{\!\!{}_{G_+}} \, $  and  $ \, \Big(\, 1_{{}_{G_+}} \! + \, \eta_i \, \eta_j \, [Y_i,Y_j] \Big)_{\!\!{}_{G_+}} \, $  denotes two well-defined elements in  $ G_+(A) $   --- see the proof of  Lemma \ref{tang-group}  for a reminder ---   that in the sequel we shall denote more simply as  $ \, \Big(\, 1 + \eta''_i \, \eta'_i \, Y_i^{\langle 2 \rangle} \Big) \, $  and  $ \, \Big(\, 1 + \eta_i \, \eta_j \, [Y_i,Y_j] \Big) \; $.
                                                             \par
   This yields the
 functor  $ G_{{}_{\!\cP}} $  on objects, and one then defines it on morphisms in the obvious way.
 \vskip5pt
   {\it (b)}\,  We define a  $ \bk $--functor  $ \; G_-^{\scriptscriptstyle \,<} : \salg_\bk \!\relbar\joinrel\longrightarrow \sets \; $  as follows.  For  $ \, A \in \salg_\bk \, $  we set
 \vskip-5pt
  $$  G_-^{\scriptscriptstyle \,<}(A)  \; := \;  \bigg\{\, {\textstyle \prod\limits_{i \in I}^\rightarrow} \big( 1 + \eta_i \, Y_i \big) \;\bigg|\; \eta_i \in A_\uno \;\, \forall \; i \in I \,\bigg\}  \qquad  \big(\, \subseteq G_{{}_{\!\cP}}(A) \,\big)  $$
 \vskip-7pt
\noindent
 where  $ \, \prod\limits_{i \in I}^\rightarrow \, $  denotes an  {\sl ordered product\/}  (with respect to the fixed total order in  $ I \, $).  This defines the functor  $ G_-^{\scriptscriptstyle \,<} $  on objects, and its definition on morphism is then the obvious one.
 \hfill   $ \diamondsuit $
\end{definition}

\smallskip

\begin{remarks}  \label{remarks_post-def G_- / G_P - gen}  {\ }
 \vskip3pt
   {\it (a)}\,  By its very definition  $ G_-^{\scriptscriptstyle \,<} $  can be thought of as a subfunctor of  $ G_{{}_{\!\cP}} \, $.
 \vskip3pt
   {\it (b)}\,  By definition both  $ G_-^{\scriptscriptstyle \,<} $  and  $ G_{{}_{\!\cP}} $  depend on the choice of the ordered  $ \bk $--basis  $ \, {\big\{ Y_i \big\}}_{i \in I} \, $  of  $ \fg_\uno \, $;  nevertheless, basing on remark  {\it (c)\/}  here below one can easily show --- by the same arguments used for  Proposition \ref{fact-G_P}  ---   that  {\sl  $ G_{{}_{\!\cP}} $  is actually independent of this choice}.
 \vskip3pt
   {\it (c)}\,  Alternatively, one can modify the very definition of  $ G_{{}_{\!\cP}} \, $,  giving a different presentation of it which is intrinsically independent of any choice of basis of  $ \fg_\uno \, $,  as it  {\sl does not\/}  make use of any  $ \bk $--basis  $ \, {\big\{ Y_i \big\}}_{i \in I} \, $  of  $ \fg_\uno \, $.   Indeed, for each  $ \, A \in \salg_\bk \, $  one takes the group  $ \, G^\bullet_{{}_{\!\cP}}(A) \, $  with the (larger) set of generators
 \vskip-13pt
  $$  \Gamma_{\!A}^{\,\bullet}  \,\; := \;\,  \big\{\, g_+ \, , \big( 1 + \eta \, Y \big) \,\big|\, g_+ \in G_+(A) \, , \, (Y,\eta) \, \in \, \fg_\uno \!\times\! A_\uno \,\big\}  \,\; = \;\,  G_+(A) \,{\textstyle \bigcup}\, {\big\{ (1 + \eta \, Y) \big\}}_{(Y,\eta) \, \in \, \fg_\uno \!\times\! A_\uno}  $$
and (larger) set of relations is (for  $ \, g'_+ \, , g''_+ \in G_+(A) \, $,  $ \, \eta \, , \eta' \, , \eta'' \in A_\uno \, $,  $ \, Y \, , Y' \, , Y'' \in \fg_\uno $)
 \vskip-13pt
  $$  \displaylines{
   g'_+ \cdot\, g''_+  \,\; = \;\,  g'_+ \,\cdot_{\!\!\!{}_{G_+}} g''_+  \quad  ,
 \qquad \qquad  \big( 1 + \eta \, Y \big) \cdot g_+  \,\; = \;\,  g_+ \cdot \big( 1 + \eta \, \text{\sl Ad}\big(g_+^{-1}\big)(Y) \big)  \cr
   \big( 1 + \eta' \, Y \big) \cdot \big( 1 + \eta'' \, Y \big)  \; = \;  \Big( 1_{{}_{G_+}} \!\! + \, \eta'' \, \eta' \, Y^{\langle 2 \rangle} \Big)_{\!\!{}_{G_+}} \!\!\cdot \big(\, 1 + \big( \eta' + \eta'' \big) \, Y \big) \phantom{{}_{\big|}}  \cr
   \big( 1 + \eta'' \, Y'' \big) \cdot \big( 1 + \eta' \, Y' \big)  \; = \;  \Big( 1_{{}_{G_+}} \!\! + \, \eta' \, \eta'' \, \big[Y',Y''\big] \Big)_{\!\!{}_{G_+}} \!\!\cdot \big( 1 + \eta' \, Y' \big) \cdot \big( 1 + \eta'' \, Y'' \big)  \cr
   \big( 1 + \eta \, Y' \big) \cdot \big( 1 + \eta \, Y'' \big)  \; = \;  \big( 1 + \eta \, \big( Y' + Y'' \big) \big)  \cr
   \big( 1 + \eta \; 0_{\fg_\uno} \big)  \,\; = \;\,  1  \quad ,
 \qquad   \big( 1 + 0_{\scriptscriptstyle A} \, Y \big)  \,\; = \;\,  1  }  $$
 \vskip-1pt
\noindent
 Here almost all relations are sheer generalizations of those in  Definition \ref{def G_- / G_P - gen}{\it (a)},  the exceptions being  $ \, \big( 1 + \eta \; 0_{\fg_\uno} \big) = 1 \, $  and  $ \; \big( 1 + \eta \, Y' \big) \cdot \big( 1 + \eta \, Y'' \big) \, = \, \big( 1 + \eta \, \big( Y' + Y'' \big) \big) \; $.  In particular, the latter together with  $ \; \big( 1 + \eta \, Y \big) \cdot g_+ \, = \, g_+ \cdot \big( 1 + \eta \, \text{\sl Ad}\big(g_+^{-1}\big)(Y) \big) \; $  yields
 \vskip2pt
   \centerline{ $ \; \big( 1 + \eta_i \, Y_i \big) \cdot g_+ \, = \, g_+ \cdot \big( 1 + c_{j_1} \eta_i \, Y_{j_1} \big)  \cdot \cdots \cdot \big( 1 + c_{j_k} \eta_i \, Y_{j_k} \big) \; $ }
 \vskip2pt
\noindent
 when  $ \, \Ad\big(g_+^{-1}\big)(Y_i) = c_{j_1} \, Y_{j_1} + \cdots + c_{j_k} \, Y_{j_k} \, $.  Furthermore, again the relations of type  $ \; \big( 1 + \eta \, Y' \big) \cdot \big( 1 + \eta \, Y'' \big) \, = \, \big( 1 + \eta \, \big( Y' + Y'' \big) \big) \; $  imply that for  $ \, Y = \sum_{s=1}^k c_{j_s} Y_{j_s} \, $  we have
 \vskip2pt
   \centerline{ $ \; \big( 1 + \eta \, Y \big) \, = \, \Big( 1 + \eta \, \sum_{s=1}^k c_{j_s} Y_{j_s} \Big) \, = \, \prod_{s=1}^k \big( 1 + c_{j_s} \eta \, Y_{j_s} \big) \; $ }
 \vskip2pt
\noindent
 where the product actually can be done in any order, as the factors in it mutually commute; thus each generator  $ \, \big( 1 + \eta \, Y \big) \, $  can be obtained via the  $ \, \big( 1 + \eta_i \, Y_i \big) \, $'s.  This easily implies that mapping  $ \, g_+ \, $  and  $ \, \big( 1 + \eta_i Y_i \big) \, $  in  $ G_{{}_{\!\cP}}(A) $  respectively to  $ \, g_+ \, $  and  $ \, \big( 1 + \eta_i Y_i \big) \, $  in  $ G^\bullet_{{}_{\!\cP}}(A) $  yields a well defined epimorphism  $ \, \phi_{\!{}_{\scriptstyle A}} : G_{{}_{\!\cP}}(A) \relbar\joinrel\twoheadrightarrow G^\bullet_{{}_{\!\cP}}(A) \, $.  Conversely, considering inside  $ G_{{}_{\!\cP}}(A) $  the elements  $ \, \big( 1 + \eta \, Y \big) := \prod_{s=1}^k \big( 1 + \, c_{j_s} \eta \, Y_{j_s} \big) \, $   --- for each  $ \, Y = \sum_{s=1}^k c_{j_s} Y_{j_s} \in \fg_\uno \, $  ---   one easily sees that all relations considered above to define  $ G^\bullet_{{}_{\!\cP}}(A) $  also hold true inside  $ G_{{}_{\!\cP}}(A) \, $,  as they follow from the defining relations of the latter group.  This implies that there exists also an epimorphism  $ \, \psi_{\!{}_{\scriptstyle A}} : G^\bullet_{{}_{\!\cP}}(A) \relbar\joinrel\twoheadrightarrow G_{{}_{\!\cP}}(A) \, $  which is the inverse of  $ \phi_{\!{}_{\scriptstyle A}} $  above.  The construction of  $ \phi_{\!{}_{\scriptstyle A}} $  and  $ \psi_{\!{}_{\scriptstyle A}} $  is natural in  $ A \, $,  so in the end  $ G_{{}_{\!\cP}} $  and  $ G^\bullet_{{}_{\!\cP}} $  are isomorphic as group functors.
\end{remarks}

\medskip

   Our goal is to show that assigning to each  $ \cP $  its corresponding  $ G_{{}_{\!\cP}} $  one eventually gets a functor  $ \, \Psi_g : \sHCp_\bk \! \relbar\joinrel\longrightarrow \gssfsgrps_\bk \, $  and also that such a functor is an equivalence, quasi-inverse to  $ \, \Phi_g : \gssfsgrps_\bk \! \relbar\joinrel\longrightarrow \sHCp_\bk \, $.  We shall achieve this result in several steps.
 \vskip5pt

\medskip

\begin{free text}  \label{G_P-module V}
 {\bf The representation  $ \, G_{{}_{\!\cP}} \!\relbar\joinrel\longrightarrow \rGL(V) \, $.}  Let  $ \, \fg = \fg_\zero \oplus \fg_\uno \, $  be our given Lie superalgebra, for which  $ \fg_\uno $  is  $ \bk $--free  of finite rank (see  Definition \ref{def-sHCp}),  hence we can fix a  $ \bk $--basis  $ \, {\big\{ Y_i \big\}}_{i \in I} $  of it, where  $ I $  is some finite index set in which we fix some total order.  Recall that the  {\sl universal enveloping algebra\/}  $ U(\fg) $  is given by  $ \; U(\fg) \, := \, T(\fg) \big/ J \; $  where  $ T(\fg) $  is the tensor algebra of  $ \fg $  and  $ J $  is the two-sided ideal in  $ T(\fg) $  generated by the set
  $$  \Big\{\, x \, y - {(-1)}^{|x|\,|y|} \, y \, x - [x,y] \; , \; z^2 - z^{\langle 2 \rangle} \,\;\Big|\; x, y \in \fg_\zero \cup \fg_\uno \, , \, z \in \fg_\uno \,\Big\}  $$
It is known then   --- see for instance  \cite{vsv},  \S 7.2, with the few, obvious changes needed to take into account the relations of type  $ \, z^2 - z^{\langle 2 \rangle} = 0 \, $  (that are superfluous in the setting therein) ---   that one has splitting(s) of  $ \bk $--supermodules  (actually, even of  $ \bk $--supercoalgebras)
  $$  U(\fg)  \,\; = \;\,  U(\fg_\zero) \otimes_\bk {\textstyle \bigwedge} \, \fg_\uno  \,\; \cong \;\,  {\textstyle \bigwedge} \, \fg_\uno \otimes_\bk U(\fg_\zero)   \eqno (4.7)  $$
   \indent   In addition, by the freeness assumption on  $ \fg_\uno $  and our choice of a basis for it we have that  $ {\textstyle \bigwedge} \, \fg_\uno $  is  $ \bk $--free  too, with  $ \bk $--basis  $ \, \big\{\, Y_{i_1} Y_{i_2} \cdots Y_{i_s} \,\big|\, s \leq |I| \, , \, i_1 \! < \! i_2 \! < \! \cdots \! < \! i_s \,\big\} \, $   --- hereafter, we drop the sign  ``$ \wedge $''  to denote the product in  $ \; {\textstyle \bigwedge} \, \fg_\uno \; $.
 \vskip5pt
   Now let  $ \Uuno $  be the (one-dimensional)  {\sl trivial representation\/}  of  $ \fg_\zero \, $.  Then by the standard process of  {\sl induction\/}  from  $ \fg_\zero $  to  $ \fg $   --- the former being thought of as a Lie subsuperalgebra of the latter ---   we can consider the  {\sl induced representation\/}  $ \, V := \text{\sl Ind}_{\fg_\zero}^{\,\fg}(\Uuno\,) \, $.  Looking at  $ \Uuno $  and  $ V $  respectively as a module over  $ U(\fg_\zero) $  and over  $ U(\fg) \, $,  taking (4.7) into account we get
  $$  V  \; := \;  \text{\sl Ind}_{\fg_\zero}^{\,\fg}(\Uuno\,)  \; = \;  U(\fg) \! \mathop{\otimes}\limits_{U(\fg_\zero)} \! \Uuno  \; = \;  {\textstyle \bigwedge}\, \fg_\uno \mathop{\otimes}\limits_\bk \Uuno  \,\; \cong \;  {\textstyle \bigwedge}\, \fg_\uno   \eqno (4.8)  $$
The last one above is a natural  $ \bk $--module  isomorphism, uniquely determined once a specific element  $ \, \underline{b} \in \Uuno \, $  is fixed that forms a  $ \bk $--basis  of  $ \Uuno $  itself: the isomorphism is  $ \, \omega \otimes \underline{b} \mapsto \omega \, $  for all  $ \, \omega \in \bigwedge \fg_\uno \, $.
%
%
                                                                       \par
  This representation-theoretical construction and its outcome clearly give rise to similar functorial counterparts, for the Lie algebra valued  $ \bk $--superfunctors  $ \cL_{\fg_\zero} $  and  $ \cL_\fg \, $,  as well as for the  $ \bk $--superfunctors  associated with  $ U(\fg_\zero) $  and  $ U(\fg) \, $,  in the standard way.
                                                                       \par
   On the other hand, recall that  $ \, \fg_\zero = \Lie\,(G_+) \, $,  and clearly  $ \Uuno $  is also the trivial representation for  $ G_+ \, $,  as a classical, affine  $ \bk $--group  scheme.  Then, by construction and by (4.8), it is clear that the representation of  $ \fg $  on the space  $ V $ also induces a representation of the sHCp  $ \, \cP = (G_+,\fg) \, $  on the same  $ V $,  in other words  $ V $  itself bears also a structure of  $ (G_+,\fg) $--module,  in the sense of  Definition \ref{def-lsgrps-lsHCp}{\it (b)}   --- just drop the faithfulness requirement.  For later use, we denote by  $ \, (\boldsymbol{r}_{\!+},\rho) : (G_+,\fg) \longrightarrow \End_{\,\bk}(V) \, $  the pair of representation maps  $ \, \boldsymbol{r}_{\!+} : G_+ \!\longrightarrow \rGL(V) \, $  and  $ \, \rho : \fg \longrightarrow \rgl(V) \, $ which encode this  $ (G_+,\fg) $--module  structure on  $ V $.  Moreover, we shall also use again  $ \rho $  to denote the representation map  $ \, \rho : U(\fg) \longrightarrow \End_{\,\bk}(V) \, $  describing the  $ U(\fg) $--module  structure on  $ V $.
 \vskip7pt
   Our key step now is to remark that the above  $ (G_+,\fg) $--module  structure on  $ V $  actually ``integrate'' to a  $ G_{{}_{\!\cP}} $--module  structure, in a natural way.
\end{free text}

\smallskip

\begin{proposition}  \label{G_P-action on V}
 Retain notation as above for the  $ (G_+,\fg) $--module  $ V $.  There exists a unique structure of (left)  $ G_{{}_{\!\cP}} $--module  onto  $ V $  which satisfies the following conditions: for every  $ \, A \in \salg_\bk \, $,  the representation map  $ \, \boldsymbol{r}_{{}_{\!\cP,A}} \! : G_{{}_{\!\cP}}(A) \longrightarrow \rGL(V)(A) \, $  is given on generators of\/  $ G_{{}_{\!\cP}}(A) $   --- namely, all  $ \, g_+ \in G_+(A) \, $  and  $ \, (1 + \eta_i \, Y_i) \, $  for  $ \, i \in I \, $,  $ \, \eta_i \in A_\uno $  ---   by
  $$  \boldsymbol{r}_{{}_{\!\cP,A}}(g_+) \, := \, \boldsymbol{r}_{\!+}(g_+) \quad ,  \qquad  \boldsymbol{r}_{{}_{\!\cP,A}}(1 + \eta_i \, Y_i) \, := \, \rho(1 + \eta_i \, Y_i) \, = \, \text{\sl id}_{{}_V\!} + \eta_i \, \rho(Y_i)  $$
or, in other words,  $ \; g_+.v \, := \, \boldsymbol{r}_{\!+}(g_+)(v) \; $  and  $ \; (1 + \eta_i \, Y_i).v \, := \, \rho(1 + \eta_i \, Y_i)(v)  \, = \, v + \eta_i \, \rho(Y_i)(v) \; $  for all  $ \, v \in V(A) \, $.  In particular, this yields a morphism a  $ \bk $--supergroup  functors  $ \; \boldsymbol{r}_{{}_{\!\cP}} \! : G_{{}_{\!\cP}} \!\longrightarrow \rGL(V) \; $.
\end{proposition}

\begin{proof}
 This is, essentially, a straightforward consequence of the whole construction, and of the very definition of  $ G_{{}_{\!\cP}} \, $.  Indeed, by definition of representation for the sHCp  $ \cP $  we see that the operators  $ \boldsymbol{r}_{{}_{\!\cP,A}}(g_+) $  and  $ \boldsymbol{r}_{{}_{\!\cP,A}}(1 + \eta_i \, Y_i) $  on  $ V $   --- associated with the generators of  $ G_{{}_{\!\cP}}(A) $  ---   do satisfy all relations which, by  Definition \ref{def G_- / G_P - gen},  are satisfied by the generators themselves.  Thus they uniquely provide a well-defined a group morphism  $ \, \boldsymbol{r}_{{}_{\!\cP,A}} \! : G_{{}_{\!\cP}}(A) \longrightarrow \rGL(V)(A) \, $  as required.  The construction is clearly functorial in  $ A \, $,  whence the claim.
\end{proof}

\medskip

   The representation  $ \boldsymbol{r}_{{}_{\!\cP}} $  of  $ G_{{}_{\!\cP}} $  on  $ V $  will play the role which in the linear case was played by the ``intrinsic'' representation  $ V $  yielding the embedding of  $ G_{{}_{\!\cP}} $  into  $ \rGL(V) \, $.  In that case the representation was  {\sl faithful},  by assumption; in the general setup it is not the case any more.  Nevertheless, next result ensures that this representation is still ``faithful enough'' to allow us, in a sense, to adapt to the general setup the arguments used for the linear one.
%
  \eject

\begin{lemma}  \label{semi-faithful} {\ }
                                                         \par
 Let  $ \, V $  be as above,  and  $ A \! \in \! \salg_\bk \, $.  For any  $ \; \hat{g}_- := \prod\limits^\rightarrow\!{}_{i \in I} (1 + \hat{\eta}_i \, Y_i) \, \in \, \rGL(V)(A) \; $  and  $ \; \check{g}_- := \prod\limits^\rightarrow\!{}_{i \in I} (1 + \check{\eta}_i \, Y_i) \, \in \, \rGL(V)(A) \; $,  \, the following are equivalent:
 \vskip5pt
   (a) \quad  $ \hat{\eta}_i \, = \, \check{\eta}_i \;\; $  for all  $ \; i \in I \;\, $;
 \vskip3pt
   (b) \qquad  $ \hat{g}_- \, = \, \check{g}_- $  \quad ;
 \vskip4pt
   (c) \quad\,  $ \hat{g}_-.v \, = \, \check{g}_-.v \;\; $  for all  $ \; v \in V \;\, $;
 \vskip4pt
   (d) \quad  $ \hat{g}_-.\,\underline{b} \, = \, \check{g}_-.\,\underline{b} $  \quad ,  \;\quad  where  $ \, \underline{b} \in \Uuno \, $  form a  $ \bk $--basis  of\/  $ \Uuno $   --- see the remark after (4.8).
\end{lemma}

\begin{proof}
 Clearly  $ \; \text{\it (a)\/} \Longrightarrow \!\text{\it (b)\/} \Longrightarrow \!\text{\it (c)\/} \Longrightarrow \!\text{\it (d)} \; $,  \, thus we only need to prove that  $ \; \text{\it (d)\/} \Longrightarrow \!\text{\it (a)} \; $.
                                                            \par
   To avoid confusion, let us fix some additional notation.  When we are describing  $ V $  as  $ \, V = \bigwedge \fg_\uno.\,\underline{b} \, \cong \bigwedge \fg_\uno \, $,  we write the elements of the  $ \bk $--basis  $ {\{Y_i\}}_{i \in I} $  of  $ \fg_\uno $  as  $ \bar{Y}_i $  instead of  $ Y_i \, $:  thus the  $ \bk $--linear  isomorphism  $ \, \bigwedge \fg_\uno.\,\underline{b} \, \cong \bigwedge \fg_\uno \, $  is given by  $ \, (Y_{i_1} Y_{i_2} \cdots Y_{i_s}).\,\underline{b} \mapsto \bar{Y}_{i_1} \bar{Y}_{i_2} \cdots \bar{Y}_{i_s} \, $   --- for all  $ \, i_1 < i_2 < \cdots < i_s \, $.  Now, in terms of these  $ \bk $--bases  the element  $ \; \hat{g}_-.\,\underline{b} \, \in \, V = \bigwedge \fg_\uno.\underline{b} \; $  can be rewritten (by construction) as
  $$  \hat{g}_-.\,\underline{b}  \,\; = \;  {\textstyle \prod\limits^\rightarrow_{i \in I}} (1 + \hat{\eta}_i \, Y_i).\,\underline{b}  \,\; = \;  \Big( 1 \, + \, {\textstyle \sum\limits^\rightarrow_{i \in I}} \; \hat{\eta}_i \, Y_i \, + \, \cO(2) \Big).\,\underline{b}  \,\; = \;  1 \, + \, {\textstyle \sum\limits^\rightarrow_{i \in I}} \; \hat{\eta}_i \, \bar{Y}_i \, + \, \bar{\cO}(2)  $$
here above by  $ \, \Big( 1 \, + \, {\textstyle \sum\limits^\rightarrow_{i \in I}} \; \hat{\eta}_i \, Y_i \, + \, \cO(2) \Big) \, $  we denote the expansion of the product  $ \, {\textstyle \prod\limits^\rightarrow_{i \in I}} (1 + \hat{\eta}_i \, Y_i) \, $  as an element of  $ U(\fg) \, $,  with  $ \, \cO(2) \, $  which represents further summands {\sl of order at least 2 in the  $ \eta_i $'s},  using the  $ Y_i $'s  as basis elements of  $ \fg_\uno \, $.  Then of course  $ \, \Big( 1 \, + \, {\textstyle \sum\limits^\rightarrow_{i \in I}} \; \hat{\eta}_i \, \bar{Y}_i \, + \, \bar{\cO}(2) \Big) \, $  is the analogous object written in terms of the  $ \bar{Y}_i $'s.  Similarly, taking  $ \check{g}_- $  instead of  $ \hat{g}_- $  we find
  $$  \check{g}_-.\,\underline{b}  \,\; = \;  {\textstyle \prod\limits^\rightarrow_{i \in I}} (1 + \check{\eta}_i \, Y_i).\,\underline{b}  \,\; = \;  \Big( 1 \, + \, {\textstyle \sum\limits^\rightarrow_{i \in I}} \; \check{\eta}_i \, Y_i \, + \, \cO(2) \Big).\,\underline{b}  \,\; = \;  1 \, + \, {\textstyle \sum\limits^\rightarrow_{i \in I}} \; \check{\eta}_i \, \bar{Y}_i \, + \, \bar{\cO}(2)  $$
Then the identity  $ \; \hat{g}_-.\,\underline{b} = \check{g}_-.\,\underline{b} \; $  yields
 $ \; 1 \, + \, {\textstyle \sum\limits^\rightarrow_{i \in I}} \; \hat{\eta}_i \, \bar{Y}_i \, + \, \bar{\cO}(2) \, = \, 1 \, + \, {\textstyle \sum\limits^\rightarrow_{i \in I}} \; \check{\eta}_i \, \bar{Y}_i \, + \, \bar{\cO}(2) \; $,
\, an identity in  $ \, A \otimes \big( \bigwedge \fg_\uno \big) \, $,  which in turn implies  $ \, \hat{\eta}_i = \check{\eta}_i \, $  for all  $ \, i \in I \, $,  like in the proof of  Proposition \ref{dir-prod-fact-G_P - lin}.
\end{proof}

\medskip

   Roughly speaking, the equivalence between claims  {\it (b)\/}  and  {\it (c)\/}  in the above lemma is sort of a ``partial faithfulness'' of the  $ G_{{}_{\!\cP}} $--module  $ V $.  This is what we need for our next result.

\medskip

\begin{proposition}  \label{dir-prod-fact-G_P - gen}  {\ }
 \vskip5pt
   {\it (a)} \,  The restriction of group multiplication in  $ G_{{}_{\!\mathcal{P}}} $  provides superscheme isomorphisms
 \vskip-7pt
  $$  G_+ \times G_-^{\scriptscriptstyle \,<} \; \cong \; G_{{}_{\!\mathcal{P}}}  \quad ,  \qquad  G_-^{\scriptscriptstyle \,<} \times G_+ \; \cong \; G_{{}_{\!\mathcal{P}}}  $$
Moreover, the group  $ G_{{}_{\!\mathcal{P}}}(A) $  is independent of the choice of an ordered\/  $ \bk $--basis  $ {\big\{ Y_i \big\}}_{i \in I_{\phantom{|}}} \! $  of\/  $ \fg_{\mathbf{1}} $  used for its definition; the same holds true for the whole functor  $ G_{{}_{\!\cP}} \, $.  Similarly, the sets  $ \, G_+\big(A_{\mathbf{1}}^{(2)}\big) \, G_-^{\scriptscriptstyle \,<}(A) \, $  and  $ \, G_-^{\scriptscriptstyle \,<}(A) \, G_+\big(A_{\mathbf{1}}^{(2)}\big) \, $   --- cf. Section 2.1.1 ---   both coincide with the subgroup of\/  $ G_{{}_{\!\mathcal{P}}}(A) $  generated by  $ G_+\big(A_{\mathbf{1}}^{(2)}\big) $  and  $ G_-^{\scriptscriptstyle \,<}(A) \, $,  and they are independent of the choice of an ordered\/  $ \bk $--basis  of\/  $ \fg_{\mathbf{1}} \, $.
 \vskip5pt
   {\it (b)} \,  There exists a\/  $ \bk $--superscheme  isomorphism  $ \, \mathbb{A}_\bk^{0|d_-} \! \cong G_-^{\scriptscriptstyle \,<} \, $,  with  $ \, d_- := |I| = \text{\it dim}_{\,\bk}\big(\fg_\uno\big) \, $,  given on  $ A $--points  by
 $ \,\; \mathbb{A}_\bk^{0|d_-}\!(A) = A_\uno^{\,d_-} \!\!\longrightarrow
G_-^{\scriptscriptstyle \,<}(A) \, , \; {\big(\eta_i\big)}_{i \in I} \mapsto
\prod\limits_{i \in I}^\rightarrow (1 + \eta_i \, Y_i) \; $.
\end{proposition}

\begin{proof}
 {\it (a)}\,  The proof follows by the same arguments we used for  Proposition \ref{fact-G_P}  and  Proposition \ref{dir-prod-fact-G_P - lin}{\it (a)}.  Indeed, acting exactly like in the proof of  Proposition \ref{fact-G_P}  we see   --- working on  $ A $--points,  for each  $ \, A \in \salg_\bk \, $  ---   that  $ \; G_+ \cdot\, G_-^{\scriptscriptstyle \,<} = \, G_{{}_{\!\cP}} \; $,  \, i.e.~the multiplication in  $ G_{{}_{\!\cP}} $  maps  $ \, G_+ \times G_-^{\scriptscriptstyle \,<} \, $  {\sl onto\/}  $ \, G_{{}_{\!\cP}} $  itself.  Indeed, the point is that the arguments in the proof of  Proposition \ref{fact-G_P}  actually only make use of some commutation formulas among elements of  $ G_+(A) $  and elements of the form  $ (1 + \eta_i \, Y_i) \, $:  but exactly the same formulas do hold again in the  {\sl present\/}  $ G_{{}_{\!\cP}} $  we are dealing it now, by its very construction (see  Definition \ref{def G_- / G_P - gen}),  hence we can succesfully replicate the same procedure.  The same strategy of course also proves that  $ \; G_-^{\scriptscriptstyle \,<} \cdot\, G_+ = \, G_{{}_{\!\cP}} \; $,  \, i.e.~the multiplication map from  $ \, G_-^{\scriptscriptstyle \,<} \times G_+ \, $  to  $ \, G_{{}_{\!\cP}} $  is  {\sl onto\/}  again.
 \vskip5pt
  After this, we can adapt the arguments used for  Proposition \ref{dir-prod-fact-G_P - lin}{\it (a)\/}  to show that the multiplication map from  $ \, G_-^{\scriptscriptstyle \,<}(A) \times G_+(A) \, $  onto  $ \, G_{{}_{\!\cP}}(A) $  is also  {\sl injective},  for each  $ \, A \in \salg_\bk \, $,  so to prove the claim about  $ \; G_-^{\scriptscriptstyle \,<} \times G_+ \cong\, G_{{}_{\!\cP}} \; $;  similarly for  $ \; G_+ \times G_-^{\scriptscriptstyle \,<} \cong\, G_{{}_{\!\cP}} \; $.  In this case the ``adaptation'' consists in applying  Lemma \ref{semi-faithful}.
                                                                \par
   Our goal amounts to showing the following: for any  $ \, A \in \salg_\bk \, $,  if  $ \; \hat{g}_- \, \hat{g}_+ = \check{g}_- \, \check{g}_+ \; $  for  $ \; \hat{g}_- \, , \, \check{g}_- \in G_-^{\scriptscriptstyle \,<}(A) \, $,  $ \; \hat{g}_+ \, , \, \check{g}_+ \in G_+(A) \, $,  then  $ \, \hat{g}_- = \check{g}_- \, $  and  $ \, \hat{g}_+ = \check{g}_+ \; $.  Actually, the first identity implies the second one, thus we cope only with the former.  From  $ \; \hat{g}_- \, \hat{g}_+ = \check{g}_- \, \check{g}_+ \; $  we get  $ \; \big( \hat{g}_- \, \hat{g}_+ \big).v = \big( \check{g}_- \, \check{g}_+ \big).v \; $  for every  $ \, v \in V(A) \, $. But definitions yield  $ \; \big( \hat{g}_- \, \hat{g}_+ \big).v = \hat{g}_-.\big(\hat{g}_+.v\big) = \hat{g}_-.v \; $  and  $ \; \big( \check{g}_- \, \check{g}_+ \big).v = \check{g}_-.\big(\check{g}_+.v\big) = \check{g}_-.v \; $,  \, hence  $ \; \big( \hat{g}_- \, \hat{g}_+ \big).v = \big( \check{g}_- \, \check{g}_+ \big).v \; $  reads also  $ \; \hat{g}_-.v = \check{g}_-.v \; $.  Writing  $ \, \hat{g}_- = \prod\limits^\rightarrow\!{}_{i \in I} (1 + \hat{\eta}_i \, Y_i) \, $  and  $ \, \check{g}_- = \prod\limits^\rightarrow\!{}_{i \in I} (1 + \check{\eta}_i \, Y_i) \, $,  so by Lemma \ref{semi-faithful}  we find  $ \, \hat{g}_- = \check{g}_- \, $.
                                                                 \par
   Moreover,  $ G_{{}_{\!\cP}} $  is independent of the choice of basis of  $ \fg_\uno $  because of  Remarks \ref{remarks_post-def G_- / G_P - gen}{\it (b)--(c)}.
                                                                 \par
   As to the last part of claim  {\it (a)},  it is proved again like in  Proposition \ref{fact-G_P}.
 \vskip7pt
 {\it (b)}\,  By construction there exists a morphism  $ \, \mathbb{A}_\bk^{0|d_-} \! \cong G_-^{\scriptscriptstyle \,<} \, $  of  $ \bk $--superschemes  given on  $ A $--points  by
 $ \,\; \mathbb{A}_\bk^{0|d_-}\!(A) = A_\uno^{\,d_-} \!\!\longrightarrow G_-^{\scriptscriptstyle \,<}(A) \, , \; {\big(\eta_i\big)}_{i \in I} \mapsto \prod\limits_{i \in I}^\rightarrow (1 + \eta_i \, Y_i) \; $.
 By the very definition of  $ G_-^{\scriptscriptstyle \,<} $  this is even onto.  On the other hand, it is an isomorphism because on  $ A $--points  it is  {\sl injective\/}  too: indeed, this follows directly from  Lemma \ref{semi-faithful},  namely by the equivalence of claims  {\it (a)\/}  and  {\it (b)\/}  therein.
\end{proof}

\medskip

   Like in the linear case, the previous result yields the following, direct consequence:

\medskip

\begin{corollary}  \label{repres-G_P}
 For every super Harish-Chandra pair  $ \, \cP \in \sHCp_\bk \, $,  the supergroup functor  $ G_{{}_{\!\cP}} $  given by  Definition \ref{def G_- / G_P - gen}  is representable, hence it is a(n affine)  $ \bk $--supergroup  indeed.  More precisely,  $ G_{{}_{\!\cP}} $  is represented by a  $ \bk $--superalgebra  $ \cO(G_{{}_{\!\cP}}) \, $,  with  $ \bk $--algebra  isomorphisms
  $$  \cO(G_{{}_{\!\cP}})  \,\; \cong \;\,  \cO\big(G_+\big) \otimes_\bk \cO\big(G_-^{\scriptscriptstyle \,<}\big)  \,\; \cong \;\,  \cO\big(G_+\big) \otimes_\bk \bk\big[ {\{ \xi_i \}}_{i \in I} \big]  $$
   \indent   Indeed,  $ \cO(G_{{}_{\!\cP}}) $  is a Hopf\/  $ \bk $--superalgebra,  and the above are isomorphisms of super counital left\/  $ \cO\big(G_+\big) $-–comodule  algebras.
\end{corollary}

\begin{proof}
 By  Proposition \ref{dir-prod-fact-G_P - gen}  the  $ \bk $--functor  $ G_{{}_{\!\cP}} $  is the direct product of the two  $ \bk $--superschemes  $ G_+ $  and  $ \, G_-^{\scriptscriptstyle \,<} \cong \mathbb{A}_\bk^{0|d_-} $, which both are representable as functors.  Then  $ G_{{}_{\!\cP}} $  is representable as well, namely it is represented by
  $$  \cO(G_{{}_{\!\cP}}) \; \cong \; \cO\big(G_+\big) \otimes_\bk \cO\big(G_-^{\scriptscriptstyle \,<}\big) \; \cong \; \cO\big(G_+\big) \otimes_\bk \bk\big[ {\{ \xi_i \}}_{i \in I} \big]  $$
                                                              \par
   Moreover, the unit element of  $ G_{{}_{\!\cP}} $  is the product of the unit in  $ G_+ $  and the unit in  $ G_-^{\scriptscriptstyle \,<} $   --- in both factorizations  $ \, G_{{}_{\!\cP}} = G_+ \cdot G_-^{\scriptscriptstyle \,<} \, $  and  $ \, G_{{}_{\!\cP}} = G_-^{\scriptscriptstyle \,<} \cdot G_+ \, $,  so the above isomorphisms are counit-preserving. Finally, using the factorization  $ \, G_{{}_{\!\cP}} = G_+ \cdot G_-^{\scriptscriptstyle \,<} \, $  the left multiplication restricted to  $ G_+ $  corresponds to left multiplication in the left-hand factor  $ G_+ \, $,  whence the above isomorphisms also preserve the left  $ \cO\big(G_+\big) $--coaction.  The claim follows.
\end{proof}

\medskip

   We still need to fix some details to see that the recipe  $ \, \cP \mapsto G_{{}_{\!\cP}} \, $  in the end does provide a functor of the type we are looking for.  This is the outcome of next step.

\medskip

\begin{proposition}  \label{G_P functor - gen}  {\ }
 For every  $ \, \cP \in \sHCp_\bk \, $,  let  $ G_{{}_{\!\cP}} \! $  be defined as above.  Then:
 \vskip3pt
   {\it (a)} \,  $ G_{{}_{\!\cP}} \! $  is fine and globally strongly split, in short  $ \; G_{{}_{\!\cP}} \! \in \gssfsgrps_\bk \; $;
 \vskip3pt
   {\it (b)} \,  the above construction of  $ \, G_{{}_{\!\cP}} \! $  naturally extends to morphisms in  $ \sHCp_\bk \, $,  so it yields a unique functor  $ \; \Psi_g : \sHCp_\bk \!\longrightarrow \gssfsgrps_\bk \; $  given on objects by  $ \, \Psi_g(\cP) := G_{{}_{\!\cP}} \, $.
\end{proposition}

\begin{proof}
 {\it (a)}\,  Directly from definitions one has that  $ \, {\big( G_{{}_{\!\cP}} \big)}_\zero := {\big( G_{{}_{\!\cP}} \big)}_{\!\text{\it ev}} \, $  coincides with  $ G_+ \, $.  Toge\-ther with  Proposition \ref{dir-prod-fact-G_P - gen}  and  Corollary \ref{repres-G_P},  this implies that  $ G_{{}_{\!\cP}} $  is globally strongly split, a global splitting being the factorization  $ \; G_+ \times G_-^{\scriptscriptstyle \,<} \, \cong \, G_{{}_{\!\cP}} \; $  given in  Proposition \ref{dir-prod-fact-G_P - gen}.
                                                             \par
   In addition, from this factorization one sees   --- by bare hands computation, following the very definition of  $ \Lie\,(G) $  given in  Definition \ref{tangent_Lie_superalgebra}  ---   that
  $$  \Lie\,\big( G_{{}_{\!\cP}} \big)  \; = \;  \Lie\,\big( G_+ \times G_-^{\scriptscriptstyle \,<} \,\big)  \; = \;  \Lie\,\big( G_+ \big) \oplus T_e\big( G_-^{\scriptscriptstyle \,<} \,\big)  \; = \;  \cL_{\fg_\zero} \oplus \cL_{\fg_\uno}  \; = \;  \cL_\fg  $$
that is (identifying  $ \cL_\fg $  with  $ \fg $  as usual) simply  $ \, \Lie\,\big( G_{{}_{\!\cP}} \big) = \fg \, $,  this being an identification as Lie  $ \bk $--superalgebras.  As  $ \fg_\uno $  is  $ \bk $--free  of finite rank, by assumption (see  Definition \ref{def-sHCp},  we conclude that  $ G_{{}_{\!\cP}} $  is fine, as required.
 \vskip5pt
   {\it (b)}\,  This is trivial, directly from definitions.
\end{proof}

\medskip

   We have now available a functor  $ \; \Psi_g : \sHCp_\bk \!\relbar\joinrel\longrightarrow \gssfsgrps_\bk \; $  which is our candidate to be
%
%
 a quasi-inverse to
 $ \; \Phi_g : \gssfsgrps_\bk \!\relbar\joinrel\longrightarrow \sHCp_\bk \; $.
                                                       \par
   We first need to establish some additional results.  The first one is technical:

\medskip

\begin{lemma}  \label{Delta-phi}
 Let  $ \, H = \overline{H} \otimes_\bk \bigwedge \! W^H \, $  be a strongly split Hopf\/  $ \bk $--superalgebra.  Identify  $ \overline{H} $  with  $ \, \overline{H} \otimes_\bk 1 \, $  and  $ \, \bigwedge \! W^H $  with  $ \, 1 \otimes_\bk \bigwedge \! W^H \, $;  also, for  $ \, K \in \big\{ H \, , \overline{H} \, , \, \bigwedge \! W^{\! H} \,\big\} \, $  let  $ \, K^+ := \Ker\,\big( \epsilon_{{}_H}\big|_{\scriptscriptstyle K} \big) \, $.  Then for each  $ \, \phi \in {\big( \bigwedge \! W^H \big)}_\uno^+ = {\big( \bigwedge \! W^H \big)}_\uno \, $  we have (using Sweedler's like notation)
  $$  \Delta(\phi)  \;\; = \;  \phi \otimes 1 \; + \; 1 \otimes \phi \; + \; {\textstyle \sum_{(\phi)^+}} \; \phi_{(1)}^+ \otimes \phi_{(2)}^+  $$
where  $ \; \big( \phi_{(1)}^+ \, , \phi_{(2)}^+ \big) \in \big( H_\uno^{\,[2]} \! \times \! H_\uno \big) \cup \big( H_\uno \! \times \! \overline{H}^{\,+} \big) \cup \big( H_\uno \! \times \! H_\uno^{\,[2]} \big) \; $   --- with notation as in  \S \ref{superalgebras}.
\end{lemma}

\begin{proof}
 Let us start with  $ \, n = 1 \, $.  Since  $ \, \epsilon(\phi) = 0 \, $,  we can always write  $ \Delta(\phi) $  in the form  $ \, \Delta(\phi) = \phi \otimes 1 + 1 \otimes \phi + \sum_{(\phi)} \phi_{(1)}^+ \otimes \phi_{(2)}^+ \, $  with  $ \; \phi_{(1)}^+ \, , \, \phi_{(2)}^+ \in H^+ \; $.  After that, recall that the ``strong splitting''  $ \, H = \overline{H} \otimes_\bk \bigwedge \! W^H \, $  is an isomorphism as augmented algebras  {\sl with a left  $ \overline{H} $--action}.  By the way these  $ \overline{H} $--actions  are defined (see  \S \ref{aug/spl/Hsalg})  we see that this means that  $ \; \Delta(\phi) \equiv 1 \otimes \phi \, \mod \big( J_{\!H} \otimes_\bk \! H \big) \, $;  in turn, this implies that we can write
  $$  \Delta(\phi)  \,\; = \;\,  \phi \otimes 1 \, + \, 1 \otimes \phi \, + \, {\textstyle \sum_{(\phi)^+}} \, \phi_{(1)}^+ \otimes \phi_{(2)}^+   \eqno  \text{with}  \quad  \phi_{(1)}^+ \in J_{\!H} \, , \; \phi_{(2)}^+ \in H^+  \quad  $$
By the way, note that  $ \, J_{\!H} = H_\uno^{\,[2]} \oplus H_\uno \subseteq H^+ \, $  and  $ \, H^+ = \overline{H}^{\,+} \! \oplus J_{\!H} = \overline{H}^{\,+} \! \oplus H_\uno^{\,[2]} \oplus H_\uno \, $.
 \vskip5pt
   Finally, as  $ \Delta $  is parity-preserving one has  $ \, \big| \Delta(\phi) \big| = |\phi| = \uno \, $,  thus  $ \, \big| \phi_{(1)}^+ \big| + \big| \phi_{(2)}^+ \big| = \uno \, $  too.
 \vskip5pt
   First assume  $ \, \big| \phi_{(1)}^+ \big| = \zero \, $;  then we have  $ \, \big| \phi_{(2)}^+ \big| = \uno \, $,  which means  $ \, \phi_{(2)}^+ \in \big( H_\uno \cap H^+ \big) = H_\uno \, $.  In addition,  $ \, \big| \phi_{(1)}^+ \big| = \zero \, $  means  $ \, \phi_{(1)}^+ \in H_\zero \, $,  so  $ \, \phi_{(1)}^+ \in \big( H_\zero \cap J_{\!H} \big) = \big( H_\zero \cap\big( H_\uno^{\,[2]} \oplus H_\uno \big) \big) = H_\uno^{\,[2]} \, $.
 \vskip5pt
   Second, let  $ \, \big| \phi_{(1)}^+ \big| = \uno \, $;  then  $ \, \big| \phi_{(2)}^+ \big| = \zero \, $,  which means  $ \, \phi_{(2)}^+ \in \big( H_\zero \cap H^+ \big) \, $,  hence from the above remark  $ \, \phi_{(2)}^+ \in \big( H_\zero \cap H^+ \big) = \big( H_\zero \cap \big( \overline{H}^{\,+} \! \oplus H_\uno^{\,[2]} \oplus H_\uno \big) \big) = \overline{H}^{\,+} \! \oplus H_\uno^{\,[2]} \, $.  Eventually, we can split  $ \, \phi_{(2)}^+ \in \overline{H}^{\,+} \!\! \oplus H_\uno^{\,[2]} \, $  into the sum of a term in  $ \overline{H}^{\,+} \! $  plus another in  $ H_\uno^{\,[2]} $,  getting a result as claimed.
\end{proof}

\medskip

   Next three results concern a finer analysis of a gs-split fine supergroup.

\medskip

\begin{proposition}  \label{supergr_gen_G0+1-par_odd}
 Given  $ \, G \in \gssfsgrps_\bk \, $,  a\/  $ \bk $--basis  $ {\{Y_i\}}_{i \in I} $  of  $ \fg_\uno $  and  $ \, A \in \salg_\bk \, $,  consider in  $ G(A) $  the elements  $ \, (1 + \eta_i \, Y_i) \, $  for all  $ \, \eta_i \in A_\uno \, $,  $ \, i \in I \, $  (as recalled in the proof of Lemma \ref{tang-group}).  Then  $ G(A) $  is generated by
 $ \; G_\zero(A) \cup \big\{ (1 + \eta_i \, Y_i) \,\big|\; (i,\eta_i) \in I \!\times\! A_\uno \,\big\} \; $.
\end{proposition}

\begin{proof}
 As  $ G $  is globally strongly split the  $ \bk $--superalgebra  $ \cO(G) $  identifies, up to isomorphism, with  $ \, \overline{\cO(G)} \otimes_\bk \bigwedge \! W^{\cO(G)} = \cO\big(G_\zero\big) \otimes_\bk \cO(G_\uno) \, $,  where  $ \, \bigwedge W^{\cO(G)} = \cO(G_\uno) \, $  in turn identifies with  $ \, \bk\big[ {\{ \xi_i \}}_{i \in I} \big] \, $   --- for some finite set  $ I $  such that  $ \, |I| = \text{\sl rk}_{\,\bk}(\fg_\uno) \, $  ---   $ W^H $  with  $ \, \text{\sl Span}_{\,\bk}\big( {\{ \xi_i \}}_{i \in I} \big) \, $  and  $ \fg_\uno $  with  $ \, {\big( W^H \big)}^* := \Hom_{\,\bk}\big(W^H,\bk\big) \, $.
 Moreover, by definition the subgroup  $ G_\zero $  of  $ G $  can be characterized as follows: for any  $ \, A \in \salg_\bk \, $  one has
  $$  G_\zero(A)  \, = \,  \Big\{\, g \in \Hom_{\salg_\bk}\big( \cO(G_\zero) \otimes_\bk \cO(G_\uno) , A \big) \,\Big|\; g{\big|}_{ 1 \otimes_\bk \cO(G_\uno)} = \epsilon{\big|}_{ 1 \otimes_\bk \cO(G_\uno)} \Big\} \;   \eqno (4.9)  $$
   \indent   Recall   --- see the proof of  Lemma \ref{tang-group}  ---   that if  $ \, Z \in \fg \, $  and $ \, e \in A \, $  are homogeneous of the same degree and  $ \, e^2 = 0 \, $,  then  $ \; ( 1 + e \, Z \,) \in G(A) \; $;  in particular this applies for any  $ \, Y := Z \in \fg_\uno \, $  and  $ \, \eta \in A_\uno \, $,  so that  $ \, ( 1 + \eta \, Y \,) \in G(A) \, $.  Hereafter  $ \; \eta \, Y \in A_\uno \fg_\uno \subseteq \fg(A) \; $  is thought of as the unique  $ A $--valued  $ \epsilon_{{}_{\cO(G_\zero)}} $--\,derivation  of  $ \bk\big[ {\{ \xi_i \}}_{i \in I} \big] $  which maps every  $ \xi_i $  to  $ \, \eta\,Y(\xi_i) \, $   --- through the  $ \bk $--module  identification  $ \, \fg_\uno = {\big( W^H \big)}^* := \Hom_{\,\bk}(W^H,\bk) = \Hom_{\,\bk}\big( \text{\sl Span}_{\,\bk}\big( {\{ \xi_i \}}_{i \in I} \big) , \bk \big) \, $ mentioned above ---   and acts  onto  $ \cO(G_\zero) $  as the counit map  $ \epsilon_{{}_{\cO(G_\zero)}} \, $.
                                                                     \par
   Note that, since  $ \, \eta^2 = 0 \, $,  each morphism  $ ( 1 + \eta \, Y \,) $  vanishes on  $ \, {\big( {\cO(G)}^+ \big)}^2 \, $,  hence in particular on  $ \, {\big( {\cO(G)}_\uno \big)}^2 \, $.
%
%
 Moreover, such a  $ ( 1 + \eta \, Y \,) $  also vanishes on  $ \, {\cO(G_\zero)}^+ \, $,  by construction.
                                                                     \par
   Let us now fix a (finite)  $ \bk $--basis  $ {\{ Y_i \}}_{i \in I} $  of  $ \fg_\uno \, $:  namely, we take the unique one for which  $ \, Y_i(\xi_j) = \delta_{i,j} \, $  for all  $ i $  and  $ j \, $;  also, we set  $ \, d_- := |I| \, $  and we fix a total order in  $ I $  by numbering its elements, so that  $ \, I = \{ i_1 , \dots , i_{d_-} \} \, $.
                                                                     \par
   Given  $ \, g \in G(A) \, $,  set  $ \, \eta_i := g(\xi_i) \in A_\uno \, $,  for every  $ \, i \in I \, $,  and  $ \, \gamma_g := \prod\limits_{i=1}^{d_-} ( 1 + \eta_i \, Y_i) \in G(A) \, $.
                                                                               \par
   First consider  $ \, g_1 := g \cdot \gamma_g^{-1} = g \cdot \! \prod\limits_{i=\,d_-}^1 \!\! {( 1 + \eta_i \, Y_i)}^{-1} = g \cdot \! \prod\limits_{i=\,d_-}^1 \!\! ( 1 - \eta_i \, Y_i) \, $   --- the product now being in reversed order.  Second, as the product in  $ \, G(A) = \Hom_{\salg_\bk}\big(\cO(G),A\big) \, $  is given by convolution, for any  $ \, \phi \in {\big( \bigwedge \! W^H \big)}_\uno = \, {\bk\big[ {\{\xi_i\}}_{i \in I} \big]}_\uno \, $  we have
  $$  \displaylines{
   \;\;   g_1(\phi)  \,\; = \;  \bigg(\, g \cdot \! {\textstyle \prod\limits_{i=\,d_-}^1} \! ( 1 - \eta_i Y_i \,) \!\bigg)(\phi)  \,\; = \;\,  m_{\!{}_A} \bigg(\! \bigg(\, g \otimes \bigg(\, {\textstyle \mathop{\otimes}\limits_{i=\,d_-}^1} ( 1 - \eta_i Y_i \,) \!\bigg) \!\bigg) \big( \Delta^{d_-}(\phi) \big) \!\bigg)  \; =   \hfill  \cr
   \hfill   = \;\;\,  {\textstyle \sum\limits_{(\phi)}} \;\; g\big( \phi_{(1)} \big) \cdot {\textstyle \prod\limits_{i=\,d_-}^1} ( 1 - \eta_i Y_i)\big( \phi_{(i+1)} \big)  \quad  }  $$
where  $ \, \sum_{(\phi)} \phi_{(1)} \otimes \phi_{(2)} \otimes \cdots \otimes \phi_{(d_- + 1)} = \Delta^{d_-}(\phi) \, $  as usual.  Now, by repeated applications of  Lemma \ref{Delta-phi}  to  $ \, H = \cO(G) \, $,  we can achieve such an expansion in the form
  $$  \Delta^{d_-}(\phi)  \,\; = \;\,  {\textstyle \sum_{r+s\,=\,d_-}} 1^{\otimes r} \! \otimes \phi \otimes 1^{\otimes s} \, + \, {\textstyle \sum_{(\phi)^+}} \, \phi_{(1)} \otimes \phi_{(2)} \otimes \cdots \otimes \phi_{(d_- + \,1)}  $$
where each monomial  $ \, \phi_{(1)} \otimes \cdots \otimes \phi_{(d_- + 1)} \, $  satisfies either one of three possible properties, namely:
 \vskip7pt
   {\it (I)} \qquad\!\!  $ \phi_{(1)} \otimes \phi_{(2)} \otimes \cdots \otimes \phi_{(d_- + 1)} \, = \, 1^{\otimes r} \! \otimes \phi \otimes 1^{\otimes s} $  \!\!\quad  for any possible  $ (r,s) $  giving  $ \, r + s = d_- \, $   --- and any such possibility actually occurs;
 \vskip5pt
   {\it (II)} \qquad  $ \phi_{(\ell)} \in \Big(\, {\overline{\cO(G)}}^{\,+} \! \cup {\cO(G)}_\uno^{\,[2]} \Big) = \Big( {\cO\big(G_\zero\big)}^+ \cup {\cO(G)}_\uno^{\,[2]} \Big) \; $  \quad  for some  $ \,\; \ell > 1 \;\; $;
 \vskip5pt
   {\it (III)} \qquad  $ \phi_{(1)} \otimes \phi_{(2)} \otimes \cdots \otimes \phi_{(d_- + 1)} \; \in \; \bigotimes_{s=1}^{d_- + \, 1} \! {\cO(G)}_\uno^{\,[e_s]} $  \quad  with  \;\;  $ \sum_{s=1}^{d_- + \, 1} \! e_s \, \geq \, 3 \;\; $.
 \vskip9pt
   When  {\sl case (I)\/}  occurs, the contribution to
  $$  g_1(\phi) \; = \; {\textstyle \sum_{(\phi)}} g\big( \phi_{(1)} \big) \cdot {\textstyle \prod_{i=\,d_-}^1} ( 1 - \, \eta_i Y_i)\big( \phi_{(i+1)} \big)  $$
is  $ \; g(\phi) \; $  for  $ \, r = 0 \, $  and   $ \; ( 1 - \eta_\ell \, Y_\ell)(\phi) \; $  for all  $ \, r > 1 \, $;  \, then summing over all values of  $ \, r \in \{0,1,\dots,d_-\} \, $  yields the total contribution  $ \; g(\phi) + \sum_{r=1}^{d_-} ( 1 - \eta_r \, Y_r)(\phi) \; $.
                                                        \par
   When  {\sl case (II)\/}  occurs, the product  $ \, \prod_{i=\,d_-}^1 \, ( 1 - \eta_i \, Y_i)\big( \phi_{(i+1)} \big) \, $  vanishes: indeed, it contains the  {\sl zero\/}  factor  $ \, ( 1 - \eta_\ell \, Y_\ell)\big( \phi_{(\ell+1)} \big) = 0 \, $,  because we saw above that  $ \, ( 1 - \eta_\ell \, Y_\ell) \, $  vanishes on  $ \, {\cO(G_\zero)}^+ \cup {\cO(G)}_\uno^{\,[2]} \; $.  So the contribution to  $ g_1(\phi) $  is  $ \; g\big( \phi_{(1)} \big) \cdot \prod_{i=\,d_-}^1 ( 1 - \, \eta_i Y_i)\big( \phi_{(i+1)} \big) = 0 \; $.
                                                        \par
   In the end, when  {\sl case (III)\/}  occurs the product  $ \; g\big( \phi_{(1)} \big) \cdot \prod_{i=\,d_-}^1 ( 1 - \eta_i \, Y_i)\big( \phi_{(i+1)} \big) \; $  belongs to  $ \, \fa_\uno^{\big[ \sum_{s=1}^{d_- + \, 1} \! e_s \big]} \! \subseteq \fa_\uno^{\,[3]} \, $,  with  $ \, \fa := \big( {\big\{ \eta_i := g(\xi_i) \big\}}_{i \in I} \big) \, $  the ideal of  $ A $  generated by all the  $ \eta_i $'s.
 \vskip5pt
   Eventually, the outcome is that, for every  $ \, \phi \in {\big( \bigwedge \! W^H \big)}_\uno = \, {\bk\big[ {\{\xi_i\}}_{i \in I} \big]}_\uno \, $,  one has
 \vskip-5pt
  $$  g_1(\phi)  \; = \;  g(\phi) \, + \, {\textstyle \sum_{\ell=1}^{\,d_-}} \, ( 1 - \eta_\ell \, Y_\ell)(\phi) \, + \, \alpha_3   \eqno \text{for some}  \quad  \alpha_3 \, \in \, \fa_\uno^{\,[3]}  \qquad  $$
 \vskip1pt
\noindent
   We apply this result to  $ \, \phi = \xi_j \, $  with  $ \, j = 1, \dots, d_- \; $:  this yields
 \vskip-13pt
  $$  g_1\big(\xi_j\big)  \; = \;  g\big(\xi_j\big) \, + \, {\textstyle \sum_{\ell=1}^{\,d_-}} \, ( 1 - \eta_\ell \, Y_\ell)\big(\xi_j\big) \, + \, \alpha_3  \; = \;  g\big(\xi_j\big) \, + \, {\textstyle \sum_{\ell=1}^{\,d_-}} \, (-\eta_\ell \, \delta_{\ell,j}\big) \, + \, \alpha_3  \; = \; \alpha_3  \,\; \in \;\, \fa_\uno^{\,[3]}  $$
 \vskip-1pt
\noindent
 since  $ \, \eta_j := g\big(\xi_j\big) \, $  by construction.  Therefore, we have proved the following
 \vskip5pt
   {\it  $ \underline{\text{Claim}} \, $:} \,  {\sl  $ g_1 := g \cdot \gamma_g^{-1} \, $  when restricted to  $ \, 1 \otimes_\bk \cO(G_\uno) = 1 \otimes_\bk \bk\big[ {\{\xi_i\}}_{i \in I} \big] \, $  takes its values in  $ \fa_\uno^{\,(3)} $,  the unital  $ \bk $--subalgebra  of  $ A $  generated by  $ \fa_\uno^{\,[3]} \, $}.
 \vskip7pt
   We can repeat the procedure with  $ \, g_1 \, $  replacing  $ \, g \, $.  Then we consider the corresponding  $ \, \gamma_{g_1} \, $,  which we use to define  $ \, g_2 := g_1 \cdot \gamma_{g_1}^{-1} = g \cdot \gamma_g^{-1} \cdot \gamma_{g_1}^{-1} \, $.  The same arguments   --- or, more directly, the claim above applied to  $ g_1 $  instead of  $ g $  ---   prove that  {\sl  $ \, g_2 \, $  when restricted to  $ \, 1 \otimes_\bk \cO(G_\uno) \, $  takes its values in  $ \fa_\uno^{\,(3^2)} $,  the unital  $ \bk $--subalgebra  of  $ A $  generated by  $ \fa_\uno^{\,[3^2\,]} \, $}.
                                                                        \par
   Iterating the process, we construct elements  $ \, g_s := g_{s-1} \cdot \gamma_{g_{s-1}}^{-1} = g \cdot \gamma_g^{-1} \cdot \gamma_{g_1}^{-1} \cdot \gamma_{g_2}^{-1} \cdots \gamma_{g_{s-1}}^{-1} \, $  for increasing  $ s $  by recursion; their remarkable property is that  {\sl each  $ g_s $  when restricted to  $ \, 1 \otimes_\bk \cO(G_\uno) \, $  takes its values in  $ \fa_\uno^{\,(3^s)} $,  the unital  $ \bk $--subalgebra  of  $ A $  generated by  $ \fa_\uno^{\,[3^s]} \, $}.
                                                                        \par
   Now, as  $ \fa $  is an ideal generated by finitely many  {\sl odd\/}  elements, we have  $ \, \fa_\uno^{\,[n]} \, $  for  $ \, n \gg 0 \, $.  Thus there exists an  $ \, \bar{s} \in \N_+ \, $  such that  $ g_{\bar{s}} $  when restricted to  $ \, 1 \otimes_\bk \cO(G_\uno) \, $  takes its values in  $ \bk \, $,  which means that the restriction of  $ g_{\bar{s}} $  to  $ \, 1 \otimes_\bk \cO(G_\uno) \, $  coincides with the counit map of $ \cO(G) $  followed by the unit map of  $ A \, $.  But this means that  $ \, g_{\bar{s}} \in G_\zero(A) \, $,  thanks to (4.9).
                                                                        \par
   Finally, from  $ \, g_{\bar{s}} = g \cdot \gamma_g^{-1} \! \cdot \gamma_{g_1}^{-1} \! \cdot \gamma_{g_2}^{-1} \! \cdots \gamma_{g_{{\bar{s}}-1}}^{-1} \, $  we get  $ \, g = g_{\bar{s}} \cdot \gamma_{g_{{\bar{s}}-1}} \!\! \cdots \gamma_{g_2} \cdot \gamma_{g_1} \cdot \gamma_g \, $,  which shows that  $ g $  belongs to the subgroup of  $ G(A) $  generated by  $ G_\zero(A) $  and all the  $ (1 + \eta_i \, Y_i) $'s.
\end{proof}

\medskip

\begin{corollary}  \label{fact-G_strspl}
 Keep notation as in  Proposition \ref{supergr_gen_G0+1-par_odd}  above.  Fix a total order in  $ I \, $,  and for any  $ \, A \in \salg_\bk \, $  let
 $ \,\; G_\uno(A) := \prod\limits^\rightarrow\hskip-5pt\phantom{\Big|}_{i \in I}
\big( 1 \! + A_\uno \, Y_i \big) = \Big\{\,
\prod\limits^\rightarrow\hskip-5pt\phantom{\Big|}_{i \in I}
\big( 1 \! + \eta_i \, Y_i \big) \;\Big|\; (i,\eta_i) \in I \!\times\! A_\uno \,\Big\} \; $
 where  $ \, \prod\limits^\rightarrow\hskip-5pt\phantom{\Big|}_{i \in I} \, $  denotes an  {\sl ordered product}.  Then  there exist group-theoretic factorizations
 \vskip-5pt
  $$  G(A) \; = \; G_\zero(A) \cdot G_\uno(A)  \quad ,  \qquad  G(A) \; = \; G_\uno(A) \cdot G_\zero(A)  $$
\end{corollary}

\begin{proof}
 We apply again, almost  {\it verbatim},  the proof of  Proposition \ref{fact-G_P}.  Indeed, the arguments therein only used the relations mentioned in  Lemma \ref{tang-group},  which do hold in  $ G \, $.
\end{proof}

\medskip

   The previous result can be improved as follows:

\medskip
 \vskip3pt

\begin{proposition}  \label{dir-prod-fact-G_strspl}
 The factorizations in  Corollary \ref{fact-G_strspl}  above correspond to  $ \bk $--superscheme  isomorphisms: namely, the multiplication in  $ G $  provides  $ \bk $--supersche\-me  isomorphisms
  $$  G_\zero \times G_\uno \; \cong \; G  \quad ,  \qquad  G_\uno \times G_\zero \; \cong \; G  $$
 \vskip-3pt
   Moreover, there exists a  $ \bk $--superscheme  isomorphism  $ \; \mathbb{A}_\bk^{0|d_-} \! \cong G_\uno \; $  with  $ \, d_- := |I| \; $,
given on  $ A $--points  by
 $ \,\; \mathbb{A}_\bk^{0|d_-}\!(A) = A_\uno^{\,d_-} \!\!\longrightarrow G_\uno(A) \, , \;\; {\big(\eta_i\big)}_{i \in I} \mapsto \prod\limits_{i \in I}^\rightarrow (1 + \eta_i \, Y_i) \; $.
\end{proposition}

\begin{proof}
 The statement is a strict analogue of  Proposition \ref{dir-prod-fact-G_P - lin}  and  Proposition \ref{dir-prod-fact-G_P - gen}, and can be proved along the same lines.  However, the main technical device   --- which previously was provided by  Lemma \ref{lemma-triv_prod}  and Lemma \ref{semi-faithful}  respectively ---  must now be re-conceived in yet another way, tailored for the present context.
                                                                     \par
   Given  $ \, A \in \salg_\bk \, $,  assume that one has  $ \; \hat{g}_0 \, \hat{g}_1 = \check{g}_0 \, \check{g}_1 \; $  for some  $ \, \hat{g}_0 \, , \, \check{g}_0 \in G_\zero(A) \, $  and  $ \, \hat{g}_1 \, , \, \check{g}_1 \in G_\uno(A) \, $;  we number the elements of  $ I $  following their order, so that we can write
  $$  \hat{g}_1 = {\textstyle \prod\limits_{i \in I}^\rightarrow} \big( 1 + \hat{\eta}_i \, Y_i \big) \, = {\textstyle \prod\limits_{i=1}^{d_-}} \big( 1 + \hat{\eta}_i \, Y_i \big) \; ,  \,\;\; \check{g}_1 = {\textstyle \prod\limits_{i \in I}^\rightarrow} \big( 1 + \check{\eta}_i \, Y_i \big) \, = {\textstyle \prod\limits_{i=1}^{d_-}} \big( 1 + \check{\eta}_i \, Y_i \big)  \;\quad  \text{for some \ } \hat{\eta}_i , \check{\eta}_i \in A_\uno \; .  $$
                                                                     \par
   We shall prove now that  $ \, \hat{\eta}_i = \check{\eta}_i \, $  for all  $ \, i \in I \, $.  In particular, assuming  $ \, \hat{g}_0 = \check{g}_0 \, $  this is enough to prove the last part of the statement: namely, this proves the injectivity of the superscheme morphism therein, whose surjectivity is automatic.  In addition, this also implies that  $ \, \hat{g}_1 = \check{g}_1 \, $,  whence (as  $ \; \hat{g}_0 \, \hat{g}_1 = \check{g}_0 \, \check{g}_1 \; $  by assumption) it follows  $ \, \hat{g}_0 = \check{g}_0 \, $  too.
                                                                     \par
   Letting  $ \, H := \cO(G) \, $,  we act much like in the proof of  Proposition \ref{supergr_gen_G0+1-par_odd}.  Therefore, for any  $ \, \phi \in {\big( \bigwedge W^H \big)}_\uno = {\cO(G_\uno)}_\uno \, $  we find
  $$  \big( \hat{g}_0 \, \hat{g}_1 \big)(\phi)  \,\; = \;\,  \hat{g}_0\big(\phi_{(1)}\big) \, \hat{g}_1\big(\phi_{(2)}\big)  \,\; = \;\,  \hat{g}_0(\phi) + \hat{g}_1(\phi) \, + \,
 \hat{g}_0\big(\phi^+_{(1)}\big) \, \hat{g}_1\big(\phi^+_{(2)}\big)  \,\; = \;\,  \hat{g}_1(\phi)  $$
because  $ \, \Delta(\phi) = \phi^+_{(1)} \otimes \phi^+_{(2)} \equiv 1 \otimes \phi  \mod\, \big( J_{\!H} \otimes_\bk H \big) \, $,  by the fact that  $ \, H = \cO(G) \, $  is strongly split (see  Theorem \ref{gl-str-split_sgroup-Th}),  and  $ \, g_0(J_{\!H}) = \{0\} \, $.  The same occurs of course for  ``$ \, \check{g} \, $''  replacing  ``$ \, \hat{g} \, $''  everywhere: hence in the end we have
   $$  \big( \hat{g}_0 \, \hat{g}_1 \big)(\phi) \, = \, \hat{g}_1(\phi) \;\; ,  \quad  \big( \check{g}_0 \, \check{g}_1 \big)(\phi) \, = \, \check{g}_1(\phi)  \qquad \qquad  \forall \;\; \phi \in {\cO(G_\uno)}_\uno   \eqno (4.10)  $$
   \indent   On the other hand we have
  $$  \displaylines{
   \quad   \hat{g}_1(\phi)  \,\; = \;  \bigg(\, {\textstyle \prod\limits_{i=1}^{d_-}} \big( 1 + \hat{\eta}_i \, Y_i \big) \!\bigg)(\phi)  \,\; = \;\,  {\textstyle \sum\limits_{(\phi)}} \;\; {\textstyle \prod\limits_{i=1}^{d_-}} \big( 1 + \hat{\eta}_i \, Y_i \big)\big( \phi_{(i)} \big)  \quad  }  $$
where  $ \, \sum_{(\phi)} \phi_{(1)} \otimes \phi_{(2)} \otimes \cdots \otimes \phi_{(d_-)} = \Delta^{d_- - 1}(\phi) \, $  as usual.  Now, like for  Proposition \ref{supergr_gen_G0+1-par_odd},  by repeatedly applying  Lemma \ref{Delta-phi}  we eventually find, for every  $ \, \phi \in {\big( \bigwedge \! W^H \big)}_\uno = \, {\bk\big[ {\{\xi_i\}}_{i \in I} \big]}_\uno \, $,
  $$  \hat{g}_1(\phi)  \; = \;  {\textstyle \sum\limits_{\ell=1}^{d_-}} \, ( 1 - \eta_\ell \, Y_\ell)(\phi) \, + \, \alpha_\phi\big(\, \underline{\hat{\eta}} \,\big)  $$
where  $ \, \alpha_\phi = \alpha_\phi\big(\, \underline{\eta} \,\big) = \alpha_\phi\big(\eta_1 , \dots , \eta_{d_-} \!\big) \, $  is some polynomial (depending on  $ \phi $)  in the variables  $ \, \eta_1 $,  $ \dots $,  $ \eta_{d_-} \in A_\uno \, $  in which only monomials may occur whose degree is odd and at least three.  Applying all this  $ \, \phi = \xi_j \, $  ($ \, j = 1, \dots, d_- \, $)  and  writing  $ \, \alpha_j := \alpha_{\,\xi_j} \, $  for each  $ j \, $,  we get
  $$  \hat{g}_1\big(\xi_j\big)  \; = \;  {\textstyle \sum\limits_{\ell=1}^{d_-}} \, ( 1 - \hat{\eta}_\ell \, Y_\ell)\big(\xi_j\big) \, + \, \alpha_j\big(\, \underline{\hat{\eta}} \,\big)  \; = \;  {\textstyle \sum\limits_{\ell=1}^{d_-}} \, \hat{\eta}_\ell \, \delta_{\ell,j} \, + \, \alpha_j\big(\, \underline{\hat{\eta}} \,\big)  \; = \;  \hat{\eta}_j \, + \, \alpha_j\big(\, \underline{\hat{\eta}} \,\big)  $$
Clearly, the parallel result holds for  $ \check{g}_1 $,  hence in the end we have (with the same  $ \alpha_j $  twice!)
  $$  \hat{g}_1\big(\xi_j\big)  \; = \;  \hat{\eta}_j \, + \, \alpha_j\big(\, \underline{\hat{\eta}} \,\big)  \quad  ,  \qquad  \check{g}_1\big(\xi_j\big)  \; = \;  \check{\eta}_j \, + \, \alpha_j\big(\, \underline{\check{\eta}} \,\big)   \qquad \qquad  \forall \;\; j \in I   \eqno (4.11)  $$
As  $ \; \hat{g}_0 \, \hat{g}_1 \, = \, \check{g}_0 \, \check{g}_1 \; $,  from (4.10) and (4.11) we get  $ \; \hat{\eta}_j \, + \, \alpha_j\big(\, \underline{\hat{\eta}} \,\big) \, = \, \check{\eta}_j \, + \, \alpha_j\big(\, \underline{\check{\eta}} \,\big) \; $  for all  $ \, j \in I \, $.  As a last remark, we notice that
 $ \; \Big(\, \eta_j \mapsto \eta_j + \alpha_j\big(\, \underline{\eta} \,\big) \; , \;\, \forall \; j \! \in \! I \,\Big) \; $
 defines (the value on  $ A $--points  of) a  $ \bk $--superscheme  automorphism of  $ \mathbb{A}_\bk^{0|d_-} \, $:  therefore, from  $ \; \hat{\eta}_j \, + \, \alpha_j\big(\, \underline{\hat{\eta}} \,\big) \, = \, \check{\eta}_j \, + \, \alpha_j\big(\, \underline{\check{\eta}} \,\big) \; $  for all  $ \, j \in I \, $  we get eventually  $ \, \hat{\eta}_j = \check{\eta}_j \; $  for all  $ \, j \in I \, $,  q.e.d.
\end{proof}

\medskip

   We are ready for next result, the main one of the present section, which extends  Theorem \ref{Psi_ell-inverse_Phi_ell}:

\medskip

\begin{theorem}  \label{Psi_g-inverse_Phi_g}
 The functor  $ \; \Psi_g : \sHCp_\bk \! \relbar\joinrel\longrightarrow \gssfsgrps_\bk \; $  is inverse, up to a natural isomorphism, to the functor  $ \; \Phi_g : \gssfsgrps_\bk \! \relbar\joinrel\longrightarrow \sHCp_\bk \; $.  In other words, the two of them are category equivalences, quasi-inverse to each other.
\end{theorem}

\begin{proof}
 In the proof of  Proposition \ref{G_P functor - gen}  we saw that, for  $ \, \cP = (G_+,\fg) \in \sHCp_\bk \, $  and  $ \, G_{{}_{\!\cP}} := \Psi_g(\cP) \, $,  we have  $ \, {\big( G_{{}_{\!\cP}} \big)}_\zero = G_+ \, $  and  $ \, \Lie\,\big(G_{{}_{\!\cP}}\big) = \fg \, $,  thus  $ \, \Phi_g\big(\Psi_g(\cP)\big) = \cP \, $.  So in one direction we are done.
 \vskip3pt
   Conversely, let  $ \, G \in \gssfsgrps_\bk \, $,  and set  $ \, \fg := \Lie\,(G) \, $,  $ \, \cP := \Phi_g(G) = (G_\zero,\fg) \, $.  We look at the supergroup  $ \, \Psi_g\big(\Phi_g(G)\big) = \Psi_g(\cP) := G_{{}_{\!\cP}} \, $,  aiming to prove that it is naturally isomorphic to  $ G \, $.
                                                                          \par
   Given  $ \, A \in \salg_\bk \, $,  by abuse of notation we denote with the same symbols any element  $ \, g_0 \in G_\zero(A) \, $ as belonging to  $ G(A) $   --- via the embedding of  $ G_\zero(A) $  into  $ G(A) $  ---   and as an element of  $ G_{{}_{\!\cP}}(A) $   --- actually, one of the distinguished generators given from scratch.
 \vskip3pt
   With this convention, it is immediate to see that  Lemma \ref{tang-group}  yields the following:  {\sl there exists a unique group morphism  $ \; \Omega_{{}_A} \! : G_{{}_{\!\cP}}(A) \!\relbar\joinrel\relbar\joinrel\longrightarrow \! G(A) \; $  such that  $ \; \Omega_{{}_A}(g_0) = g_0 \; $  for all  $ \, g_0 \in G_\zero(A) \, $  and  $ \; \Omega_{{}_A}\big( (1 + \eta_i \, Y_i) \big) = (1 + \eta_i \, Y_i) \; $  for all  $ \, \eta_i \in A_\uno \, $,  $ \, i \in I \, $}.
 \vskip5pt
   By  Proposition \ref{supergr_gen_G0+1-par_odd}  above we have that  {\sl the morphism  $ \Omega_{{}_A} $  is actually surjective}.  On the other hand, the  {\sl direct product factorizations\/}  for  $ G_{{}_{\!\cP}} $  (see  Proposition \ref{dir-prod-fact-G_P - gen})  and for  $ G $  (see  Proposition \ref{dir-prod-fact-G_strspl})  easily imply that  {\sl the morphism  $ \Omega_{{}_A} $  is also injective},  hence it is a group isomorphism.  Finally, it is clear that the morphisms  $ \Omega_{{}_A} $'s  are natural in  $ A \, $,  thus overall they provide an isomorphism between  $ \, G_{{}_{\!\cP}} = \Psi_g\big(\Phi_g(G)\big) \, $  and  $ G \, $,  which ends the proof.
\end{proof}

\medskip

\begin{free text}  \label{constr-G_P}
 {\bf An alternative realization of  $ \mathbf{G_{{}_{\!\cP}}} \, $.}  Let  $ \, \cP = \big( G_+ \, , \fg \big) \in \sHCp_\bk \, $  be a super Harish-Chandra pair; we present now a different way of realizing the  $ \bk $--supergroup  $  G_{{}_{\!\cP}} $  introduced in  Definition \ref{def G_- / G_P - gen}{\it (a)}.  In the following, if  $ K $  is any group presented by generators and relations, we write  $ \, K = \big\langle \varGamma \,\big\rangle \Big/ \big( \cR \big) \, $  if  $ \varGamma $  is a set of free generators (of  $ K \, $),  $ \cR $  is a set of ``relations'' among generators and  $ \big( \cR \big) $  is the normal subgroup in  $ K $  generated by  $ \cR \, $.  As a matter of notation, given a presentation  $ \, K = \big\langle \varGamma \,\big\rangle \Big/ \big( \cR \big) = \big\langle \varGamma \,\big\rangle \Big/ \big( \cR_1 \cup \cR_2 \big) \, $  with  $ \, \cR = \cR_1 \cup \cR_2 \; $,  the Double Quotient Theorem gives us
 \vskip-17pt
  $$  K  \; = \;  \big\langle \varGamma \,\big\rangle \!\Big/\! \big( \cR \big)  \; = \;  \big\langle \varGamma \,\big\rangle \!\Big/\! \big( \cR_1 \cup \cR_2 \big)  \; = \;  \big\langle \varGamma \,\big\rangle \!\Big/\! \big( \cR_1 \big) \Bigg/\! \big( \cR_1 \cup \cR_2 \big) \Big/\! \big( \cR_1 \big)  \; = \;  \big\langle\, \overline{\varGamma} \,\big\rangle \!\Big/\! \big(\, \overline{\cR_2} \,\big)  \quad   \eqno (4.12)  $$
where  $ \overline{\varGamma} $  and  $ \overline{\cR_2} $  respectively denote the images of  $ \varGamma $  and of  $ \cR_2 $  in the quotient group  $ \, \big\langle \varGamma \,\big\rangle \Big/ \big( \cR_1 \big) \; $.
 \vskip3pt
   For a fixed  $ \, A \in \salg_\bk \, $,  we consider  $ G_+(A) $  and inside it the normal subgroup  $ G_\thickapprox(A) $  given by
%
  $$  G_\thickapprox(A)  \; := \;  \bigg\langle \Big\{\, g_+ \! \in G_+(A) \;\Big|\; g_+ = \big( 1 \! + \eta' \eta'' X \big) \, , \, \eta', \, \eta'' \in A_\uno \, , \, X \in [\fg_\uno,\fg_\uno] \,\cup\, \fg_\uno^{\langle 2 \rangle} \Big\} \bigg\rangle  $$
Then consider also the three sets
%
%
  $$  \Gamma_{\!A}^{\,+}  \; := \;  G_+(A)  \;\; ,  \quad
 \Gamma_{\!A}^{\,\thickapprox} \; := \; G_\thickapprox(A)  \;\; ,  \quad
 \Gamma_{\!A}^{\,-}  \; := \;  \Gamma_{\!A}^{\,\thickapprox} \,{\textstyle \bigcup}\, {\big\{ (1 \! + \eta \, Y) \big\}}_{(Y,\,\eta) \,\in\, \fg_\uno \!\times\! A_\uno}  $$
and the sets of relations (for  $ \, g_+ \, , g'_+ \, , g''_+ \in G_+(A) \, $,  $ \, g_\thickapprox \, , \, , g''_\thickapprox \in G_\thickapprox(A) \, $,  $ \, \eta \, , \eta' , \eta'' \in A_\uno \, $,  $ \, X \in [\fg_\uno,\fg_\uno] \cup \fg_\uno^{\langle 2 \rangle} $,  $ \, Y , Y' , Y'' \in \fg_\uno \, $,  with  $ \;\cdot_{\!\!\!{}_{G_+}} $  and  $ \;\cdot_{\!\!\!{}_{G_\thickapprox}} $  being the product in  $ G_+(A) $  and in  $ G_\thickapprox(A) $  respectively)
 \vskip-13pt
  $$  \displaylines{
   \mathcal{R}_{\!A}^+ \, : \quad  g'_+ \cdot\, g''_+  \,\; = \;\,  g'_+ \,\cdot_{\!\!\!{}_{G_+}} g''_+  \cr
   \mathcal{R}_{\!A}^- \, : \,
 \begin{cases}
   \qquad \qquad \qquad \qquad \quad  g'_\thickapprox \cdot\, g''_\thickapprox  \,\; = \;\,  g'_\thickapprox \,\cdot_{\!\!\!{}_{G_\thickapprox}} g''_\thickapprox  \cr
   \qquad \quad  \big( 1 + \eta \, Y \big) \cdot g_\thickapprox  \; = \;  g_\thickapprox \cdot \big( 1 + \eta \, Y \big) \cdot \big( 1 + \eta \, \text{\sl Ad}\big(g_\thickapprox^{-1}\big)(Y) \big)  \cr
    \quad  \big( 1 + \eta' \, Y \big) \cdot \big( 1 + \eta'' \, Y \big)  \; = \;  \Big( 1 \! + \, \eta'' \, \eta' \, Y^{\langle 2 \rangle} \Big) \cdot \big(\, 1 + \big( \eta' + \eta'' \big) \, Y \,\big) \phantom{{}_{\big|}}  \cr
   \,  \big( 1 + \eta'' Y'' \big) \cdot \big( 1 + \eta' Y' \big)  =  \Big( 1 \! + \, \eta' \, \eta'' \big[Y',Y''\big] \Big) \cdot \big( 1 + \eta' Y' \big) \cdot \big( 1 + \eta'' Y'' \big)  \cr
   \qquad \qquad \quad  \big( 1 + \eta \, Y' \big) \cdot \big( 1 + \eta \, Y'' \big)  \; = \;  \big( 1 + \eta \, \big( Y' + Y'' \big) \big)  \cr
   \qquad \qquad \qquad \;\;  \big( 1 + \eta \; 0_{\fg_\uno} \big)  \,\; = \;\,  1  \quad ,
 \qquad   \big( 1 + 0_{\scriptscriptstyle A} \, Y \big)  \,\; = \;\,  1
 \end{cases}   \cr
   \mathcal{R}_{\!A}^\ltimes \, : \quad  g_\thickapprox \cdot g_+  \, = \;  g_+ \cdot \big(\, g_+^{-1} \,\cdot_{\!\!\!{}_{G_+}} g_\thickapprox \,\cdot_{\!\!\!{}_{G_+}} g_+ \big) \;\; ,  \;\;\;
  \big( 1 + \eta \, Y \big) \cdot g_+  \, = \;  g_+ \cdot \big( 1 + \eta \, \text{\sl Ad}\big(g_+^{-1}\big)(Y) \big)   \phantom{\Big|^{\big|}}  \cr
   \mathcal{R}_{\!A}^\thickapprox \, : \quad  \big( g_\thickapprox \big)_{\!{}_{\Gamma_A^{\,+}}}  = \;\,  \big( g_\thickapprox \big)_{\!{}_{\Gamma_A^{\,\thickapprox}}}   \phantom{\big|^{|}}  \cr
   \mathcal{R}_{\!A}  \; := \;  \mathcal{R}_{\!A}^+ \,{\textstyle \bigcup}\, \mathcal{R}_{\!A}^- \,{\textstyle \bigcup}\, \mathcal{R}_{\!A}^\ltimes \,{\textstyle \bigcup}\, \mathcal{R}_{\!A}^\thickapprox   \phantom{\big|^{\big|}}  }  $$
and  {\sl define\/}  a new group, by generators and relations, as  $ \; G_-(A) := \big\langle\, \Gamma_{\!A}^- \,\big\rangle \Big/ \big(\, \mathcal{R}_{\!A}^- \,\big) \;\, $.
 \vskip5pt
   It follows from  Remarks \ref{remarks_post-def G_- / G_P - gen}{\it (c)}  that
  $$  \hskip3pt   G_{{}_{\!\cP}}(A)  \; = \; \big\langle\, \Gamma_{\!A}^+ \,{\textstyle \bigcup}\; \Gamma_{\!A}^- \,\big\rangle \Big/ \big( \mathcal{R}_{\!A} \big)  \; = \; \big\langle\, \Gamma_{\!A}^+ \,{\textstyle \bigcup}\; \Gamma_{\!A}^- \,\big\rangle \Big/ \big(\, \mathcal{R}_{\!A}^+ \,{\textstyle \bigcup}\; \mathcal{R}_{\!A}^- \,{\textstyle \bigcup}\; \mathcal{R}_{\!A}^\ltimes \,{\textstyle \bigcup}\; \mathcal{R}_{\!A}^\thickapprox \,\big)   \eqno (4.13)  $$
indeed, we are just taking larger sets of generators and of relations, with enough redundancies as to get in the end a different presentation of  {\sl the same\/}  group.
 \vskip5pt
   From this we find a neat description of  $ G_{{}_{\!\cP}}(A) $  by achieving the presentation (4.12) in a series of intermediate steps, namely adding only one bunch of relations at a time.  As a first step, we have
  $$  \big\langle\, \Gamma_{\!A}^+ \,{\textstyle \bigcup}\; \Gamma_{\!A}^- \,\big\rangle \Big/ \big(\, \mathcal{R}_{\!A}^+ \,{\textstyle \bigcup}\; \mathcal{R}_{\!A}^- \,\big)  \,\; = \;\,  \big\langle\, \Gamma_{\!A}^+ \,\big\rangle \Big/ \big(\, \mathcal{R}_{\!A}^+ \,\big) \,*\, \big\langle\, \Gamma_{\!A}^- \,\big\rangle \Big/ \big(\, \mathcal{R}_{\!A}^- \,\big)  \,\; \cong \;\,  G_+(A) \,*\, G_-(A)   \eqno (4.14)  $$
where  $ \; G_+(A) \, \cong \, \big\langle\, \Gamma_{\!A}^+ \,\big\rangle \Big/ \big(\, \mathcal{R}_{\!A}^+ \,\big) \; $  by construction and  $ \, * \, $  denotes the free product (of two groups).
 \vskip5pt
  For the next two steps we can follow two different lines of action.  On the one hand, one has
  $$  \big\langle\, \Gamma_{\!A}^+ \,{\textstyle \bigcup}\; \Gamma_{\!A}^- \,\big\rangle \Big/ \big(\, \mathcal{R}_{\!A}^+ \,{\textstyle \bigcup}\; \mathcal{R}_{\!A}^- \,{\textstyle \bigcup}\; \mathcal{R}_{\!A}^\ltimes \,\big)  \;\; \cong \;\;  \big( G_+(A) \,* G_-(A) \,\big) \Big/ \big(\, \overline{\mathcal{R}_{\!A}^\ltimes} \,\big)  \;\; \cong \;\;  G_+(A) \ltimes G_-(A)  $$
because of (4.12) and (4.14) together, where  $ \, G_+(A) \ltimes G_-(A) \, $  is the semidirect product of  $ \, G_+(A) \, $  with  $ \, G_-(A) \, $  with respect to the obvious (``adjoint'') action of the former on the latter.  Then
  $$  \displaylines{
   \qquad   \big\langle\, \Gamma_{\!A}^+ \,{\textstyle \bigcup}\; \Gamma_{\!A}^- \,\big\rangle \Big/ \big(\, \mathcal{R}_{\!A} \big)  \;\; \cong \;\;  \big\langle\, \Gamma_{\!A}^+ \,{\textstyle \bigcup}\; \Gamma_{\!A}^- \,\big\rangle \Big/ \big(\, \mathcal{R}_{\!A}^+ \,{\textstyle \bigcup}\; \mathcal{R}_{\!A}^- \,{\textstyle \bigcup}\; \mathcal{R}_{\!A}^\ltimes \,{\textstyle \bigcup}\; \mathcal{R}_{\!A}^\thickapprox \,\big)  \;\; \cong   \hfill  \cr
   \hfill   \cong \;\;  \big( G_+(A) \ltimes G_-(A) \big) \Big/ \big(\, \overline{\mathcal{R}_{\!A}^\thickapprox} \;\big)  \;\; \cong \;\;  \big( G_+(A) \ltimes G_-(A) \big) \Big/ N_\thickapprox(A)   \qquad  }  $$
where  $ \, N_\thickapprox(A) \, $  is the normal subgroup of  $ \, G_+(A) \ltimes G_-(A) \, $  generated by  $ \; \Big\{ \big(\, g_\thickapprox \, , g_\thickapprox^{-1} \big) \,\Big|\; g_\thickapprox \in G_\thickapprox(A) \Big\} \; $.
 \eject
\noindent
 This together with (4.13) eventually yields
  $$  G_{{}_{\!\cP}}(A)  \;\; = \;\;  \big( G_+(A) \ltimes G_-(A) \big) \Big/ N_\thickapprox(A)  $$
 \vskip3pt
   On the other hand, again from (4.12) and (4.14) together we get
  $$  \big\langle\, \Gamma_{\!A}^+ \,{\textstyle \bigcup}\; \Gamma_{\!A}^- \,\big\rangle \Big/ \big(\, \mathcal{R}_{\!A}^+ \,{\textstyle \bigcup}\; \mathcal{R}_{\!A}^- \,{\textstyle \bigcup}\; \mathcal{R}_{\!A}^\thickapprox \,\big)  \;\; \cong \;\;  G_+(A) \,*\, G_-(A) \Big/ \big(\, \overline{\mathcal{R}_{\!A}^\thickapprox} \,\big)  \;\; \cong \;\;  G_+(A) \!\mathop{*}_{G_\thickapprox(A)}\! G_-(A)  $$
where  $ \, G_+(A) \!\mathop{*}\limits_{G_\thickapprox(A)}\! G_-(A) \, $  is the amalgamated product of  $ \, G_+(A) \, $  and  $ \, G_-(A) \, $  over  $ \, G_\thickapprox(A) \, $  with respect to the obvious natural monomorphisms  $ \, G_\thickapprox(A) \longrightarrow G_+(A) \, $  and  $ \, G_\thickapprox(A) \longrightarrow G_-(A) \; $.  Then
  $$  \displaylines{
   \qquad   \big\langle\, \Gamma_{\!A}^+ \,{\textstyle \bigcup}\; \Gamma_{\!A}^- \,\big\rangle \Big/ \big(\, \mathcal{R}_{\!A} \big)  \;\; \cong \;\;  \big\langle\, \Gamma_{\!A}^+ \,{\textstyle \bigcup}\; \Gamma_{\!A}^- \,\big\rangle \Big/ \big(\, \mathcal{R}_{\!A}^+ \,{\textstyle \bigcup}\; \mathcal{R}_{\!A}^- \,{\textstyle \bigcup}\; \mathcal{R}_{\!A}^\thickapprox \,{\textstyle \bigcup}\; \mathcal{R}_{\!A}^\ltimes \,\big)  \;\; \cong   \hfill  \cr
   \hfill   \cong \;\;  \big( G_+(A) \!\mathop{*}\limits_{G_\thickapprox(A)}\! G_-(A) \big) \Big/ \big(\, \overline{\mathcal{R}_{\!A}^\ltimes} \;\big)  \;\; \cong \;\;  \big( G_+(A) \!\mathop{*}\limits_{G_\thickapprox(A)}\! G_-(A) \big) \Big/ N_\ltimes(A)   \qquad  }  $$
where  $ \, N_\ltimes(A) \, $  is the normal subgroup of  $ \, G_+(A) \!\mathop{*}\limits_{G_\thickapprox(A)}\! G_-(A) \, $  generated by
 \vskip5pt
   \centerline{ \quad  $ {\Big\{\, g_+ \cdot \big( 1 + \eta \, Y \big) \cdot g_+^{-1} \cdot {\big( 1 + \eta \, \text{\sl Ad}(g_+)(Y) \big)}^{-1} \,\Big\}}_{(Y,\eta) \,\in\, \fg_\uno \!\times\! A_\uno \, , \, g_+ \,\in\, G_+(A)}  \;\; {\textstyle \bigcup} $   \hfill }
 \vskip4pt
   \centerline{ \hfill   $ {\textstyle \bigcup} \;\;  {\Big\{\, g_+ \cdot g_\thickapprox \cdot g_+ \cdot {\big( g_+ \,\cdot_{\!\!\!{}_{G_+}}\! g_\thickapprox \;\cdot_{\!\!\!{}_{G_+}}\! g_+ \big)}^{-1} \,\Big\}}_{g_+ \in G_+(A) \, , \, g_\thickapprox \in G_\thickapprox(A)} $  \quad }
 \vskip7pt
All this along with (4.13) eventually gives
  $$  G_{{}_{\!\cP}}(A)  \;\; = \;\;  \big( G_+(A) \!\mathop{*}\limits_{G_\thickapprox(A)}\! G_-(A) \big) \Big/ N_\ltimes(A)  $$
%
%
for all  $ \, A \in \salg_\bk \, $.  In functorial terms this yields
  $$  G_{{}_{\!\cP}} \; = \; \big( G_+ \ltimes G_- \big) \Big/ N_\thickapprox  \qquad  \text{and}  \qquad G_{{}_{\!\cP}} \; = \; \big( G_+ \!\mathop{*}\limits_{G_\thickapprox}\! G_- \big) \Big/ N_\ltimes  \;\qquad  \text{or}  \qquad\;  G_{{}_{\!\cP}} \, = \; \big( G_+ \!\mathop{\ltimes}\limits_{G_\thickapprox}\! G_- \big)  $$
where the last, (hopefully) more suggestive notation  $ \; G_{{}_{\!\cP}} = \, \big( G_+ \!\mathop{\ltimes}\limits_{G_\thickapprox}\! G_- \big) \;\, $  tells us that  $ \, G_{{}_{\!\cP}} $  is the ``amalgamated semidirect product'' of  $ G_+ $  and  $ G_- $  over their common subgroup  $ G_\thickapprox \; $.
\end{free text}

\medskip

 \subsection{Examples, applications, generalizations}  \label{examples-appls-genrs} {\ }


   We shall now illustrate how the equivalence we established between (globally strongly split fine) affine supergroups and super Harish-Chandra pairs applies to specific examples.  In particular, we show that one recovers the construction of ``Chevalley supergroups'' as presented in  \cite{fg1,fg2,ga1,ga2}.
                                                                                   \par
   We also have applications to representation theory.  First, if  $ G $  and  $ (G_+,\fg) $  respectively are a supergroup and a sHCp which correspond to each other under the previously mentioned equivalence, then we shall find an equivalence between the category of (left or right)  $ G $--modules  and the category of  $ (G_+,\fg) $--modules.  Second, given a supergroup  $ \, G \in \gssfsgrps_\bk \, $  and any  $ G_\zero $--module  $ V $  we provide an explicit construction of the induced  $ G $--module  $ \text{\sl Ind}_{G_\zero}^{\,G}(V) \, $.
                                                                                   \par
   Finally, we discuss a bit the possibilities to extend our results to a more general setup.

\medskip

\begin{free text}  \label{first-ex_Chev-sgrps}
 {\bf The example of ``Chevalley supergroups''.}  Let  $ \fg $  be a simple Lie superalgebra over an algebraically closed field  $ \mathbb{K} $  of characteristic zero.  A complete classification of these objects was found by Kac (and others, see  e.g.~\cite{ka}),  who split them in two main (disjoint) families: those of ``classical'' type   --- still divided into ``basic'' and ``strange'' types ---   and those of ``Cartan'' type.
 \eject
   In a series of papers, Fioresi and Gavarini devised a systematic procedure to find affine  $ \Z $--supergroups  $ G $  having the given  $ \fg $  as tangent Lie superalgebra   --- see \cite{fg1, fg2, ga1}  for the classical type, and  \cite{ga2}  for the Cartan type.  Indeed, the outcome there is an explicit recipe to construct all supergroups of this type which in addition are  {\sl connected}.
   Their construction
 starts with a faithful, finite-dimensional  $ \fg $--module  $ V $,  and eventually realizes one model of the required  $ \Z $--supergroup  $ G $  as a closed  $ \Z $--subsupergroup  of  $ \rGL(V) \, $.  The procedure mimics and extends the classical one developed by Chevalley to construct (connected) algebraic groups associated with any simple Lie algebra over  $  \mathbb{K} \, $:  for this reason, the resulting supergroups are named ``Chevalley supergroups''.
                                                  \par
   On the other hand, if one revisits the work of Fioresi and Gavarini in the spirit of the present paper, one realizes the following:  {\it the construction of Fioresi and Gavarini is nothing but a special   --- and peculiar, for extra features occur, of course ---   instance of  Theorem \ref{Psi_ell-inverse_Phi_ell}}.  Indeed, once  $ \fg \, $,  a faithful   $ \fg $--module  $ V $,  and suitable  $ \Z $--forms  $ \fg_{{}_\Z} $  and  $ V_{{}_\Z} $  of them are fixed, one can consider the even part  $ \fg_\zero $  and realize, following Chevalley, a (classical) affine group-scheme  $ G_\zero $  over  $ \Z $  which is a (connected) closed subgroup-scheme of  $ \rGL(V) $  such that  $ \, \Lie\,\big(G_\zero\big) = \fg_\zero \, $.  Then  $ \, \cP := \big( (G_\zero , \fg) \, , V \,\big) \, $  is a linear super Harish-Chandra pair over  $ \Z \, $,  i.e.~$ \, \cP
 \in \lsHCp_\Z \, $.
                                                   \par
   Now, after  Theorem \ref{Psi_ell-inverse_Phi_ell}  it makes sense to consider the associated linear fine supergroup  $ \, G_{{}_{\!\cP}} := \Psi_\ell\,(\cP) \in \lgssfsgrps_\Z \, $  over  $ \Z \, $.  Then a direct comparison shows that  {\it the very definition of this  $ G_{{}_{\!\cP}} $  actually coincides with the definition of the  $ \Z $--supergroup  $ G_V $  provided by the recipe of Fioresi and Gavarini}.  Indeed, we can say that Fioresi and Gavarini's construction consists in ``composing'' Chevalley's classical construction   --- to produce a group scheme  $ G_+ $  out of a faithful  $ \fg_+ $--module  $ V_+ \, $,  if  $ \fg_+ $  is a Lie algebra (plus technicalities) ---   and (as a second step) the functor  $ \Psi_\ell \, $.  It follows then that all ``Chevalley supergroups'' (both of classical or of Cartan type) as provided by Fioresi and Gavarini are linear fine supergroups; in particular, they are globally strongly split.
\end{free text}

\vskip7pt

\begin{free text}  \label{representations}
 {\bf Representations 1: the equivalence [supergroup modules  $ \simeq $  sHCp-modules].}  An important byproduct of the equivalence between  $ \gssfsgrps_\bk $  and  $ \sHCp_\bk $  comes as an application to representation theory.  Indeed, let  $ \, G \in \gssfsgrps_\bk \, $  and  $ \, \cP \in \sHCp_\bk \, $  respectively be a supergroup and a sHCp which correspond to each other through the above mentioned equivalence   --- namely,  $ \, G = \Psi_g(\cP) \, $  and  $ \, \cP = \Phi_g(G) \, $.  Then we let  $ G $--Mod and  $ \cP $--Mod  respectively be the category of  $ G $--supermodules  and of  $ \cP $--modules;  in short, here we mean that a  $ G $--module  is the datum of a finite free  $ \bk $--supermodule  $ M' $  with a morphism  $ \, \varOmega : G \longrightarrow \rGL(M') \, $  of supergroups, i.e.~in  $ \sgrps_\bk \, $,  whereas a  $ \cP $--module  is the datum of a finite free  $ \bk $--supermodule  $ M'' $  with a morphism  $ \, (\Omega_+,\omega) : \cP \longrightarrow \big( {\rGL(M'')}_\zero \, , \rgl(M'') \big) \, $  of sHCp's, i.e.~in  $ \sHCp_\bk \, $.
 \vskip4pt
   Now assume  $ M' $  is a  $ G $--module.  Applying  $ \, \Phi_g : \gssfsgrps_\bk \!\longrightarrow \sHCp_\bk \, $  to the morphism  $ \, \varOmega : G \longrightarrow \rGL(M') \, $  we find a morphism  $ \, \Phi_g(\varOmega) : \Phi_g(G) \longrightarrow \Phi_g\big(\rGL(M')\big) \, $  between the corresponding objects in  $ \sHCp_\bk \, $.  But  $ \, \Phi_g(G) = \cP \, $  by assumption and  $ \, \Phi_g\big(\rGL(M')\big) = \big( {\rGL(M')}_\zero \, , \rgl(M') \big) \, $,  so what we have is a morphism  $ \, \Phi_g(\varOmega) \! : \cP \! \longrightarrow \!\big( {\rGL(M')}_\zero \, , \rgl(M') \big) \, $  making  $ M' $ into a  $ \cP $--module.
                                                                               \par
   Conversely, let  $ M'' $  be a  $ \cP $--module.  Applying the functor  $ \, \Psi_g \! : \! \sHCp_\bk \!\longrightarrow \! \gssfsgrps_\bk \, $  to the corresponding morphism  $ \, (\Omega_+,\omega) : \cP \!\rightarrow \big( {\rGL(M'')}_\zero \, , \rgl(M'') \big) \, $   we get a morphism  $ \, \Psi_g\big((\Omega_+,\omega)\big) : \Psi_g(\cP) \longrightarrow \Psi_g\big(\big( {\rGL(M'')}_\zero \, , \rgl(M'') \big)\big) \, $  between the corresponding supergroups.  As  $ \, \Psi_g(\cP) = G \, $  by our assumptions while  $ \, \Psi_g\big(\big( {\rGL(M'')}_\zero \, , \rgl(M'') \big)\big) = \rGL(M'') \, $,  we find a morphism  $ \, \Psi_g\big((\Omega_+,\omega)\big) \! : G \! \longrightarrow \! \rGL(M'') \, $  in  $ \gssfsgrps_\bk $  which makes  $ M'' $ into a  $ G $--module.
                                                                               \par
   The reader can easily check that the previous discussion has the following outcome:
\end{free text}

\vskip0pt

\begin{theorem}  \label{main-thm_sHCp-equiv}
 Let  $ \, G \in \gssfsgrps_\bk \, $  and  $ \, \cP \in \sHCp_\bk \, $  correspond to each other as above.
 \vskip2pt
 (a)\,  For any fixed finite free  $ \bk $--supermodule  $ M $,  the above constructions provide two bijections, inverse to each other, between  $ G $--module  structures and  $ \cP $--module  structures on  $ M \, $.
 \vskip2pt
 (b)\,  The whole construction above is natural in  $ M $,  in that the above bijections over two  $ \bk $--supermodules  $ \widehat{M} $  and  $ \widetilde{M} $ are compatible with  $ \bk $--supermodule  morphisms from  $ \widehat{M} $  to  $ \widetilde{M} \, $.  Thus, all the bijections mentioned in (a) --- for all different  $ M $'s  ---   do provide equivalences, quasi-inverse to each other, between the category of all  $ G $--modules  and the category of all  $ \cP $--modules.
\end{theorem}

\vskip7pt

\begin{remarks}
 Professor Masuoka kindly shared with the author the following observations:
 \vskip2pt
   {\it (a)}\,  Here above we considered modules  $ M $  (over supergroups or over sHCp's) that are finite free as  $ \bk $--supermodules   --- which is consistent with our description of  $ \rGL(M) $  as given in  Examples \ref{exs-supvecs}{\it (b)}.  On the other hand, one can weaken this assumption, requiring only that  $ M $  be finite projective:  Theorem \ref{main-thm_sHCp-equiv}  above will then still hold true.
 \vskip2pt
   {\it (b)}\,  With the above mentioned projectivity assumption, our  Theorem \ref{main-thm_sHCp-equiv}  improves Proposition 5.4 and Theorem 5.8 in  \cite{mas-shi}.  Indeed, Proposition 5.4 holds true over any commutative ring, just assuming (with notation of  \cite{mas-shi})  that the map  $ \, O(G) \longrightarrow {\text{hy}(G)}^* \, $  be injective (geometrically, it means that  $ G $  is connected).
\end{remarks}

\smallskip

\begin{free text}  \label{appl_ind-funct}
 {\bf Representations 2   --- induction from  $ G_\zero $  to  $ G \, $.}  Let a supergroup  $ \, G \in \gssfsgrps_\bk \, $  be given, with associated classical subsupergroup  $ \, G_{\text{\it ev}} = G_\zero \, $.  Let  $ V $  be any  $ G_\zero $--module:  we shall now provide an explicit construction of the  {\it induced  $ G $--module}  $ \text{\sl Ind}_{G_\zero}^{\,G}\!(V) \, $.

\vskip3pt

   Being a  $ G_\zero $--module,  $ V $  is also, automatically, a  $ \fg_\zero $--module.  Then one does have the induced  $ \fg $--module  $ \text{\sl Ind}_{\fg_\zero}^{\,\fg}\!(V) \, $,  which can be realized as
 \vskip3pt
 \centerline{ $ \text{\sl Ind}_{\fg_\zero}^{\,\fg}\!(V)  \,\; = \;\,  \text{\sl Ind}_{U(\fg_\zero)}^{\,U(\fg)}(V)  \,\; = \;\,  U(\fg) \otimes_{U(\fg_\zero)} \! V $ }
 \vskip5pt
\noindent
 By construction, it is clear that this bears also a unique structure of  $ G_\zero $--module  which is compatible with the  $ \fg $--action  and coincides with the original  $ G_\zero $--action  on  $ \, \bk \otimes_{U(\fg_\zero)} \! V \cong V \, $  given from scratch.  Indeed, we can describe explicitly this  $ G_\zero $--action,  as follows.  First, by construction we have
 \vskip4pt
 \centerline{ $ \text{\sl Ind}_{\fg_\zero}^{\,\fg}\!(V)  \,\; = \;\, U(\fg) \otimes_{U(\fg_\zero)} \! V  \,\; = \;\,  {\textstyle \bigwedge} \, \fg_\uno \! \otimes_\bk \! V $ }
 \vskip6pt
\noindent
--- because  $ \, U(\fg) \cong {\textstyle \bigwedge} \, \fg_\uno \! \otimes_\bk \! U(\fg_\zero) \, $  as a  $ \bk $-module,  by the PBW theorem for Lie superalgebras, see (4.7) ---   with the  $ \fg_\zero $--action  given by  $ \; x.(y \otimes v) = \text{\sl ad}(x)(y) \otimes v + y \otimes (x.v) \; $  for  $ \, x \in \fg_\zero \, $,  $ \, y \in {\textstyle \bigwedge} \, \fg_\uno \, $,  $ \, v \in V \, $,  where by  $ \, \text{\sl ad} \, $  we denote the unique  $ \fg_\zero $--action  on  $ \, {\textstyle \bigwedge} \, \fg_\uno \, $  by algebra derivations induced by the adjoint  $ \fg_\zero $--action  on  $ \fg_\uno \, $.  Second, this action clearly integrates to a (unique)  $ G_\zero $--action  given by  $ \; g_0.(y \otimes v) \, := \, \text{\sl Ad}(g_0)(y) \otimes (g_0.v) \; $  for  $ \, g_0 \in G_\zero \, $,  $ \, y \in {\textstyle \bigwedge} \, \fg_\uno \, $,  $ \, v \in V \, $,  where we write  $ \, \text{\sl Ad} \, $  for the unique  $ G_\zero $--action  on  $ \, {\textstyle \bigwedge} \, \fg_\uno \, $  by algebra automorphisms induced by the adjoint  $ G_\zero $--action  on  $ \fg_\uno \, $.
                                                                      \par
   The key point is that the above  $ G_\zero $--action  and the built-in  $ \fg $--action on  $ \text{\sl Ind}_{\fg_\zero}^{\,\fg}\!(V) $  are actually compatible, in the sense that they make  $ \text{\sl Ind}_{\fg_\zero}^{\,\fg}\!(V) $  into a  $ (G_\zero\,,\fg) $--module,  i.e.~a module for the super Harish-Chandra pair  $ \, \cP := (G_\zero\,,\fg) \, $.  Since  $ \, \Psi_g\big((G_\zero\,,\fg)\big) = G \, $,  by  \S \ref{representations}  we have that  $ \text{\sl Ind}_{\fg_\zero}^{\,\fg}\!(V) $  bears a unique structure of  $ G $--module  which correspond to the previous  $ \cP $--action   --- i.e., it yields (by restriction and ``differentiation'') the previously found  $ G_\zero $--action  and  $ \fg $--action.
                                                                      \par
   Therefore,  {\sl we define as  $ \text{\sl Ind}_{G_\zero}^{\,G}\!(V) $  the space  $ \text{\sl Ind}_{\fg_\zero}^{\,\fg}\!(V) $  endowed with this  $ G $--action\/}:  one easily check that this construction is functorial in  $ V  $  and has the universal property which makes it into the adjoint of ``restriction'' (from  $ G $--modules  to  $ G_\zero $--modules),  so it has all rights to be called ``induction'' functor (from  $ G_\zero $--modules  to  $ G $--modules).
                                                                      \par
   In addition, if the original  $ G_\zero $--module  $ V $  is faithful then the induced  $ G $--module  $ \text{\sl Ind}_{G_\zero}^{\,G}\!(V) $  is faithful too: in particular, this means that if  $ G_\zero $  is linearizable, then  $ G $  is linearizable too; more precisely, from a linearization of  $ G_\zero $  one can construct (via induction) a linearization of  $ G $  as well.
\end{free text}

\smallskip

\begin{free text}  \label{furth-genrs}
 {\bf Further generalizations.}  Our construction of a quasi-inverse  $ \Psi_g $  to the functor  $ \Phi_g $  is flexible enough to apply to other contexts.  Hereafter we briefly discuss some possibilities.
 \vskip4pt
   {\sl  $ \underline{\text{The non-affine case\/}} $}.  We may deal with a more general notion of sHCp, modifying  Definition \ref{def-sHCp}  in one aspect: instead of taking as  $ G_+ $  any affine group-scheme over  $ \bk \, $,  we drop the ``affine'' assumption, and allow  $ G_+ $  to be any group-scheme over  $ \bk \, $.  Correspondingly, we consider also  $ \bk $--supergroup-schemes  which are not necessarily ``affine'', i.e.~they are not necessarily representable (as supergroup  $ \bk $--functors)  but still they are obtained by globally pasting together suitable affine  $ \bk $--superschemes  (see  \cite{ccf},  Ch.~11, \S 11.1, for a detailed definition).  In this more general setup, there still exists a natural functor  $ \Phi $  from the category of all (non necessarily affine)  $ \bk $--supergroup-schemes  to the category of ``super Harish-Chandra pairs'' in the present, new sense.
                                         \par
   The whole discussion in the present section can then be repeated: in particular, we can construct a functor  $ \Psi $  from sHCP's to supergroup-schemes, which takes values in the (full sub)category of those supergroup-schemes which are ``globally strongly split'', in the sense of  Definition \ref{gl-str-split_sgroup-Def}. The final outcome then will be that the restriction of  $ \Phi $  to the latter (sub)category and the functor  $ \Psi $  are quasi-inverse to each other:
%
%
 therefore, the category of globally strongly split supergroup-schemes (over  $ \bk \, $,  say) is equivalent to the category of sHCp's (over  $ \bk \, $).  In a nutshell,  Theorem \ref{Psi_g-inverse_Phi_g}  extends to this more general (non-affine) framework.
                                                        \par
   {\it Warning:}  there is just one specific step in the whole procedure, namely  Proposition \ref{dir-prod-fact-G_strspl},  where (in the proof) I concretely made use of the fact that a given supergroup  $ G $ under exam was  {\sl affine},  hence the classical subgroup  $ \, G_\zero = G_+ \, $  in its associated sHCp is affine too.  At this point one must definitely adopt some different argument to get the analogous result in the non-affine case.
 \vskip5pt
   {\sl  $ \underline{\text{Dropping finiteness assumptions\/}} $}.  Still keeping the assumption that  $ \, \Lie\,(G) = \cL_\fg \, $  is representable and  $ \, \fg_\uno $  is  $ \bk $--free,  one can drop the finiteness assumption on  $ \, \text{\sl rk}_{\,\bk}(\fg_\uno) \, $.  In this case, our construction of  $ \, G_{{}_{\!\cP}} $  still makes sense, yielding a supergroup which is automatically fine but is ``gs-split'' only in a modified sense: indeed, we have now  $ \, G_{{}_{\!\cP}} \cong G_+ \times {}^{\text{\it ind}}\mathbb{A}_\bk^{0|d_-} \, $  where  $ \, {}^{\text{\it ind}}\mathbb{A}_\bk^{0|d_-} \, $  is some  {\sl ind-affine, totally odd superspace},  and  $ d_- $  is now a possibly infinite cardinal number.  As to function algebras, we have  $ \; \cO\big(G_{{}_{\!\cP}}\big) \, \cong \, \overline{\cO(G_+)} \otimes_\bk \cO\Big( {}^{\text{\it ind}}\mathbb{A}_\bk^{0|d_-} \Big) \; $  where  $ \, \cO\Big( {}^{\text{\it ind}}\mathbb{A}_\bk^{0|d_-} \Big) \, $  is no longer (a priori) a Grassmann algebra.  Our main result   --  Theorem  \ref{Psi_g-inverse_Phi_g}  ---   about the equivalence between sHCp's and fine supergroups which are ``split'' (in a suitable sense) must then be modified accordingly.
                                                        \par
   On the other hand,  Theorem \ref{Masuoka}  is proved by Masuoka (see  \cite{ms1},  Theorem 4.5)  making no special finiteness assumption on commutative Hopf superalgebras: in our language, this means that when  $ \bk $  is a field (with  $ \, \text{\sl char}(\bk) \not= 2 \, $)  {\sl every affine  $ \bk $--supergroup  $ G $  is gs-split in the sense of  Definition \ref{gl-str-split_sgroup-Def}},  with no modifications whatsoever!  This ought to mean that one should be able to ``read'' our construction of  $ G_{{}_{\!\cP}} $  so as to achieve the same object  $ \, {}^{\text{\it ind}}\mathbb{A}_\bk^{0|d_-} \, $,  but now presented in such a way that one recognizes it as being a true affine (totally odd) superspace, with  $ \, \cO\Big( {}^{\text{\it ind}}\mathbb{A}_\bk^{0|d_-} \Big) \, $  now being recognized as a Grassmann algebra.  This clarification clearly needs a finer analysis, which goes beyond the goals of the present paper.
                                                        \par
   Finally, in this ``non-finite'' setup one can deal with non-affine supergroups: the remarks in the above paragraph (for the non-affine case) apply again, so one ends up with the same conclusions.
 \vskip5pt
   {\sl  $ \underline{\text{The real smooth and complex analytic cases\/}} $}.  In the differential setup one studies real Lie supergroups; similarly, in the analytic framework one deals with complex Lie supergroups.  In both cases, as it is customary to do, we assume that the super-dimension is finite.
                                                           \par
   In both cases, one can adopt a functorial language, which is strictly close to the one used in the algebro-geometric setup (as we did in the present paper).  With such a choice of language  --- and of technical tools to work with ---   one can then also reproduce the construction presented in this paper, in particular that of the functor  $ \Psi_g \, $.  The outcome then will be that  $ \Phi_g $  and  $ \Psi_g $  will be equivalences, quasi-inverse to each other, between the category of (real or complex) Lie supergroups and the category of (real or complex) sHCp's.  Note that here we do not need any ``globally strongly split'' assumption, since such a property always holds true for real or complex Lie supergroups (see  \cite{ccf},  Proposition 7.4.9,  and  \cite{vis}  respectively).
\end{free text}

\vskip15pt

\end{document}